\documentclass[oneside,english,reqno]{amsart}
\usepackage{amsmath,amsthm,amssymb,bbm}
\usepackage{mathrsfs}

\usepackage{tikz}
\usepackage{graphicx}
\usepackage{calculator}

\newtheorem{theorem}{Theorem}[section]
\newtheorem{corollary}[theorem]{Corollary}

\newtheorem{proposition}[theorem]{Proposition}
\newtheorem{remark}[theorem]{Remark}
\newtheorem{lemma}[theorem]{Lemma}
\newtheorem{definition}[theorem]{Definition}
\newtheorem{example}[theorem]{Example}
\newtheorem{notation}[theorem]{Notation}

\newenvironment{myproof}[2] {\paragraph{{\it Proof of} {#1} {#2}:}}{\hfill$\square$}

\newcommand{\munu}{\mu}

\newcommand{\lcolor}[1]{\textcolor{magenta}{#1}}
 
\newcommand{\qmq}[1]{\quad \mbox{#1} \quad}
\newcommand{\equaldist}{\stackrel{\lower0.5pt\hbox{$\scriptstyle\mathrm d$}}=} 
\newcommand{\bflat}{\flat}

\newcommand{\zcirc}{Free-}

\numberwithin{equation}{section}

\newcommand{\ignore}[1]{}
\usepackage{todonotes}
\usepackage{enumitem}
\usepackage{wasysym}

\usepackage[linktocpage,colorlinks,linkcolor=blue,anchorcolor=blue,citecolor=blue,urlcolor=black,pagebackref]{hyperref}

\renewcommand*{\backref}[1]{\ifx#1\relax \else Page #1 \fi}
\renewcommand*{\backrefalt}[4]{
  \ifcase #1 \footnotesize{(Not cited.)}
  \or        \footnotesize{(Cited on page~#2.)}
  \else      \footnotesize{(Cited on pages~#2.)}
  \fi
}

\begin{document}

\title{Bias and Division in the Free World}
\author{Larry Goldstein\textsuperscript{1}}
\address{\textsuperscript{1}Department of Mathematics, University of Southern California, Los Angeles}
\email{larry@usc.edu}
\author{Todd Kemp\textsuperscript{2}}\thanks{\textsuperscript{2} supported in part of NSF Grants DMS-2055340 and DMS-2400246}
\address{\textsuperscript{2}Department of Mathematics, University of California, San Diego}
\email{tkemp@ucsd.edu}

\begin{abstract}
Sampling bias is a foundational concept in statistics; associated bias transforms, such as size bias, have come to play important roles in probability theory of late.  The first author and G.\ Reinert introduced {\em zero bias}, a transform whose unique fixed point is the normal distribution; it has become a standard tool in Stein's method and Gaussian approximation.  Very recently, connections between zero bias and the class of infinitely divisible distributions have been found.

In this paper, we develop a free probabilistic analog of the zero bias transform, proving its existence and regularity.  The {\em free zero bias} has the semicircle law (free probability's central limit distribution) as its unique fixed point.  We offer a construction of the free zero bias that mirrors a classical one incorporating {\em square bias} with a mollifier, and in the process develop a surprisingly new class of distributional operations through their Cauchy transforms.

We then explore connections between the free zero bias, and size bias, with the class of {\em freely} infinitely divisible distributions.  We develop a new self-contained treatment of the subject, together with a new characterization of free infinite divisibility using bias transforms.  We also develop a parallel treatment of {\em positively} freely infinitely divisible distributions, which can also be characterized by a new kind of L\'evy--Khintchine formula that has no known classical analogue, and we use this to both give several new descriptions of such distributions and furnish new examples using these bias methods.

\end{abstract}

\maketitle

\tableofcontents

\section{Introduction} 
Distributional `bias' transformations and their fixed points have been important tools in multiple contexts throughout probability, statistics, and beyond. Well known is the one that samples quantities proportional to their size, the {\bf size bias} \eqref{def:Xs.0}. As common sampling methods may suffer inadvertently from its effects, it was likely first regarded only as a statistical nuisance. Indeed, schemes for determining the sizes of new oil reserves where fields are explored uniformly at random, or surveys which hope to determine family sizes by asking participants how many siblings they have, will result in biased estimates of the desired quantities. Later work revealed that size bias can play more positive roles: it is responsible for waiting-time paradoxes, and has applications to tightness, analysis of the lognormal distribution, Skorohod embedding, number theory, concentration inequalities, Stein's method, and infinite divisibility; see \cite{arratia2019size}. Presently, we focus on the final topic, which we explore using size bias and related transforms like square bias, zero bias, and new bias transforms (introduced below) possessing notable unique distributional fixed points such as the normal or the semi-circle law.

The size bias transformation $X \mapsto X^s$ with domain $\mathcal{D}_m^+$ (the set of Borel probability measures on $[0,\infty)$ with mean $E[X]=m>0$) is characterized by its distribution function 
\begin{equation} \label{def:Xs.0}
P(X^s \le x) = \frac{E[X{\bf 1}(X\le x)]}{E[X]}.
\end{equation}
Equivalently, $X^s$ can be characterized by 
expectations over a function class by
\begin{align}\label{eq:def.mean.m.szb}
E[Xf(X)]=mE[f(X^s)] \quad \mbox{for all $f \in C_b^\infty(\mathbb{R})$,}
\end{align}
where $C_b^\infty(\mathbb{R})$ is the class of all bounded infinitely differentiable functions.
Size biasing has proved especially useful in connection with the Poisson distribution, see \cite{chen1975poisson}, \cite{barbour1992poisson} and \cite{arratia1990poisson}. 
In particular, the mapping $X \mapsto X^s-1$ on $\mathcal{D}_m^+$ has the Poisson distribution with mean $m$ as its unique fixed point.
Similarly, on the domain $\mathcal{D}_{m,\sigma^2}$ of Borel probability measures on $\mathbb{R}$ with mean $m\in\mathbb{R}$ and variance $\sigma^2>0$, we define the {\bf square bias} transformation $X \mapsto X^\Box$ via 
\begin{equation}
\label{def:Xbox}
P(X^\Box \le x) = \frac{E[X^2{\bf 1}(X\le x)]}{E[X^2]},
\end{equation}
which is likewise equivalent to 
\begin{align}\label{eq:def.secmom.sqb}
E[X^2f(X)]=E[X^2]E[f(X^\Box)] \quad \mbox{for all $f \in C_b^\infty(\mathbb{R})$.}
\end{align}

The {\bf zero bias} transform was introduced in \cite{goldstein1997stein} as a component of Stein's method \cite{chen2010normal,stein1972bound} for normal approximation.   Given a random variable $X$ with law $\mathcal{L}(X)=\mu \in \mathcal{D}_{0,\sigma^2}$, there exists a unique distribution $\mu^\ast$ defined by the property that, for a random variable $X^\ast$ with law $\mathcal{L}(X^*)=\mu^\ast$,
\begin{align}
\label{char.zb}
E[Xf(X)]=\sigma^2 E[f'(X^\ast)] \quad \mbox{for all $f \in C_b^\infty(\mathbb{R})$.}
\end{align}
The associated mapping $\mu\mapsto\mu^\ast$ with domain $\mathcal{D}_{0,\sigma^2}$ has the Gaussian $\mathcal{N}(0,\sigma^2)$ as its unique fixed point (i.e.\ reproducing
the usual Gaussian integration by parts formula that motivates the Stein equation). Despite a characterizing equation less intuitive than \eqref{eq:def.mean.m.szb} for size biasing, there is a simple probabilistic construction of $X^\ast$: first take its square bias $X^\Box$, and then multiply it by an independent $\mathrm{Unif}[0,1]$ random variable $U$; that is,
\begin{equation}\label{eq:blah}
X^\ast \equaldist UX^\Box. 
\end{equation}
The proof can be found in \cite[Theorem 2.1]{goldstein2005distributional}.   

This and related transformations have found uses in normal and other distributional approximations \cite{chen2010normal} and e.g. \cite{chen1975poisson}, concentration inequalities \cite{chatterjee2006stein}, \cite{chatterjee2010applications}, \cite{ghosh2011concentration}, \cite{johnson2021concentration}, and  
data science \cite{erdogdu2016scaled}, \cite{liu2016kernelized}, \cite{liu2016stein}, 
\cite{goldstein2019non},
\cite{fathi2020relaxing}.

In this paper we provide free probability analogs of these results and the existence of the associated transformations, starting with the zero bias, which we call the {\bf free zero bias} $\mu\mapsto\mu^\circ$ (or correspondingly for random variables $X\mapsto X^\circ$), also defined on the set $\mathcal{D}_{0,\sigma^2}$. We can extend the transformation to variables $X$ with non-zero mean $m$, as is in Theorem \ref{thm:G.R.corresp}, by
\begin{align}\label{eq:def.mean.m.fzb}
X^\circ=(X-m)^\circ+m,
\end{align}
that is, by
applying the transformation to the centered variable $X-m$, and then adding $m$ back in. 

To define the putative free zero bias $X^\circ$ of a random variable $X$, we rely on a core plank of ``free harmonic analysis'': replacing the derivative $f' = f^{(1)}$ by the difference quotient $f^{[1]}$:
\[ f^{[1]}(x,y)=\frac{f(x)-f(y)}{x-y}. \]
The difference quotient is a function of two variables, and  we will evaluate it on 
{\em two i.i.d.\ copies} 
$X^\circ \equaldist Y^\circ$, giving the following characterizing equation, analogous to \eqref{char.zb}:
\begin{align}\label{char.freezb}
E[Xf(X)]=\sigma^2 E[f^{[1]}(X^\circ,Y^\circ)] \quad \mbox{for all $f \in C_b^\infty(\mathbb{R})$}.
\end{align}
It is easy to verify that $X^\circ$ in \eqref{eq:def.mean.m.fzb} satisfies \eqref{char.freezb} when $X$ on the left hand side is replaced by $X-m$, with similar remarks applying to $X^*$ and \eqref{char.zb}.

The difference quotient $f^{[1]}$ appears in this form often in literature on perturbation of eigenvalues, cf.\ \cite{Kato95}, and noncommutative function theory, cf.\ \cite[Proposition 4.3.1]{Vaki-NCk}, \cite[Example 6.5]{JKN}, owing to its relationship to ordinary (Fr\'echet) derivatives of functional calculus.  In free probability, a more commonly used object is the {\em cyclic derivative} $\mathscr{D}f$, which is a noncommutative version of the difference quotient.  For a polynomial $f$ in one variable, $\mathscr{D}f$ is a noncommutative polynomial in two variables; $f\mapsto \mathscr{D}f$ is linear, and for the monomial $p_n(X) = X^n$
\[ (\mathscr{D}p_n)(X,Y) = \sum_{k=1}^n X^{k-1}Y^{n-k}. \]
If $X$ and $Y$ commute, this coincides with the difference quotient $p_n^{[1]}(X,Y)$.

The characterizing equation \eqref{char.freezb} could be formulated using the cyclic derivative instead: we again ask for a distribution commonly shared by $X^\circ$ and $Y^\circ$, now assumed to be {\em freely} independent (see Definition \ref{def:freeness}), satisfying
\begin{equation} \label{char.freezb.free} E[X f(X)] = \sigma^2 E[\mathscr{D}f(X^\circ,Y^\circ)] \quad \mbox{for all polynomials $f$}. \end{equation}
Because the cyclic derivative introduces no terms that alternate more than once between $X$ and $Y$, \ref{char.freezb} and \ref{char.freezb.free} yield exactly the same moment recursion relations when applied to polynomial functions $f$, and thus the free zero bias $X^\circ$ defined by \eqref{char.freezb} also satisfies \eqref{char.freezb.free}; the reverse is also true, at least for distributions determined by their moments.  We work with the formulation in \eqref{char.freezb} as the classical independence there affords a larger robust toolset to study the free zero bias.

\ignore{

\lcolor{The substitution of the derivative with the finite difference} forms a core plank of ``free harmonic analysis''. It ultimately arises from perturbation theory (of eigenvalues), or more simply from the relationship between derivatives and functional calculus.  If $p\colon\mathbb{C}\to\mathbb{C}$ is a single-variable polynomial, it can act as a function $p_{M_N}$ on $N\times N$ matrices, or more generally as a function $p_{\mathcal{A}}$ on any Banach algebra $\mathcal{A}$ by functional calculus: if $p(z) = \sum_{i=0}^d \alpha_i z^i$ then $p_{\mathcal{A}}(a) = \sum_{i=0}^d \alpha_i a^i$ for any $a\in\mathcal{A}$.  Then the function $p_{\mathcal{A}}$ is smooth, and its ordinary Fr\'echet derivative $[Dp_{\mathcal{A}}(a)]$ (a linear map acting on elements $h\in\mathcal{A}$) is actually given by the two-variable polynomial $p^{[1]}(y,x)$ acting on $h$ (on both sides), evaluated at $y=x=a$; see \cite[Proposition 4.3.1]{Vaki-NCk} and \cite[Example 6.5]{JKN}.  To be more precise: $f^{[1]}$ is a function of two variables; the quantity on the right-hand-side of Equation \eqref{char.freezb} should be interpreted as $E[f^{[1]}(Y^\circ,X^\circ)]$ for two (classically) independent random variables both having law $\mathcal{L}(X^\circ)$.  In other words,
$$
E[f^{[1]}(X^\circ,Y^\circ)]= \lcolor{\sigma^2}E\left[\frac{f(Y^\circ)-f(X^\circ)}{Y^\circ-X^\circ}\right].
$$
This non-linear definition appears quite complicated, but as summarized in Theorem \ref{main theorem 1} below, the equation does uniquely determine a nice probability measure.  }

Our first main theorem, Theorem \ref{main theorem 1}, asserts that the free zero bias distribution putatively defined by \eqref{char.freezb} or \eqref{char.freezb.free} does indeed exist and is unique; the theorem also collects all of its main regularity properties.  Going forward, let $\rightharpoonup$ denote weak convergence of a sequence of probability measures. 

\begin{theorem} \label{main theorem 1} Given any $\sigma^2>0$ and probability measure $\mu$ in $\mathcal{D}_{0,\sigma^2}$, there is a unique probability measure $\mu^\circ$ on $\mathbb{R}$ that satisfies Equation \eqref{char.freezb} with $X\equaldist \mu$ and $X^\circ\equaldist \mu^\circ$.  This measure is called the {\bf free zero bias}; we also refer to any random variable $X^\circ$ with distribution $\mu^\circ$ as a {\bf free zero bias} of $X$.

The free zero bias has the following properties:
\begin{enumerate}[wide = 0pt,label=\alph*)] 
\item \label{thm1.a} The unique fixed point of $\mathcal{D}_{0,\sigma^2}\ni \mu\mapsto\mu^\circ$ is the semicircle law of mean $0$ and variance $\sigma^2$, with density $\frac{1}{2\pi\sigma^2}\sqrt{(4\sigma^2-x^2)_+}$.
\item \label{thm1.ac} For any mean zero $\mu$ with finite, non-zero variance, the free zero bias $\mu^\circ$ is absolutely continuous with respect to Lebesgue measure, and for all $a \le b$
$$
\mu^\circ([a,b])^2 \le \left(\frac{b-a}{\sigma^2}\right)E[|X|].
$$
\item \label{thm1.supp} The support of $\mu^\circ$ is contained in the convex hull of the support of $\mu$.
\item \label{thm1.d} If $\mu_n\rightharpoonup\mu$ and $\int t^2\mu_n(dt) \to \int t^2\mu(dt)$, then $\mu_n^\circ\rightharpoonup\mu^\circ$.
\end{enumerate}
\end{theorem}
The existence claim of Theorem 1.1 and part \ref{thm1.a} follow from a more constructive definition of the free zero bias in Definition \ref{def:fzb} which mirrors the construction of the zero bias in \eqref{eq:blah}, and is shown in Lemma \ref{lem:Xcirc.fixed.pt} to be equivalent to \eqref{char.freezb}; parts \ref{thm1.ac} and \ref{thm1.supp} are shown in Theorem \ref{thm:ac}, and part \ref{thm1.d} is Lemma \ref{lem:Gcirc.props}(\ref{lem.Gcirc.3}.

It is convenient to write the difference quotient in the form
\begin{equation} \label{diffquo.int}
f^{[1]}(x,y) = \frac{f(x)-f(y)}{x-y} = E[f'(Ux+(1-U)y)]
\end{equation}
where $U$ is a Uniform random variable on $[0,1]$; if, for a smooth $f$, we choose to define $f^{[1]}$ on the diagonal as $f^{[1]}(x,x) = f'(x)$, this identity holds for all $x,y$ including $x=y$.  A priori, Equation \eqref{char.freezb} requires the two variable function $f^{[1]}$ to be defined on the diagonal, and so we may canonically make the above choice when $f\in C^1$.
That being said: due to Theorem \ref{main theorem 1}(\ref{thm1.ac},
$X^\circ$ always has an absolutely continuous distribution, and as such, the diagonal values of $f^{[1]}$ do not affect the value of $E[f^{[1]}(X^\circ,Y^\circ)]$.

\begin{remark} Theorem \ref{main theorem 1}(\ref{thm1.supp} indicates that the free zero bias cannot ``expand'' the support of the measure, which is also true for the (classical) zero bias.  However, the two differ on this point: the support of the classical zero bias $\mu^\ast $ is always {\em equal} to the convex hull of the support of $\mu$, cf.\ \cite{goldstein1997stein}, while strict containment is possible for the free zero bias $\mu^\circ$ whose support need not be connected (see Example \ref{ex:three.pt.zb}, Theorem \ref{theorem disjoint support} and \eqref{eq:sup.disc.zb}).
\end{remark}

\begin{remark} Taking test functions $f\in C_b^\infty(\mathbb{R})$ and comparing \eqref{char.zb} and \eqref{char.freezb} in light of \eqref{diffquo.int}, we see the following relationship between the (classical) zero bias and the free zero bias:
\begin{align}\label{Xstar.intermsof.Xcirc}
X^*\equaldist UX^\circ + (1-U)Y^\circ
\end{align}
where $U$ is a uniform random variable on $[0,1]$ (classically) independent from $X^\circ$ and $Y^\circ$.  This equivalence is proved carefully as \eqref{eq:X*.unif.interpolation} in Lemma \ref{lem:Xcirc.fixed.pt} and will be useful below in Sections \ref{sec:exist} and \ref{sec:abscont}.
\end{remark}

A number of benefits reaped by the classical zero and size bias transformations are consequences of a key `replace one property' that allows one to construct the zero bias, or size bias, of a sum of independent variables by replacing a single summand. We derive the free analogs in Theorems \ref{eq:Todd's.formula} and \ref{eq:size.bias.change.1.formula}. Though these forms are somewhat more complicated than their classical counterparts, they still expresses the (free) zero bias and size bias of a sum in terms of the (free) zero bias and size bias of a single summand.

\medskip

As a main application of our results on and around the free zero bias, we give a new approach to, and provide new results about, {\bf free infinite divisibility}.  Recall that a random variable $X$ is (classically) infinitely divisible if, for each $n \ge 1$, $X$ can be represented as an independent sum
\begin{align}\label{def:X.inf.div}
X \equaldist X_{1,n}+\cdots+X_{n,n} \quad \mbox{where $\{X_{j,n}\}_{1\le j\le n}$ are i.i.d.} 
\end{align}
In the 1930s Kolmogorov, and then independently L\'evy and Khintchine, proved the following: an $L^2$ random variable $X$ with mean $m$ and variance $\sigma^2$ is infinitely divisible if and only if there exist a probability measure $\nu$ on $\mathbb{R}$, known as the {\bf L\'evy--Khintchine measure}, such that the characteristic function $\psi_X(\xi) = E[\exp(i\xi X)]$ of $X$ satisfies
 \begin{align} \label{eq:logphi.zero}
\log \psi_X(\xi) = im\xi + \sigma^2\int_{\mathbb{R}}\frac{\exp(i\xi t)-1-i\xi t}{t^2}\,\nu(dx), \qquad t \in \mathbb{R}.
\end{align}
(Here the integrand is taken to equal its limit value $-\frac12\xi^2$ at $t=0$; this is relevant if $\nu$ has mass at $0$, which happens, for example, when $X$ is normally distributed.)  The measure $\nu$ is unique when $\sigma^2>0$.  These results were later extended beyond the $L^2$ setting to all infinitely divisible distributions; see Section \ref{sect.background.inf.div}.

The (classical) zero bias and size bias transformations have deep connections to infinitely divisible distributions. \cite[Proposition 3.8]{arras2019stein}, \cite[Lemma 4.2(b), eq.\ (4.3)]{liu2022geometric}, \cite{uwe}, \cite{steutel1973some}, and \cite{arratia2019size} give, via the `replace one property', insightful characterizations of infinitely divisible probability distributions on the real line that shed light on the meaning of the L\'evy--Khintchine measure.   Notably, in \cite{uwe} it is shown that \eqref{eq:logphi.zero} holds if and only if on a possibly enlarged probability space there exists a uniform $U \equaldist \mathrm{Unif}[0,1]$ random variable and $Y_0$ such that
\begin{equation}\label{eq:ind.inc.zero}
X^* \equaldist X + UY_0\qquad\text{with $U$, $X$, $Y_0$ independent,}
\end{equation}
and in this case, the distribution of $Y_0$ in \eqref{eq:ind.inc.zero} is precisely the L\'evy--Khintchine measure $\nu$ in \eqref{eq:logphi.zero}. The work \cite{steutel1973some} yields an analogous result in the positive $L^1$ category, showing that ${\mathcal L}(X) \in \mathcal{D}_m^+$ is infinitely divisible if and only if its size bias $X^s$ satisfies
\begin{equation}\label{eq:ind.inc.sb}
X^s \equaldist X+Y_+ \qquad\text{with $X$ and $Y_+$ independent.}
\end{equation}

In this paper, we prove free probability analogs of \eqref{eq:ind.inc.zero} and 
\eqref{eq:ind.inc.sb}, giving new characterizations of the free probability version of infinite divisibility and, in fact, shedding new light on both the classical and free L\'evy--Khintchine formulas.

To properly state these results, we briefly introduce the distributional transforms needed.  The characteristic function $\psi_X$ encodes the distribution of $X$, and also encodes independence particularly well in terms of the cumulant-generating function $\log\psi_X$: $X,Y$ are independent iff $\log\psi_{X+Y} = \log\psi_X + \log\psi_Y$ (i.e.\ the cumulant-generating function linearizes independence).  In free probability, instead of $\psi_X$ it is more convenient to use the {\bf Cauchy transform}\label{page.CauchyTransform}
$$
G_X(z) = E\left[ \frac{1}{z-X} \right], \quad \mbox{$z \in \mathbb{C}_+$}
$$
where $\mathbb{C}_+$ denotes the upper half-plane $\{z\in\mathbb{C}\colon \mathrm{Im}\, z>0\}$.  A Cauchy transform $G_X$ takes values in the lower half-plane $\mathbb{C}_- = -\mathbb{C}_+$; as such, since $0$ is not in its range, the {\bf reciprocal Cauchy transform} $F_X = 1/G_X$ is well-defined, and will also be used throughout this paper.

The `free' in free probability refers to an alternative notion of independence, modeled on free groups, which makes sense for a wider class of random variables that do not necessarily commute under multiplication (for example high-dimensional random matrices); see Definition \ref{def:freeness}.  The analog of the cumulant-generating function which linearizes {\bf free independence} is called the {\bf $R$-transform} $R_X$, defined in an open domain in $\mathbb{C}_+$; its relation \eqref{Grr:relations} to $G_X$ is not as simple as the exponential relation between $\psi_X$ and $\log\psi_X$, but the two functions determine each other in an often computationally feasible way, and $R_{X+Y} = R_X+R_Y$ precisely when $X,Y$ are freely independent.  Corresponding to the reciprocal Cauchy transform, it is often useful to use a  ``reflected'' version of $R_X$, the {\bf Voiculescu transform} $\varphi_X(z) = R_X(1/z)$. \label{page.voiculescu} Section \ref{sect:R.Voic.trans} below describes all these transforms and their domains in precise detail.

{\bf Free infinite divisibility} is defined as in \eqref{def:X.inf.div} with (classical) independence of the summands $X_{j,n}$ replaced by free independence; see Section \ref{sect.background.inf.div}. In \cite{bercovici1993free}, Bercovici and Voiculescu proved that an $L^2$ random variable $X$ with mean $m$ and variance $\sigma^2$ is freely infinitely divisible if and only if its $R$-transform is of the form
\begin{align} \label{LK.via.nica}
R_X(z)= m + \sigma^2 \int_{\mathbb{R}} \frac{z}{1-tz}\, \nu(dt)
\end{align}
for some probability measure $\nu$ on $\mathbb{R}$.  Equation \eqref{LK.via.nica} is a free analog of the L\'evy--Khintchine formula \eqref{eq:logphi.zero}.  It is no accident that we use $\nu$ to denote the measure here as we did in \eqref{LK.via.nica}.  The triple $(m,\sigma^2,\nu)$ characterizes all $L^2$ classically infinitely divisible distributions and likewise characterizes all $L^2$ freely infinitely divisible distributions.  This sets up an important bijection, known as the {\em Bercovici--Pata bijection} cf.\ \cite{BercoviciPata1999}, between the classical and free cases, which we describe in Section \ref{sect.background.inf.div}.  (Moreover: as in the classical case, the setup can be modified to capture all freely infinitely divisible distributions, not only $L^2$ ones; we explain this in Section \ref{sect.background.inf.div}.)

The equivalence between the free L\'evy--Khintchine formula \eqref{LK.via.nica} (or more generally \eqref{BV2.LK-formula}) and free infinite divisibility was a landmark result of Bercovici--Voiculescu in \cite{bercovici1993free,bercovici1992levy}.  Theorem \ref{main theorem 2} below gives a different approach to this equivalence along with a new characterization --- the free analog \eqref{eq:free.Schmock} of \eqref{eq:ind.inc.zero} --- using the free zero bias.  What's more, in Theorem \ref{main theorem 3}, we provide a completely new parallel development and new equivalent characterization of {\em positively} freely infinitely divisible distributions (see Definition \ref{def:pfid}) in terms of a different L\'evy--Khintchine type formula, using the size bias.

\ignore{
a third equivalent characterization of free infinite divisibility for $L^2$ random variables, stated in terms of the free zero bias, yielding the free analog . Moreover, our theorem gives a probabilistic interpretation for what the L\'evy--Khintchine measure of $X$ actually {\em is}: namely, it is the limit of the square biases of the ``$n$th root'' distributions of $X$.
}

In order to state these theorems, we need the following result which introduces (for our purposes) a free analog of multiplying a random variable by an independent uniform random variable.

\begin{lemma} \label{lem:elGordo} Let $\mu$ be any probability measure on $\mathbb{R}$.  There is a unique probability measure $\mu^\bflat$ whose Cauchy transform satisfies $G_{\mu^\bflat}(z) = -\sqrt{G_\mu(z)/z}$. \end{lemma}

If $W$ has distribution $\mathcal{L}(W) = \mu$, we abuse notation as usual and let $W^\bflat$ denote a random variable with distribution $\mathcal{L}(W^\bflat) = \mu^\bflat$. Lemma \ref{lem:elGordo} is restated and proved as Lemma \ref{lem:bflat}, which follows from the curious fact (Lemma \ref{lem:root.Geometric.Mean}) that the geometric mean of two Cauchy transforms of probability measures is again the Cauchy transform of some probability measure; this result may be of independent interest.

\begin{remark} We view the operation $W\mapsto W^\bflat$ as a free version of the operation $Y\mapsto UY$ of multiplying by an independent $\mathrm{Unif}[0,1]$ random variable because of the form of Theorem \ref{main theorem 2}(\ref{thm2.c} as it relates to \eqref{eq:ind.inc.zero}; see also our construction \eqref{eq:circ.equiv.box} of the free zero bias compared to its classical counterpart \eqref{eq:blah}.  To date, the authors have not been able to realize $W\mapsto W^\bflat$ in more probabilistic terms (i.e.\ as a direct operation on random variables, or on von Neumann algebras) rather than on distributional transforms; this is a good question for future thought.
\end{remark}

The following main theorem introduces a new take, and a totally new equivalent condition, on infinite divisibility in the free world.

\begin{theorem} \label{main theorem 2}  Let $X$ be a random variable with mean $m\in\mathbb{R}$ and variance $\sigma^2\in(0,\infty)$.  The following are equivalent.
\begin{enumerate}[wide = 0pt,label=\alph*)] 
    \item \label{thm2.a} $X$ is freely infinitely divisible.
    \item \label{thm2.b} There is a random variable $V_0$ so that the Voiculescu transform $\varphi_X$ has the form
    \begin{align} \label{eq:LK0}
    \varphi_X(z) = m + \sigma^2 G_{V_0}(z). \end{align}
    \item \label{thm2.c} The free zero bias of $X$ satisfies
    \begin{equation} \label{eq:free.Schmock} F_{X^\circ}(z) = F_{V_0^\bflat}(F_X(z)),\end{equation}
    for some random variable $V_0$.
\end{enumerate}
In the case that these equivalent statements hold true, the random variables $V_0$ in \ref{thm2.b} and \ref{thm2.c} have the same law, which is the weak limit as $n\to\infty$ of the square biases $X_{1,n}^\Box$ (cf.\ \eqref{def:Xbox}) of the summand distributions where $X=X_{1,n}+\cdots+X_{n,n}$ for freely independent identically distributed $\{X_{j,n}\}_{1\le j\le n}$.
\end{theorem}

The equivalence of \ref{thm2.a} and \ref{thm2.b} was proved in \cite{bercovici1993free}, and the characterization of the free L\'evy measure as a limit of square biases (without using that language) was proved in \cite[Theorem 3.3(2)]{BercoviciPata1999}.  Our contribution, other than providing probabilistic arguments rather than purely complex analytic ones, is the third leg \ref{thm2.c} of the equivalence, which crucially gives the free L\'evy--Khintchine measure $\mathcal{L}(V_0)$ new meaning, relating to the free zero bias.  Our proof, which is the subject of Section \ref{sec:infdiv}, proceeds by first treating compound free Poisson distributions (cf.\ Section \ref{sect:FP}), and then extending results broadly by showing all freely infinitely divisible distributions are well-controlled limits of compound free Poissons.  (Indeed, we prove that compound free Poisson distributions are precisely those for which the random variable $V_0$ in Theorem \ref{main theorem 2} satisfies $E[V_0^{-2}]<\infty$; see Theorem \ref{thm:Compound.Poisson}.)  Our proof is fully self-contained in the case of compact support (Section \ref{sect:compact}), relying weakly on the results of \cite{bercovici1992levy,bercovici1993free} for requisite tightness in the general case.

We also move beyond the $L^2$ category and treat $L^1$ freely infinitely divisible distributions that are positive in a strong sense.

\begin{definition} \label{def:pfid}
We say $X$ is {\bf positively freely infinitely divisible}, or free regular as in \cite{arizmendi2011law}, when for all $n \ge 1$
\begin{align} \label{eq:X.is.sum.X_{i,n}}
X \equaldist X_{1,n}+ \cdots + X_{n,n}
\end{align}
where $\{X_{j,n}\}_{1 \le j \le n}$ are freely independent, identically distributed, and $X_{j,n}\ge 0$.
\end{definition}

The key defining feature here is that the summands $X_{j,n}$ are all non-negative.  For classically infinitely divisible random variables that are non-negative, this holds automatically (owing to the linearity of classical convolution); in the free world, it is a strong constraint: for example shifted semicircular distributions can be positive and are freely infinitely divisible, but are never positive{\em ly} freely infinitely divisible; see Section \ref{sect.background.inf.div}.

\begin{theorem} \label{main theorem 3}
 Let $X$ be a non-negative random variable with mean $m>0$.  The following are equivalent.
\begin{enumerate}[wide = 0pt,label=\alph*)] 
    \item \label{thm2.a.sb} $X$ is positively freely infinitely divisible.
    \item \label{thm2.b.sb} There is a random variable $V_+$ so that the Voiculescu transform $\varphi_X$ has the form
    \begin{align} \label{eq:LK+}
    \varphi_X(z) = mzG_{V_+}(z). \end{align}
    \item \label{thm2.c.sb} The size bias of $X$ satisfies
    \begin{equation} \label{eq:free.Schmock.sb} 
    F_{X^s}(z) = F_{V_+}(F_X(z)). 
    \end{equation}
\end{enumerate}
In the case that these equivalent statements hold true, the random variables $V_+$ in \ref{thm2.b.sb} and \ref{thm2.c.sb} have the same law, which is the weak limit as $n\to\infty$ of the size biases $X_{1,n}^s$ (cf.\ \eqref{def:Xs.0}) of the summand distributions where $X=X_{1,n}+\cdots+X_{n,n}$ for freely independent identically distributed $\{X_{j,n}\}_{1\le j\le n}$.
\end{theorem}

Both conditions \ref{thm2.b.sb} and \ref{thm2.c.sb} are entirely new to the theory of free infinite divisibility; moreover, the `positive L\'evy--Khintchine formula' of \eqref{eq:LK+} has no classical equivalent that the authors are aware of.

One more advantage of our approach: it yields the previously unknown result that that {\em every} probability measure is the free L\'evy--Khintchine measure for a two-parameter family of $L^2$ freely infinitely divisible distributions, indexed by their mean and variance; also, every positively-supported probability measure yields a one-parameter family of positively freely infinitely divisible distributions, indexed by their mean.

\begin{proposition} \label{prop.any.V.will.do} Let $\nu$ be a probability measure on $\mathbb{R}$, and let $m\in\mathbb{R}$ and $\sigma^2\ge 0$.
\begin{enumerate}[wide = 0pt]
    \item \label{prop.any.V.1} There is a unique $L^2$ freely infinitely divisible random variable $X$ with mean $m$ and variance $\sigma^2$ such that $\varphi_X(z) = m+\sigma^2 G_\nu(z)$; i.e.\ such that Theorem \ref{main theorem 2}(\ref{thm2.b} holds with $V_0\equaldist \nu$.
    \item \label{prop.any.V.2} If $\mathrm{supp}\,\nu\subseteq[0,\infty)$ and $m>0$, there is a unique $L^1$ positively freely infinitely divisible random variable $X$ with mean $m$ such that $\varphi_X(z) = mzG_\nu(z)$; i.e.\ such that Theorem \ref{main theorem 3}(\ref{thm2.b.sb} holds with $V_+\equaldist \nu$.
\end{enumerate}
\end{proposition}
The proof can be found in Propositions \ref{prop:ForAnyV} and \ref{prop:ForAnyV.sb}.

In the special case that a positively freely infinitely divisible random variable is not just $L^1$ but $L^2$, it possess both types of free L\'evy--Khintchine measures, i.e.\ both $V_0$ and $V_+$.  They are nicely related, which also gives a new elegant characterization of $L^2$ positively infinitely divisible distributions (and their translates).

\begin{theorem} \label{prop:L2.pos.fee.inf.div} Let $X$ be a freely infinitely divisible $L^2$ random variable. Assume $X$ is bounded below.  The following are equivalent:

\begin{enumerate}
    \item \label{L2.pos.(1)} There is some $\alpha>0$ so that $X+\alpha$ is  a positively freely infinitely divisible random variable.
    \item \label{L2.pos.(2)} The random variable $V_0$ corresponding to $X$ in Theorem \ref{main theorem 2}(\ref{thm2.b} is non-negative and has a finite negative moment: $V_0\ge 0$ and $E[V_0^{-1}]<\infty$.
\end{enumerate}
When these equivalent conditions hold, if $X'$ is any positively freely infinitely divisible translate of $X$, and $V_+$ is the random variable corresponding to $X'$ in Theorem \ref{main theorem 3}(\ref{thm2.b.sb}, then
\[ (V_+)^s = V_0 \quad \text{ and }\quad E[V_+] = \frac{\mathrm{Var}[X']}{E[X']}. \] 
\end{theorem}
The proof of this final main theorem appears towards the end of Section \ref{sec:inf.div}. 

\begin{remark}\label{remark.translations.1}
For $X$ and $Y$ two freely independent, positively freely infinitely divisible $L^1$ random variables with means $m,m'$ respectively, by 
the additivity of the Voiculescu transformation
and Theorem \ref{main theorem 3}(\ref{thm2.b},
\begin{align*}
\varphi_{X+Y}(z) = mzG_{V_+(X)}(z)+m'zG_{V_+(Y)}(z)= (m+m')zG_{V_+(X+Y)}(z)
\end{align*}
where 
\begin{align} \label{eq:V_+(X+Y)}
V_+(X + Y) = \frac{m}{m + m'} V_+(X) + \frac{m'}{m + m'} V_+(Y).
\end{align}
That is, the free Lévy--Khintchine measure $V_+$ of the sum is the mixture of the measures of the summands, taken proportional to the means.

Similarly when $X$ and $Y$ are freely independent, freely infinitely divisible $L^2$ random variables we have via Theorem \ref{main theorem 2}(\ref{thm2.b} that 
$$
\varphi_{X+Y}(z) = m+m' + (\sigma^2+\tau^2)G_{V_0(X+Y)}(z)
$$
where $\sigma^2,\tau^2$ are the respective variances of $X$ and $Y$, and 
\begin{align}\label{eq:V_0(X+Y)}
V_0(X+Y)=\frac{\sigma^2}{\sigma^2+\tau^2}V_0(X)+\frac{\tau^2}{\sigma^2+\tau^2}V_0(Y). 
\end{align}
That is, in the $L^2$ case the free Lévy--Khintchine measure $V_0$ of the sum is the mixture taken proportional to the variances. 

By the relation $\varphi_{\alpha X}(z)=\alpha \varphi (z/\alpha)$ one can verify that non-zero scalings in the $L^2$ case, and positive scalings in the $L^1$ case, both yield the same scaling of $V_0$ and $V_+$ respectively. However, the two cases differ widely for translations, as can be seen by specializing the relations above to the case where $Y$ is a point mass. 

In particular, for a given infinitely divisible $L^2$ random variable $X$, its corresponding $V_0$ in Theorem \ref{main theorem 2}(\ref{thm2.b} is invariant under translations of $X$; the free L\'evy--Khintchine formula \eqref{eq:LK0} only changes by translating the isolated mean term. 

On the other hand, as Theorem \ref{prop:L2.pos.fee.inf.div} involves translations which leave $V_0$ invariant but do change $V_+$, one may wonder about the validity of the final claim that $(V_+)^s = V_0$.  This apparent problem is resolved by specializing \eqref{eq:V_+(X+Y)} to the case $\mathcal{L}(Y)=\delta_{m'}$, yielding
\begin{equation} \label{eq.V+.translate}  V_+(X+m') =  \frac{m}{m+m'}V_+(X) + \frac{m'}{m+m'}\delta_0, \end{equation}
which shows that translations affect $V_+$ only by mixing with a point mass at $0$, and such mixtures do not change the size bias of $V_+$, cf.\ Lemma \ref{lem:X.Xneg0.same.sb}(\ref{lem:sqbiaa-part2.sb}.

More generally, we may move directly from \eqref{eq:V_+(X+Y)} to \eqref{eq:V_0(X+Y)} using 
\eqref{eq:WBox.as.mixture.sb}, and in the reverse direction taking the inverse size bias of both its sides, and also using the final claim of Theorem \ref{prop:L2.pos.fee.inf.div} which provides the factor $E[V_+]=\sigma^2/m$ that converts a mixture by the mean to one taken by the variance.
\end{remark}

This work is organized as follows.  The remainder of this Introduction, Section \ref{sect:Stein.kernels}, discusses the connection between the zero bias transform and Stein kernels in both the classical and free probability contexts.  Section \ref{sec:free prob} is a fairly self-contained primer on free probability, in particular as it relates to analytic distributional transforms, conditioning, and infinite divisibility.  In Section \ref{sec:exist} we prove the existence of the free zero bias and related distributional transforms, and develop their basic properties and relations to each other. In particular, Theorems \ref{eq:Todd's.formula} and \ref{eq:size.bias.change.1.formula} provide free analogs of the classical key `replace one properties' for the free zero and size bias transformations.  In Section \ref{sec:abscont} we show that, like its classical cousin, the free version zero bias has an absolutely continuous distribution, completing the proof of Theorem \ref{main theorem 1}. and explore how disconnected its support can be with a full treatment of free zero bias transforms of finitely-supported distributions. We then turn, in Section \ref{sec:infdiv}, to infinitely divisible random variables, and prove Theorems \ref{main theorem 2}, \ref{main theorem 3}, and their consequences discussed above.  Finally, Section \ref{sect:examples} uses the results of Theorems \ref{main theorem 2} and \ref{main theorem 3} to compute (quite effectively) the distributions of several new freely infinitely divisible distributions, including a new two-parameter family of positively freely infinitely divisible distributions.

\subsection{(Free) Zero Bias and (Free) Stein Kernels}\label{sect:Stein.kernels}

Our work is related to that of \cite{cebron2020note} and  \cite{fathi2017free}, where Stein kernels, familiar in classical probability \cite{cacoullos1992lower,chatterjee2009fluctuations}, are shown to also exist in free probability. Classically, we say that a random variable $A_\ast(X)$ is a Stein kernel for $X$ if
\begin{align} \label{eq:classical.kernel}
E[Xf(X)]=E[A_\ast(X)f'(X)] \qmq{for all $f \in C_b^\infty(\mathbb{R})$.}
\end{align}
We note that \eqref{eq:classical.kernel} is similar to \eqref{char.zb} in that it modifies Stein's original identity for the normal by changing some feature on the right hand side, for \eqref{char.zb} the distribution of $X$, and for \eqref{eq:classical.kernel} the multiplier of the derivative of $f$.  If the law $\mathcal{L}(X^*)$ is absolutely continuous with respect to $\mathcal{L}(X)$ with Radon-Nikodym derivative $D_*$, then from \eqref{char.zb} and the change-of-variables theorem,
\begin{align*}
E[Xf(X)]=\sigma^2 E[f'(X^*)] =E\left[\sigma^2 D_*(X)f'(X)\right].
\end{align*}
In particular, under the stated absolute continuity requirement, $A_\ast(X) = \sigma^2 D_*(X)$ provides a Stein kernel.

The notion of a Stein kernel in free probability was first introduced in \cite{fathi2017free}, and used to obtain a free counterpart of the HSI inequality, connecting entropy, Stein discepancy and free Fisher information, as well as a rate of convergence in the entropic free central limit theorem. The work \cite{cebron2020note} showed that free Stein kernels, unlike in the classical case, always exist; moreover, it showed that free Stein kernels are never unique by providing two constructions that lead to distinct results, one inspired by an unpublished note of Giovanni Peccati and Roland Speicher in the case where the potential is $V(x)=n^{-1}\sum_{i=1}^n x_i^2$ and where $n=1$, and a second one constructed via the Reisz representation theorem.

For a single random variable $X$, we say that $A_0(X,Y)$ is a Free Stein kernel for $X$ when (consistent with our choice above to replace derivative with difference quotient \eqref{diffquo.int})
\begin{align} \label{eq:free.kernel}
E[Xf(X)]=E[A_0(X,Y)f^{[1]}(X,Y)] \qmq{for all $f \in C_b^\infty(\mathbb{R})$,}
\end{align}
where $Y$ is a (classically) independent copy of $X$. A relation similar to the one in the classical case holds between free Stein kernels and the free zero bias. Indeed, when the distribution $\mathcal{L}(X^\circ)$ of $X^\circ$ is absolutely continuous with respect to $\mathcal{L}(X)$ with Radon--Nikodym derivative $D_0$, from \eqref{char.freezb} and the change-of-variables theorem we have
\begin{align} \label{eq:zb.kernel.is.product}
E[Xf(X)]=\sigma^2 E[f^{[1]}(X^\circ,Y^\circ)]=\sigma^2 E[D_0(X)D_0(Y)f^{[1]}(X,Y)].
\end{align}
Comparing with \eqref{eq:free.kernel}, we see, as in the classical case under a similar absolute continuity assumption, that free zero bias provides a free Stein kernel, here given by $A_0(X,Y)=D_0(X)D_0(Y)$, which has the form of a product of some function of $X$ multiplied by the same function of $Y$.

In the setting of \eqref{eq:free.kernel}, the first of the two distinct Stein kernels discussed in \cite{cebron2020note} corresponds to $A_0(x,y)=(x-y)^2/2$, for which \eqref{eq:free.kernel} is nearly immediately verified to hold upon applying the definition of the difference quotient $f^{[1]}$ and the fact that $E[(X-Y)(f(X)-f(Y))]=2E[Xf(X)]$. This kernel must be distinct from the one in \eqref{eq:zb.kernel.is.product} generated by the free zero bias, as $(x-y)^2$ cannot be written in the necessary product form. The construction of the second kernel in \cite{cebron2020note} requires that $X$ satisfies the free Poincaré inequality. However, no such requirement is needed for the kernel for $X$ in \eqref{eq:zb.kernel.is.product}.  In other words: the free zerio bias affords yet a third distinct construction of a free Stein kernel.

\section{A Free Probability Primer}\label{sec:free prob}

Free probability is a noncommutative probability theory introduced by Voiculescu in the late 1980s and early 1990s, initially with the intent to import core ideas from probability, statistics, and information theory to solve structure and isomorphism problems in operator algebras.  Through its intimate connection with random matrix theory, it has grown over the last three decades into a flourishing field of study on its own.  The reader may wish to consult \cite{nica2006lectures,mingo2017book} for detailed treatises on free probability.

\subsection{Noncommutative Probability Spaces}

The arena is a {\em noncommutative probability space} $(\mathcal{A},E)$, which replaces the algebra of bounded random variables and the expectation functional on them in the classical case.  $\mathcal{A}$ is a unital $\ast$-algebra, like the algebra of $n\times n$ complex matrices for example, where the adjoint $\ast$ operation is conjugate transpose.  (Note: to avoid notational conflict with the (classical) zero bias of \eqref{char.zb}, for the rest of this section we use $\dagger$ to refer to the adjoint operation rather than $\ast$.) $E\colon\mathcal{A}\to\mathbb{C}$ is a {\em state}: a linear functional satisfying $E(1)=1$ that is {\em positive} in the sense that, for any $X\in\mathcal{A}$ other than $0$, $E(X^\dagger X) >0$.  As the notation suggests, $E$ should be thought of as {\em expectation}; indeed, if $\mathcal{A}$ is the algebra of bounded random variables on some probability space, then {\em the} expectation is the standard choice for $E$.  Some authors prefer to use a letter like $\tau$ or $\varphi$ for the state functional in free probability; we will use $E$ throughout.

For some results (which we will need), more assumptions are required.  A {\em $W^\ast$-probability space} is a noncommutative probability space $(\mathcal{A},E)$ where $\mathcal{A}$ is a von Neumann algebra (a weakly closed subalgebra of the algebra of bounded operators on a Hilbert space), and $E$ is {\em tracial}: $E(X_1X_2) = E(X_2X_1)$ for all $X_1,X_2\in\mathcal{A}$, and also continuous in the weak$^\ast$ topology (identifying the algebra with its predual).  This latter continuity condition, sometimes called {\em normal continuity}, is designed to make a version of the Dominated Convergence Theorem hold for $E$; in particular, an equivalent statement is that, for any bounded increasing net of positive operators $\{X_\alpha\}$, $\sup_\alpha E[X_\alpha] = E[\sup_\alpha X_\alpha]$.  We will make these assumptions going forward, for convenience.  Typical examples are $n\times n$ matrices equipped with the normalized trace; any classical probability space; and various standard type $\mathrm{II}_1$ von Neumann algebras.  As with classical probability, one doesn't typically work in some particular explicit (noncommutative) probability space; the random variables (i.e.\ elements of $\mathcal{A}$) are assumed to ``live'' on a space with enough structure to make the usual results work, and that typically means a $W^\ast$ probability space.

\subsection{Free Independence}

The {\em free} in free probability refers to a noncommutative ``independence'' notion for random variables. Algebraically, statistical independence of random variables (corresponding to joint laws being product measures) corresponds to tensor product; in free probability, the independence rule is instead modeled on {\em free} product from group theory.  Ultimately, it is best to express it as a moment-factoring condition. For classical bounded random variables, independence can be stated combinatorially as the exepctation factoring property $E(X_1^nX_2^m) = E(X_1^n)E(X_2^m)$ for all $n,m\in\mathbb{N}$.  This equality will also hold for freely independent random variables, but it is not sufficient: when random variable multiplication is noncommutative, this rule doesn't handle expressions like $E(X_1X_2X_1X_2)$.  To handle such expressions, freeness is most easily stated in terms of subalgebras.

\begin{definition}\label{def:freeness} Let $(\mathcal{A},E)$ be a noncommutative probability space. 
 A collection $\mathcal{A}_1,\ldots,\mathcal{A}_d$ of unital $\ast$-subalgebras of $\mathcal{A}$ is called {\bf freely independent} if the following holds true: given indices $i_1,i_2,\ldots,i_n$ that are consecutively distinct (i.e.\ $i_1\ne i_2$, $i_2\ne i_3$, etc.) and random variables $X_k\in \mathcal{A}_{i_k}$ that are centered (i.e.\ $E(X_k)=0$), the product $X_1X_2\cdots X_k$ is also centered, that is, $E[X_1X_2\cdots X_k]=0$.

We say random variables $X_1,\ldots,X_d$ are freely independent if the $\ast$-subalgebras they generate are freely independent.
 \end{definition}

For our purposes, it is only necessary to deal with {\em selfadjoint} variables $X=X^\dagger$. For a selfadjoint $X$, the $\ast$-algebra it generates is just the set of all polynomial functions of it, $\{p(X)\colon p\text{ is a polynomial of one variable}\}$.  For such variables, then, free independence can be expressed a bit more directly: selfadjoint $X_1,\ldots,X_d$ are freely independent if, given any consecutively distinct indices $i_1\ne i_2\ne i_3\ne\cdots\ne i_n$, and polynomials $p_1,\ldots,p_n$ such that $E[p_1(X_{i_1})]=\cdots=E[p_n(X_{i_n})]=0$, it follows that $E[p_1(X_{i_1})\cdots p_n(X_{i_n})]=0$.  By using the usual centering trick, applying this property to variables $X_i^m-E[X_i^m]$, one learns that free random variables satisfy the usual factoring property $E[X_1^nX_2^m]=E[X_1^n]E[X_2^m]$, as well as a litany of more complex noncommutative factoring properties, for example $E[X_1X_2X_1X_2] = E[X_1^2]E[X_2]^2 + E[X_1]^2E[X_2^2]-E[X_1]^2E[X_2]^2$.  In general, if random variables are freely independent, any mixed moment can be expressed (as a polynomial function) in terms of the variables' individual moments.

Another benefit of selfadjoint variables is that, like classical random variables, they have probability distributions.  Associated to each selfadjoint $X$ in a $W^\ast$ probability space is a unique probability measure $\mu_X$ on $\mathbb{R}$ sharing the same moments:
\[ \int_{\mathbb{R}} x^n\,\mu_X(dx) = E[X^n], \qquad n\in\mathbb{N}. \]
Mirroring the classical setting, we denote $\mu_X = \mathcal{L}(X)$ as the {\bf law} of $X$.  Existence follows from the spectral theorem for selfadjoint operators: associated to $X$ is a projection-valued Borel measure $\Pi^X$ on $\mathbb{R}$ such that $X = \int_{\mathbb{R}} x\,\Pi^X(dx)$, and then functional calculus of $X$ is given by $f(X) = \int_{\mathbb{R}} f(x)\,\Pi^X(dx)$ (which is consistent with analytic functional calculus when $f$ is analytic, e.g.\ a polynomial).  Composing with the state $E$ on both sides yields $E[f(X)] = \int_{\mathbb{R}} f(x)\,E\circ\Pi^X(dx)$, showing that $\mu_X = E\circ\Pi^X$ is the desired law.

For now, we restrict to the $W^\ast$ probability space context, meaning that the random variables are bounded operators on a Hilbert space.  This implies that their distributions are compactly-supported (we will discuss removing this requirement shortly), and hence uniquely determined by their moments, which establishes uniqueness of the law.  Moreover, any compactly-supported probability distribution $\mu$ arises this way (just take any classical random variable with law $\mu$).  A deeper result is that any collection of compactly-supported probability measures can be coupled together as the laws of freely independent random variables on a common space (see Definition \ref{def:freeness}).  This leads to a new operation on such measures, {\bf free convolution} $\boxplus$.  The measure  $\mu\boxplus\nu$ is defined to be the law of $X_1+X_2$ where $X_1$ and $X_2$ are freely independent with laws $\mu$ and $\nu$ respectively.  To see this is well defined, we look at moments.  The moments of $X_1+X_2$ are all linear combinations of mixed moments in $X_1$ and $X_2$, which in turn factor in terms of moments of $X_1$ and $X_2$ separately.  Ergo, the moments of $X_1+X_2$ are completely determined by the moments of $X_1$ and $X_2$ separately.
Unlike the classical case, there is no simple integral formula for $\mu\boxplus\nu$ in terms of $\mu$ and $\nu$; but there is an approach using analytic transforms which we now describe.

\subsection{Cauchy Transform}

As mentioned on page \pageref{page.CauchyTransform}, for a probability measure $\mu$ on $\mathbb{R}$, its {\bf Cauchy transform} $G_\mu$ is the analytic function defined on the complement of the support of $\mu$ by
\begin{equation} \label{eq:def.G.int} G_\mu(z) = \int_{\mathbb{R}} \frac{1}{z-t}\,\mu(dt). \end{equation}
If $X$ is a selfajoint random variable, then from the definition (above) of its law, we have
\[ G_{\mu_X}(z) = E[(z-X)^{-1}] \]
where the function $f_z(x) = (z-x)^{-1}$ is real analytic on $\mathbb{R}$ for any $z\in\mathbb{C}_+$.  This reflects the change-of-variables formula in classical probability.  When convenient, we may write this as $G_X = G_{\mu_X}$.  

Any Cauchy transform is an analytic map from the upper half plane $\mathbb{C}_+ = \{x+iy\colon x\in\mathbb{R},y>0\}$ into the lower half plane $\mathbb{C}_- = -\mathbb{C}_+$.  (Since $\mu$ is a real measure, the domain of $G_\mu$ can really be taken as $\mathbb{C}\setminus\mathbb{R}$, and $G_\mu(\bar{z})= \overline{G_\mu(z)}$, hence it also maps $\mathbb{C}_-$ into $\mathbb{C}_+$.)  The Cauchy transform also satisfies $iy\,G_\mu(iy)\to 1$ as $y\uparrow\infty$, and these conditions characterize Cauchy transforms.

\begin{proposition} Let $G\colon \mathbb{C}_+\to\mathbb{C}_-$ be a holomorphic function with the property that $\lim_{y\to\infty} iy\,G(iy) = 1$.  Then there is a unique probability measure $\mu$ on $\mathbb{R}$ such that $G = G_\mu$.
\end{proposition}
\noindent More generally, the asymptotic $\frac1z$ behavior of Cauchy transforms holds along any path with $|z|\to\infty$ that stays in a cone $y>m|x|$, $m>0$ in the upper half plane.

In the case that a random variable $X$ has an exponential moment $E[\exp(\alpha|X|)]<\infty$ for some $\alpha>0$, the Cauchy transform $G_X$ has a Laurent series expansion
\begin{equation} \label{eq.mgf} G_X(z) = E\left[\frac{1}{z-X}\right] = \frac{1}{z}E\left[\frac{1}{1-X/z}\right] = \sum_{n\ge 0} E[X^n] z^{-(n+1)} \end{equation}
which converges locally uniformly for $|z|$ sufficiently large.  (Indeed, the ordinary moment generating function $M_X(z) = \sum_n E[X^n]z^n$ is given by $M_X(z) = \frac{1}{z}G_X(\frac{1}{z})$ in this case, and has positive radius of convergence.)

The Cauchy transform is well-defined even if $X$ has no moments at all; \eqref{eq.mgf} suggests that information about moments can be derived from certain limits or derivatives of the Cauchy transform.  We provide one simple such result here which will be useful in Section \ref{sec:inf.div}.

\begin{lemma} \label{lem.Cauchy.moment} Let $X\ge 0$ be a non-negative random variable.  Suppose that the limit
\[ \lim_{y\to\infty} iy\,(iy\,G_X(iy)-1) = r \]
exists and is in $\mathbb{R}$.  Then $X\in L^1$ and $E[X]=r$.   
\end{lemma}

\begin{proof} For any $z\in\mathbb{C}_+$, note that
\[ zG_X(z)-1 = zE\left[\frac{1}{z-X}\right]-1 = E\left[\frac{z}{z-X}\right]-E\left[\frac{z-X}{z-X}\right] = E\left[\frac{X}{z-X}\right]. \]
Hence
\[ z(zG_X(z)-1) = E\left[\frac{zX}{z-X}\right] = E\left[\frac{X}{1-X/z}\right]. \]
Taking $z=iy$ we have
\[ iy(iy\,G_X(iy)-1) = E\left[\frac{X}{1-X/iy}\right] = E\left[\frac{X(1+X/iy)}{1+X^2/y^2}\right]. \]
Since both $\frac{X}{1+X^2/y^2}$ and $\frac{X^2}{y(1+X^2/y^2)}$ are bounded hence $L^1$, this decomposes into its real and imaginary parts as
\[ iy(iy\,G_X(iy)-1) = E\left[\frac{X}{1+X^2/y^2}\right] - iE\left[\frac{X^2}{y(1+X^2/y^2)}\right]. \]
By assumption, the limit as $y\to\infty$ is equal to $r\in\mathbb{R}$, and so it follows that
\[ \lim_{y\to\infty} E\left[\frac{X}{1+X^2/y^2}\right] = r \quad\text{and}\quad \lim_{y\to\infty}E\left[\frac{X^2}{y(1+X^2/y^2)}\right]=0. \]
Since $X\ge 0$, $\frac{X}{1+X^2/t^2}$ increases to $X$ as $t\uparrow\infty$; thus, the result follows from the Monotone Convergence Theorem.
\end{proof}

The Cauchy transform of a probability measure $\mu$ completely characterizes $\mu$, and moreover this relationship is robust under weak limits.  The following is a complex analysis exercise; a detailed proof can be found in \cite[Section 8]{Kemp-RMT}.

\begin{proposition} \label{prop.Cauchy.robust} Let $\mu$ be a probability measure on $\mathbb{R}$, with Cauchy transform $G_\mu$.  The {\bf Stieltjes inversion formula} recovers $\mu$ from $G_\mu$ as the weak limit $\mu_\epsilon \rightharpoonup \mu$ as $\epsilon\downarrow 0$, where $\mu_\epsilon$ has real analytic density $\varrho_\epsilon$
\[ \varrho_\epsilon(x) =  -\frac1\pi\mathrm{Im}\, G_{\mu}(x+i\epsilon). \]

Let $\{\mu_n\}_{n=1}^\infty$ be a sequence of probability measures on $\mathbb{R}$.
\begin{enumerate}
\item \label{robust.1} If $\mu_n\rightharpoonup \mu$ for some probability measure $\mu$, then $G_{\mu_n}$ converges to $G_\mu$ uniformly on compact subsets of $\mathbb{C}_+$.
\item \label{robust.2} Conversely, if $G_{\mu_n}$ converges pointwise to a function $G$ that is analytic on $\mathbb{C}_+$, then $G=G_\mu$ for some finite measure $\mu$ with $\mu(\mathbb{R})\le 1$, and $\{\mu_n\}_{n=1}^\infty$ converges vaguely to $\mu$.
\end{enumerate}
\end{proposition}

\begin{example} \label{example.Cauchy.dist} Let $X$ be a Cauchy distributed random variable, with density $\varrho_X(x) = \frac{1}{\pi}\frac{1}{1+x^2}$.  Then its Cauchy transform is $G_X(z) = \frac{1}{z+i}$.  This can be computed from the definition using contour integrals; it also follows directly from Proposition \ref{prop.Cauchy.robust}.  Indeed, it is straightforward to check that the function $G(z) = \frac{1}{z+i}$ is a Cauchy transform, since $zG(z)\to 1$ as $|z|\to\infty$ and $G(\mathbb{C}_+) \subseteq \mathbb{C}_-$.  From the Stieljtes inversion formula, the measure whose Cauchy transform is $G$ has density
\[ -\lim_{\epsilon\downarrow 0} \frac{1}{\pi}\mathrm{Im} G(x+i\epsilon) = -\lim_{\epsilon\downarrow 0} \frac{1}{\pi}\mathrm{Im}\,\left(\frac{1}{x+i\epsilon +i}\right) = \frac{1}{\pi}\frac{1}{1+x^2}. \]
Thus, the unique measure whose Cauchy transform is $G(z) = \frac{1}{z+i}$ is the Cauchy distribution.

This gives the odd association that the Cauchy distribution behaves like a point mass at $-i$.  Indeed, it points out that the injectivity $\mu\mapsto G_\mu$ only holds for probability measures on $\mathbb{R}$, not generally for probability measures on $\mathbb{C}$.
\end{example}

\begin{remark} Example \ref{example.Cauchy.dist} shows that the limit in Lemma \ref{lem.Cauchy.moment} can exist (as a complex number) even if $X\notin L^1$.  Indeed, if $X$ has the Cauchy distribution,
\[ z(zG_X(z)-1) = z\left(z\cdot\frac{1}{z+i}-1\right) = \frac{-i}{1+z/i} \]
and so the limit as $|z|\to\infty$ exists and is equal to $-i$.  The Cauchy distribution does not have a finite moment, however.
\end{remark}

\subsection{The Voiculescu Transform and the $R$-transform}\label{sect:R.Voic.trans}

The domain of invertibility of a Cauchy transform will be important in our analysis.  For $\alpha,\beta>0$, the {\bf truncated cone} $\Gamma_{\alpha,\beta}$ (also known as a {\bf Stolz angle}) is
\[ \Gamma_{\alpha,\beta} = \{z=x+iy\in\mathbb{C}_+\colon \alpha y>|x|, |z|>\beta\}. \]
Then $zG_\mu(z)\to 1$ for any path $z\in\Gamma_{\alpha,\beta}$ with $|z|\to\infty$.  It is often convenient to work with the {\bf reciprocal Cauchy transform} mentioned above: $F_\mu = 1/G_{\mu}$; which is an analytic self-map of the upper half plane $\mathbb{C}_+$ satisfying $F_\mu(z)\sim z$ for $z$ large in some $\Gamma_{\alpha,\beta}$.  This asymptotic behavior suggests that $F_\mu$ (and hence $G_\mu$) is an invertible function for sufficiently large $\alpha,\beta$, and this is indeed true.

\begin{proposition}[Proposition 5.4 in \cite{bercovici1993free}] \label{prop.Stolzify} If $\mu$ is any probability measure on $\mathbb{R}$, and $\alpha>0$, then for any $\epsilon\in(0,\alpha)$, there exists some $\beta>0$ such that $F_\mu$ is univalent in $\Gamma_{\alpha,\beta}$ (i.e.\ is one-to-one and has an analytic inverse), with range $F_\mu(\Gamma_{\alpha,\beta})\supseteq \Gamma_{\alpha-\epsilon,\beta(1+\epsilon)}$.
\end{proposition}
\noindent That is: on a sufficiently truncated cone, $F_\mu$ has an analytic inverse whose domain includes a truncated cone arbitrarily close to the domain of $F_\mu$.

On an appropriate truncated cone, $F_{\mu}^{-1}$ is an analytic map, and given that $F_{\mu}(z)\sim z$, the same holds true for $F_{\mu}^{-1}$.  The {\bf Voiculescu transform} $\varphi_\mu$, introduced on page \pageref{page.voiculescu}, is defined, a priori on some truncated cone $\Gamma_{\alpha,\beta}$, by 
\begin{align} \label{def:Voiculescu.trans}\varphi_\mu(z) = F_{\mu}^{-1}(z)-z \quad \mbox{for $z \in \Gamma_{\alpha,\beta}$.}
\end{align}
From Proposition \ref{prop.Stolzify}, the domain of $\varphi_\mu$ can be taken as a truncated cone $\Gamma_{\alpha,\beta_\alpha}$ for any $\alpha>0$ (given large enough $\beta_\alpha$); and the conjugate symmetry implies the vertical reflection in the lower half plane can be included as well.  So $\varphi_\mu$ is defined on the union over $\alpha>0$ of such truncated cones.

The Voiculescu transform satisfies $\varphi_{\mu}(z)/z \to 0$ as $z\to\infty$ in its domain. Moreover, this convergence (and any truncated cone in the domain) is robust under weak convergence of measures.

\begin{proposition}[Proposition 5.7 in \cite{bercovici1993free}] \label{prop.uniform.Stolz} A sequence $\{\mu_n\}_{n=1}^\infty$ of probability measures on $\mathbb{R}$ converges weakly $\mu_n\rightharpoonup\mu$ to some probability measure $\mu$ if and only if there are $\alpha,\beta>0$ such that $\varphi_{\mu_n}$ all converge uniformly on compact subsets of $\Gamma_{\alpha,\beta}$ to a function $\varphi$ and $\varphi_{\mu_n}(z) = o(z)$ uniformly in $n$ as $|z|\to\infty$, $z\in\Gamma_{\alpha,\beta}$.  In this case, $\varphi = \varphi_{\mu}$ in $\Gamma_{\alpha,\beta}$. \end{proposition}

The {\bf $R$-transform} $R_\mu$ is defined in term of the Voiculescu transform as $R_\mu(z) = \varphi_\mu(1/z)$; from \eqref{def:Voiculescu.trans}, the domain of $R_\mu$ contains the reciprocal image of the above-described region, which is an open disk centered at the origin with narrow regions along the real axis possibly removed.  Equation \eqref{def:Voiculescu.trans} defining the Voiculescu transform $\varphi_\mu$ on a precise domain can be rephrased as an implicit equation for $\varphi_\mu$ in terms of the reciprocal Cauchy transform $G_\mu$; correspondingly, $R_\mu$ satisfies an implicit equation in terms of the Cauchy transform $G_\mu$:

\begin{align}\label{Grr:relations}
F_\mu(\varphi_\mu(z)+z) = z, \qquad 
G_X\left(R_X(z)+1/z\right)=z.
\end{align}

The key property of the Voiculescu transform and the $R$-transform is that they linearize free independence:
\[ \varphi_{\mu\boxplus\nu}(z) = \varphi_{\mu}(z) + \varphi_{\nu}(z), \qquad R_{\mu\boxplus\nu}(z) = R_\mu(z)+R_\nu(z) \]
for all $z$ in appropriate domains where the transforms make sense.  This fact was proved for compactly-supported measures, the only kind for which the combinatorial definition (Definition \ref{def:freeness}) of freeness makes sense a priori, in \cite{Voiculescu86}.  Later it was shown \cite{bercovici1993free,Maassen} that for {\em any} probability measures $\mu,\nu$ on $\mathbb{R}$, there is a unique probability measure $\upsilon$ such that $\varphi_{\mu} + \varphi_{\nu} = \varphi_{\upsilon}$ in a truncated cone; this allows the extension of free convolution to all probability measures, defining $\mu\boxplus\nu = \upsilon$.  To be clear: $\mu\boxplus\nu$ is defined by computing $\varphi_\mu$ and $\varphi_\nu$ (each determined completely by its measure, through $F_\mu$ and $F_\nu$), summing, and then computing $F_{\mu\boxplus\nu}^{-1}(z) = \varphi_{\mu}(z) +\varphi_{\nu}(z)+z$ on an appropriate truncated cone.  By Proposition \ref{prop.Stolzify} and the above-mentioned papers, the inverse of this function coincides with reciprocal Cauchy transform $F_{\upsilon}$ of a unique probability measure $\upsilon$ on a closely related truncated cone. This function has an analytic continuation to all of $\mathbb{C}_+$, and then (taking reciprocal) the measure $\upsilon = \mu\boxplus\nu$ can be recovered via the Stieltjes inversion formula.

If $\mu$ is compactly-supported, the domain of analyticity of the $R$-transform $R_\mu$ always includes an open disk centered at the origin.  Its Taylor coefficients at the origin are polynomial functions of the moments of $\mu$, called the {\bf free cumulants}; indeed, the first three are the classical cumulants (mean, variance, skewness), while they disagree from the fourth on.  The $R$-transform is the free analog of the cumulant-generating function (i.e.\ the log-characteristic function). The relationship between $R_\mu$ and $G_\mu$ \eqref{Grr:relations} is more complicated than the simple exponential relationship between the characteristic function $\psi_\mu$ and the cumulant-generating function $\log\psi_\mu$, but it is still possible in many cases to use this relationship for exact computation.

\begin{remark} As with the Cauchy transform and its reciprocal, we abuse notation with the Voiculescu transform and $R$-transform and write $\varphi_X = \varphi_{\mu_X}$ and $R_X = R_{\mu_X}$, for a (potentially unbounded) selfadjoint random variable $X$.  This gives the concise statement that two random variables $X_1,X_2$ are freely independent if and only if $R_{X_1+X_2} = R_{X_1}+R_{X_2}$.
\end{remark}

Like the Cauchy transform (Proposition \ref{prop.Cauchy.robust}), analytic regularity of the $R$-transform shows that free convolution is robust under weak limits.

\begin{proposition}[Proposition 5.7 in \cite{bercovici1993free}] \label{prop.boxplus.robust} 
Let $\mu_n,\mu,\nu_n,\nu$ be probability measures on $\mathbb{R}$, and suppose $\mu_n\rightharpoonup\mu$ and $\nu_n\rightharpoonup\nu$.  Then $\mu_n\boxplus\nu_n \rightharpoonup \mu\boxplus\nu$.
\end{proposition}

\subsection{Free Poisson Distributions} \label{sect:FP}

One of the core laws in classical probability is the Poisson distribution: $\mathrm{Poiss}(\lambda) = e^{-\lambda}\sum_{n=0}^\infty \frac{\lambda^n}{n!}\delta_n$.  It arises in many contexts parallel to the way the normal distribution does, since it is a universal scaling limit (in a different scaling regime from the normal).  In simplest terms, the Poisson limit theorem states that $\mathrm{Poiss}(\lambda)$ approximates an i.i.d.\ sum of $n$ Bernoulli trials with success rate $\lambda/n$; it is the weak limit
\begin{equation} \label{e.Poiss.lim} \mathrm{Poiss}(\lambda) = \lim_{n\to\infty} \left(\left(1-\frac{\lambda}{n}\right)\delta_0 + \frac{\lambda}{n}\delta_1\right)^{\ast n} \end{equation}
(where, in \eqref{e.Poiss.lim}, the $\ast$ in the exponent refers to convolution).  More generally, {\em compound Poisson} distributions are achieved by replacing $\delta_1$ in \eqref{e.Poiss.lim} with any probability measure $\upsilon$, thence denoted the {\em jump distribution}.

{\em Free Poisson} distributions are defined using the same limit procedure as \eqref{e.Poiss.lim} but with free convolution.

\begin{definition} \label{Def.FP} Let $\lambda>0$ and let $\upsilon$ be a probability distribution on $\mathbb{R}$.  The {\bf compound free Poisson} distribution $\mathrm{FP}(\lambda,\upsilon)$ is defined to be the following weak limit of free convolution powers of mixtures of $\delta_0$ and $\upsilon$:
\[ \mathrm{FP}(\lambda,\upsilon):=\lim_{n\to\infty} \left(\left(1-\frac{\lambda}{n}\right)\delta_0 + \frac{\lambda}{n}\upsilon\right)^{\boxplus n}. \]
The parameters $\lambda$ and $\upsilon$ are called the {\em rate} and {\em jump distribution} as in the classical case (although there are no jumps in the free case).
\end{definition}

The $R$-transform can be used to characterize $\mathrm{FP}(\lambda,\upsilon)$; as shown in \cite[Proposition 12.15]{Nica-Speicher-96},
\begin{equation} \label{eq:R.FP} R_{\mathrm{FP}(\lambda,\upsilon)}(z) = \lambda \int_{\mathbb{R}}\frac{t}{1-tz}\,\upsilon(dt). \end{equation}
This shows that $\mathrm{FP}(\lambda,\nu)$ has as many finite moments as $\upsilon$ does.  Compound free Poisson distributions will play an important role in Section \ref{sec:infdiv}.

If $M_j \equaldist \mathrm{FP}(\lambda_j,\upsilon_j)$ are freely independent for $1\le j\le n$, it follows easily from Definition \ref{Def.FP} and $R$-transform calculations that the sum $M_1+\cdots+M_n$ is distributed as $\mathrm{FP}(\lambda,\upsilon)$ where $\lambda = \sum_{j=1}^n \lambda_j$ and $\nu$ is the mixture $\nu = \sum_{j=1}^n \frac{\lambda_j}{\lambda} \upsilon_j$. 
(Nearly the same calculation, using characteristic functions, shows that sums of independent compound free Poisson random variables are compound free Poisson with the sum of the rates and the appropriate mixture of the jump distributions).  Since every measure is approximated by a mixture of finitely-supported measures, in principle any compound free Poisson can be built as a linear combination of free Poisson distributions are given by the two parameter family
\[ \mathrm{FP}(\lambda,\alpha) := \mathrm{FP}(\lambda,\delta_\alpha) = \lim_{n\to\infty} \left(\left(1-\frac{\lambda}{n}\right)\delta_0 + \frac{\lambda}{n}\delta_\alpha\right)^{\boxplus n}. \]
(One could similarly define a two parameter family of classical Poisson distributions, but a simple calculation shows that $\mathrm{Poiss}(\lambda,\delta_\alpha)$ is just the dilation of $\mathrm{Poiss}(\lambda)$ by $\alpha$.  The same does not hold true for free Poissons.)  Equation \ref{eq:R.FP} shows that
\begin{equation} \label{eq:R.FP.a} R_{\mathrm{FP}(\lambda,\alpha)}(z) = \frac{\lambda\alpha}{1-\alpha z}. \end{equation}
From here it follows that the mean of $\mathrm{FP}(\lambda,\alpha)$ is $\lambda\alpha$, and the variance is $\lambda\alpha^2$.  (In general, the free cumulants of $\mathrm{FP}(\lambda,\nu)$ are $\lambda$ times the moments of $\nu$ -- this is the power series interpretation of \eqref{eq:R.FP}.)

From the $R$-transform \eqref{eq:R.FP.a}, we can explicit computation of the Cauchy transform of $\mathrm{FP}(\lambda,\alpha)$.  Indeed, from \eqref{Grr:relations}, we have that $G_{\mathrm{FP}(\lambda,\alpha)}$ is the function inverse of $\frac{\lambda\alpha}{1-\alpha z} + \frac{1}{z}$; in other words
\begin{equation} \label{eq:FP.intermediate} \frac{\lambda\alpha}{1-\alpha G_{\mathrm{FP}(\lambda,\alpha)}(z)}+\frac{1}{G_{\mathrm{FP}(\lambda,\alpha)}(z)} = z. \end{equation}
This is a quadratic equation, whose explicit solution (with the correct asymptotic and mapping behavior for a Cauchy transform) is
\begin{equation} \label{e.FP.Cauchy} G_{\mathrm{FP}(\lambda,\alpha)}(z) = \frac{z+(1-\lambda)\alpha-\sqrt{(z-(1+\lambda)\alpha)^2-4\alpha^2\lambda}}{2\alpha z}. \end{equation}
Here the square root has a branch cut along $[0,\infty)$, i.e.\ $\sqrt{re^{i\theta}} = r^{1/2}e^{i\theta/2}$ with $\theta\in[0,2\pi)$.  Using the Stiletjes inversion formula, this yields a law with a continuous compactly supported density $\varrho_{\lambda,\alpha}$ in the case that $\lambda\ge 1$:
\[ \varrho_{\lambda,\alpha}(x) = \frac{\sqrt{4\lambda\alpha^2-(x-\alpha(1+\lambda))^2}}{2\pi\alpha x}, \qquad   \alpha(1-\sqrt{\lambda})^2 \le x \le \alpha(1+\sqrt{\lambda})^2. \]
When $\lambda<1$, this density has total mass $\lambda$, and here $\mathrm{FP}(\lambda,\alpha)$ is a mixture of this density and a point mass at $0$ of weight $1-\lambda$.  (This follows from the elementary calculation $G_{\mathrm{FP}(\lambda,\alpha)}(z) = G_{\mathrm{FP}(1/\lambda,\alpha)}( z/\lambda)+(1-\lambda)\frac{1}{z}$.)

It is also worth noting that, in the special case $\alpha=1$, this law coincides with the {\bf Marchenko--Pastur} distribution of aspect ratio $\lambda$; i.e.\ it is the large-$n$ limit of the density of singular values of an $m(n)\times n$ matrix all of whose entries are i.i.d.\ with mean $0$ and variance $1/n$, where $m(n)/n\to \lambda$ as $n\to\infty$.

\subsection{Subordination and Free Conditional Expectation}

The $R$-transform and Voiculescu transform of $\mu\boxplus\nu$ are easily described in terms of those of $\mu$ and $\nu$ separately, but the implicit relationship between these and the Cauchy transform makes it harder to compute directly with $G_{\mu\boxplus\nu}$.  There is a function relating them, known as the {\bf subordinator}.

\begin{proposition}[Proposition 4.4 in \cite{Voiculescu-analogues-I}; Theorem 3.1 in \cite{beyonce}; Proposition 7.3 in \cite{BV3}]\label{prop:subordinator.def} Let $\mu,\nu$ be probability measures on $\mathbb{R}$.  There is a unique analytic function $\omega_{\mu,\nu}\colon\mathbb{C}_+\to\mathbb{C}_+$ that satisfies
\[ G_{\mu\boxplus\nu}(z) = G_\mu(\omega_{\mu,\nu}(z))\qmq{or equivalently} F_{\mu\boxplus\nu}(z) = F_\mu(\omega_{\mu,\nu}(z)), \quad z\in\mathbb{C}_+. \]
Moreover
\[ F_{\mu\boxplus\nu}(z) = F_\mu(\omega_{\mu,\nu}(z)) = F_\nu(\omega_{\nu,\mu}(z)) = \omega_{\mu,\nu}(z) + \omega_{\nu,\mu}(z) - z. \]
If $V\equaldist \mu$, $W\equaldist \nu$ are freely independent random variables, we denote $\omega_{V,W} = \omega_{\mu,\nu}$.
\end{proposition}
In principal, these properties follow simply from the defining relation of the subordinator, which can informally be written as $\omega_{\mu,\nu} = F_\mu^{-1}\circ F_{\mu\boxplus\nu}$.  Using Proposition \ref{prop.Stolzify}, this equality rigorously defines $\omega_{\mu,\nu}$ on an appropriate truncated cone, but the key to the use of subordinators is their analyticity on the whole upper half plane.

Like Cauchy transforms (and other analytic transforms defined above), subordinators are robust under weak convergence.  This fact is apparently folklore in the literature; we provide a simple proof outline below.

\begin{proposition}  \label{prop:subordinator}
If $\{\mu_n\}_{n=1}^\infty$ and $\{\nu_n\}_{n=1}^\infty$ are sequences of probability measures with weak limits $\mu_n \rightharpoonup \mu$, $\nu_n\rightharpoonup \nu$, then $\omega_{\mu_n,\nu_n}\to\omega_{\mu,\nu}$ uniformly on compact subsets of $\mathbb{C}_+$.
\end{proposition}

\begin{proof}
By Proposition \ref{prop.uniform.Stolz} there is a fixed truncated cone $\Gamma_{\alpha,\beta}$ on which $\varphi_{\mu_n}$ converges uniformly on compact subsets to $\varphi_\mu$.  By Propositions \ref{prop.Cauchy.robust} and \ref{prop.boxplus.robust}, $F_{\mu_n\boxplus\nu_n}$ converges to $F_{\mu\boxplus\nu}$ uniformly on compact subsets of $\mathbb{C}_+$.  What's more, $\varphi_{\mu_n\boxplus\nu_n}$ converges uniformly on compact subsets of another fixed truncated cone $\Gamma_{\alpha',\beta'}$.  Note that the intersection $\Gamma_{\alpha,\beta}\cap\Gamma_{\alpha',\beta'}$ is another truncated cone. Let $V$ be an open subset with compact closure $\overline{V}\subset\Gamma_{\alpha,\beta}\cap\Gamma_{\alpha',\beta'}$; thus $\varphi_{\mu_n\boxplus\nu_n}$ converges uniformly on $V$.

Appealing to the definition \eqref{def:Voiculescu.trans},
$\varphi_{\mu_n\boxplus\nu_n}(z) = F_{\mu_n\boxplus\nu_n}^{-1}(z)-z$, and thus the preimage of $V$ under $F_{\mu_n\boxplus\nu_n}$ is $F_{\mu_n\boxplus\nu_n}^{-1}(V) = \{\varphi_{\mu_n\boxplus\nu_n}(v)+v\colon v\in V\}$.
As $\varphi_{\mu_n\boxplus\nu_n}$ converges uniformly, there is thus a fixed non-empty open set $W\subseteq\mathbb{C}_+$ with $F_{\mu_n\boxplus\nu_n}^{-1}(V)\supseteq W$ for all $n$. Hence, $F_{\mu_n\boxplus\nu_n}(W)\subseteq V\subseteq\Gamma_{\alpha,\beta}$.

Applying \eqref{def:Voiculescu.trans} once more,
\[ \omega_{\mu_n,\nu_n} = F_{\mu_n}^{-1}\circ F_{\mu_n\boxplus\nu_n} = \varphi_{\mu_n}\circ F_{\mu_n\boxplus\nu_n} + F_{\mu_n\boxplus\nu_n} \]
is defined and analytic on $W$ since $F_{\mu_n\boxplus\nu_n}$ maps $W$ into $\Gamma_{\alpha,\beta}$.  The composition thus converges to $\varphi_\mu\circ F_{\mu\boxplus\nu} + F_{\mu\boxplus\nu} = \omega_{\mu,\nu}$ uniformly on compact subsets of the open subset $W$ of the upper half plane.  Since $\omega_{\mu_n,\nu_n}$ and $\omega_{\mu,\nu}$ are all analytic in the full upper half plane, the proposition now follows from Montel's theorem.
\end{proof}

Our main use of subordinators is for computations of {\bf conditional expectation} in a free probabilistic context.  If $(\mathcal{A},E)$ is a $W^\ast$-probability space, and $\mathcal{B}\subseteq\mathcal{A}$ is a von Neumann subalgebra, then there is a ``conditioning" map from $\mathcal{A}$ to $\mathcal{B}$. As in the classical case, it can be constructed as an $L^2$-projection: since the starting point is the algebra $\mathcal{A}$ (analogous to $L^\infty$) which is contained in $L^2(\mathcal{A},E)$, the orthogonal projection maps $\mathcal{A}$ into $L^2(\mathcal{B},E)$; in point of fact, the image of the projection of elements in the algebra $\mathcal{A}$ is contained in the algebra $\mathcal{B}\subset L^2(\mathcal{B},E)$.  (This requires $E$ to be a faithful normal trace.)  The projection therefore defines a linear map $E[\,\cdot\,|\mathcal{B}]\colon\mathcal{A}\to\mathcal{B}$ which extends to a contraction $L^p(\mathcal{A},E)\to L^p(\mathcal{B},E)$ for all $p\ge 1$, which satisfies the following two key properties:
\begin{enumerate}
    \item If $X_1,X_2\in\mathcal{B}$ and $Y\in\mathcal{A}$, then $E[X_1YX_2|\mathcal{B}] = X_1\,E[Y|\mathcal{B}]\,X_2$.

    \item If $\mathcal{B}_1\subseteq\mathcal{B}_2\subseteq{A}$, then for any $X\in\mathcal{A}$, $E[E[X|\mathcal{B}_2]|\mathcal{B}_1] = E[X|\mathcal{B}_1]$.

\end{enumerate}
If $X$ is freely independent from $\mathcal{B}$, then $E[X|\mathcal{B}] = E[X]$ is constant; this mirrors the same result which holds classically if $X$ is independent from $\mathcal{B}$.  However, things diverge significantly between the classical and free worlds, as we'll see shortly.

If $X$ is a random variable in $(\mathcal{A},E)$, and if $\mathcal{B} = W^\ast(X)$ is the von Neumann subalgebra generated by $X$ (in the selfadjoint case this can be interpreted as the set of all measurable bounded functions $f$ of $X$, $f(X)$), then we typically denote $E[\,\cdot\,|\mathcal{B}] = E[\,\cdot\,|X]$.  This is relevant for the present work in the case of computing $E[f(X+Y)|X]$ where $f$ is a smooth bounded test function and $X,Y$ are freely independent.  If $X,Y$ were {\em classically} independent, this is simply equal to $g(X)$ where for $t\in\mathbb{R}$ $g(t) = E[f(t+Y)]$; but that is not true in the free world. Instead, we have the following important result.

\begin{proposition}[Theorem 3.1 in \cite{beyonce}] Let $X$ and $Y$ be freely independent selfadjoint random variables with laws $\mu$ and $\nu$, and let $\omega_{\mu,\nu}$ denote the subordinator given by Proposition \ref{prop:subordinator.def}.  Let $k_{\mu,\nu}$ be the Feller--Markov kernel whose Cauchy transform is given by
\begin{equation} \label{eq:conditional.resolvant} \int_{\mathbb{R}} \frac{k_{\mu,\nu}(t,du)}{z-u} = \frac{1}{\omega_{\mu,\nu}(z)-t}, \quad t\in\mathbb{R}. \end{equation}
Then for any bounded Borel measurable function $f$,
\begin{align}\label{eq:cond.fmunu}
E[f(X+Y)|X]=f_{\mu,\nu}(X)
\qmq{
where
}
f_{\mu,\nu}(t)=\int f(u)\,k_{\mu,\nu}(t,du).
\end{align}
\end{proposition}
In particular: fixing $z\in\mathbb{C}_+$, the function $f_z(t) = 1/(z-t)$ is bounded and continuous; \eqref{eq:conditional.resolvant} and \eqref{eq:cond.fmunu} together then say
\begin{align}\label{eq:cond.resolve}
E[(z-X-Y)^{-1}|X] = (\omega_{\mu,\nu}(z)-X)^{-1}.
\end{align}

\subsection{Free Infinite Divisibility}\label{sect.background.inf.div}

As introduced already in the Introduction, a probability measure $\mu$ on $\mathbb{R}$ is called {\bf freely infinitely divisible} if, for each $n\in\mathbb{N}$, there is a probability measure $\mu_n$ such that $\mu = \mu_n\boxplus\cdots\boxplus\mu_n = \mu_n^{\boxplus n}$.  In terms of random variables: (the law of) $X$ is infinitely divisible if, for each $n$, there are freely independent identically distributed random variables $\{X_{j,n}\}_{1\le j\le n}$ so that $X = X_{1,n}+\cdots+X_{n,n}$.  A standard triangular arrays argument, analogous to the classical approach, shows that it is equivalent to ask only that such $X_{j,n}$ exist where the sum $X_{1,n}+\cdots+X_{n,n}$ {\em converges weakly} to $X$.  Prominent examples of freely independent laws include semicircular distributions (since the sum of two free semicricular random variables is semicircular), and the compound free Poisson distributions $\mathrm{FP}(\lambda,\nu)$ of Definition \ref{Def.FP} (this follows directly from their definition: the $X_{j,n}$ can be taken to have distribution $(1-\frac{\lambda}{n})\delta_0 + \frac{\lambda}{n}\nu$).

Characterizing laws in terms of their $R$-transform, $\mu$ is freely infinitely divisible if, for each $n\in\mathbb{N}$, $\frac{1}{n}R_\mu$ is the $R$-transform of a probability measure.  By taking convolution powers, this means that $tR_\mu$ has the form $R_{\mu_t}$ for some probability measure $\mu_t$ for any rational $t\ge 0$; a continuity result \cite[Lemma 5.9]{bercovici1993free} then shows that this holds for all real $t\ge 0$ as well.  Thus, freely infinitely divisible measures correspond to {\bf free convolution semigroups} $\mu_t$ satisfying $\mu_{t+s} = \mu_t\boxplus\mu_s$ for $s,t\ge0$, and $\mu_0=\delta_0$ (the original measure is then identified as $\mu = \mu_1$).  A major result identifying such measures/semigroups is \cite[Theorem 5.10]{bercovici1993free}, which states that a measure $\mu$ is freely infinitely divisible if and only if $\varphi_\mu$ (and hence $R_\mu$) has an analytic continuation to all of $\mathbb{C}_+$ (which may take values in $\mathbb{C}_-\cup\mathbb{R}$).  That same theorem also proves the free L\'evy--Khintchine formulas \eqref{LK.via.nica} and \eqref{BV2.LK-formula}.

\ignore{
\begin{remark}\label{remark:free-conv-semigroup} Curiously, in the free world, {\em every} probability measure $\mu$ is in a {\em partial} free convolution semigroup $\{\mu_t\colon t\ge 1\}$ as $\mu=\mu_1$.  In other words, for $t\ge 1$, $tR_\mu$ is always the $R$-transform of a probability measure $\mu_t$.  This was fully proved in \cite[Theorem 2.5(1)]{BelBer} using subordinator methods, generalizing earlier results \cite{Bercovici-Voiculescu-95-PTRF,Nica-Speicher-96} in the compactly supported case.  This is one of many places where free probability has intriguingly different properties from classical probability.
\end{remark}
}

One of the major theorems in free probability, connecting it to classical probability, is the {\em Bercovici--Pata bijection} $\Lambda$ introduced in \cite{BercoviciPata1999}.  Denote by $\mathrm{ID}_\ast(\mathbb{R})$ the (classically) infinitely divisible distributions on $\mathbb{R}$, and by $\mathrm{ID}_\boxplus(\mathbb{R})$ the freely infinitely divisible distributions on $\mathbb{R}$.  If $\mu\in \mathrm{ID}_\ast$ and has $2$ finite moments, the L\'evy--Khintchine formula \eqref{eq:logphi.zero} yields a probability measure $\nu$ which, together with the mean $m$ and variance $\sigma^2$, characterizes $\mu$.  The triple $(m,\sigma^2,\nu$) is {\em also} associated to a unique probability measure $\Lambda(\mu)$ with $2$ finite moments, via \eqref{LK.via.nica}.  The mapping $\mu\mapsto\Lambda(\mu)$ sets up a bijection between $\mathrm{ID}_\ast(\mathbb{R})\cap L^2$ and $\mathrm{ID}_\boxplus(\mathbb{R})\cap L^2$.  In fact, it extends all the way to a bijection $\Lambda\colon\mathrm{ID}_\ast(\mathbb{R})\to\mathrm{ID}_\boxplus(\mathbb{R})$.

\begin{theorem} \label{thm.BP.Bijection} Let $\mu$ be a probability measure on $\mathbb{R}$.
\begin{enumerate}[wide = 0pt,label=\alph*)]
    \item \label{thm:BP.a} $\mu \in \mathrm{ID}_\ast(\mathbb{R})$ if and only if there is a constant $\gamma\in\mathbb{R}$ and a finite positive Borel measure $\varsigma$ on $\mathbb{R}$ so that the characteristic function $\psi_\mu(\xi) = \int e^{i\xi t}\mu(dt)$ satisfies
    \begin{equation} \label{eq:LK.nomoments.*}
    \log\psi_\mu(\xi) = i\gamma \xi + \int_{\mathbb{R}} \frac{\exp(i\xi t)-1-i\xi t}{t^2}\,(t^2+1)\,\varsigma(dt).
    \end{equation}
    \item \label{thm:BP.b} $\mu \in \mathrm{ID}_\boxplus(\mathbb{R})$ if and only if there is a constant $\gamma\in\mathbb{R}$ and a finite positive Borel measure $\varsigma$ on $\mathbb{R}$ so that the $R$-transform $R_\mu$ satisfies
    \begin{equation} \label{BV2.LK-formula}    
    R_\mu(z) = \gamma + \int_{\mathbb{R}} \frac{z+t}{1-tz}\,\varsigma(dt).
    \end{equation}
    \item \label{thm:BP.c} Given $\mu\in\mathrm{ID}_\ast(\mathbb{R})$, let $(\gamma,\varsigma)$ be the constant and finite measure associated to $\mu$ by \eqref{eq:LK.nomoments.*}.  There is a unique probability measure $\Lambda(\mu)\in\mathrm{ID}_{\boxplus}(\mathbb{R})$ with this same pair $(\gamma,\varsigma)$ characterizing $R_{\Lambda(\mu)}$ in \eqref{BV2.LK-formula}.  The map $\Lambda\colon\mathrm{ID}_\ast(\mathbb{R})\to\mathrm{ID}_\boxplus(\mathbb{R})$ is a bijection, and satisfies $\Lambda(\mu_1\ast\mu_2) = \Lambda(\mu_1)\boxplus\Lambda(\mu_2)$.
\end{enumerate}
\end{theorem}
Item \ref{thm:BP.a} is classical, proved for example in the textbook \cite[Section 18]{GnedKolm}.  The free version in item \ref{thm:BP.b} was proved in \cite{bercovici1993free}.
\begin{remark} \label{remark.nusigma} Noting that $\frac{z}{1-tz}\cdot \frac{t}{t^2+1} = \frac{z+t}{1-tz}+\frac{t}{t^2+1}$, in the case that $\varsigma$ has a finte moment, \eqref{BV2.LK-formula} can be rewritten as
\begin{equation} \label{BV2.LK-formula.2finite.moments}
R_\mu(z) = \gamma' + \int_{\mathbb{R}} \frac{z}{1-tz}(t^2+1)\,\varsigma(dt)
\end{equation}
where $\gamma' = \gamma + \int t\,\varsigma(dt)$.  Further, if $\varsigma$ has two finite moments, then $(t^2+1)\,\varsigma(dt)$ is a finite measure, and expanding the right-hand side of \eqref{BV2.LK-formula.2finite.moments} as a power series in $z$ shows that its mass is $\sigma^2 = \mathrm{Var}(\mu)$.  Thus, letting $\nu(dt) = \frac{1}{\sigma^2}(t^2+1)\,\varsigma(dt)$, we see that \eqref{BV2.LK-formula.2finite.moments} recovers \eqref{LK.via.nica}, where $\gamma'=m$.  An analogous analysis comparing \eqref{eq:LK.nomoments.*} and \eqref{eq:logphi.zero} shows that $\Lambda$ defined in Theorem \ref{thm.BP.Bijection}(\ref{thm:BP.c} indeed maps $\mathrm{ID}_\ast(\mathbb{R})\cap L^2$ onto $\mathrm{ID}_\boxplus(\mathbb{R})\cap L^2$ bijectively as described above.
\end{remark}

\begin{remark} From the convolution interchange property of Theorem \ref{thm.BP.Bijection}(\ref{thm:BP.c} and from direct power-series expansions, one sees that $\Lambda(\mu)$ always has the same number of finite moments as $\mu$, and up to that order, the classical cumulants of $\mu$ are equal to the free cumulants of $\Lambda(\mu)$. \end{remark}

For example: the normal distribution $\mathcal{N}(m,\sigma^2)$ has L\'evy measure $\nu = \delta_0$; the image $\Lambda(\mathcal{N}(m,\sigma^2))$ under the Bercovici--Pata bijection is then the law whose $R$-transform satisfies $R(z) = m+\sigma^2 z$, which is the semicircular law $\mathcal{S}(m,\sigma^2)$ with density
\begin{equation} \label{eq:semicircle}
\frac{1}{2\pi\sigma} \sqrt{4\sigma^2 - (x-m)^2}\,\mathbf{1}_{|x-m|\le 2\sigma}.
\end{equation}
Similarly, if $\mu = \mathrm{Poiss}(\lambda,\upsilon)$ is a compound Poisson distribution, then $\Lambda(\mu) = \mathrm{FP}(\lambda,\upsilon)$ is the compound free Poisson with the same rate $\lambda$ and the same jump distribution $\upsilon$.

\subsection{Positive Free Infinite Divisibility}

Classically, infinite divisibility and positivity intertwine well.  If $\mu\in\mathrm{ID}_\ast(\mathbb{R}_+)$,i.e.\ $\mu$ is infinitely divisible and is supported on $[0,\infty)$, then all convolution roots of $\mu$ are also supported on $[0,\infty)$.  Indeed, if $X = X_{1,n}+\cdots+X_{n,n}$ are i.i.d.\ and $P(X_{j,n}<0)=p>0$ then $P(X<0)\ge P(X_{j,n}<0\text{ for }1\le j\le n) = p^n>0$.  This elementary argument fails in the free world, however.  Take, for example, any semicircular distribution $\mathcal{S}(m,\sigma^2)$ with $m\ge 2\sigma$; from \eqref{eq:semicircle} this is supported in $[0,\infty)$.  However, the $n$th $\boxplus$-root is $\mathcal{S}(m/n,\sigma^2/n)$, and no matter how large $m-2\sigma>0$ is, $\frac{m}{n}-2\frac{\sigma}{\sqrt{n}}$ is negative for all large $n$.  That is, referring to Definition \ref{def:pfid}, this example shows that there are measures that are positive and freely infinitely divisible but not positive{\em ly} freely infinitely divisible.

We denote the class of positively freely infinitely divisible measures by $\mathrm{ID}^+_\boxplus(\mathbb{R}_+)$.
\ignore{
\begin{definition} A measure $\mu\in \mathrm{ID}_\boxplus(\mathbb{R}_+)$ is called {\bf positively freely infinitely divisible} if, for all $t>0$, $\mu^{\boxplus t}$ is supported on $[0,\infty)$.  We refer to this property of a measure as {\bf positive free infinite divisibility}.  The class of such measures is denoted $\mathrm{ID}^+_\boxplus(\mathbb{R}_+)$.
\end{definition}
}
This notion was introduced first in \cite{perez2012free} and explored further in \cite{arizmendi2011law}, where such measures were called {\em free regular} and the class of them was denoted $I^{\boxplus}_{r+}$ rather than our $\mathrm{ID}^+_\boxplus(\mathbb{R}_+)$.  The authors of \cite{arizmendi2011law} were interested in these measures because (as they show) they are precisely the laws of all free L\'evy processes with positive increments.

The following theorem is proved as part of \cite[Theorem 4.2]{arizmendi2011law}.

\begin{theorem} Let $\mu$ be a probability measure on $\mathbb{R}$.  Then $\mu$ is positively freely infinitely divisible if and only if there is a measure $\nu\in \mathrm{ID}_\ast(\mathbb{R}_+)$ with $\Lambda(\eta)=\mu$.  I.e.\ $\mathrm{ID}_\boxplus^+(\mathbb{R}_+) = \Lambda(\mathrm{ID}_\ast(\mathbb{R}_+))$.  Moreover, in this case, the L\'evy--Khintchine measure $\varsigma$ for $\mu$, cf.\ \eqref{BV2.LK-formula}, is supported on $[0,\infty)$.
\end{theorem}
It is straightforward to see, using the homomorphism property of the Bercovici--Pata bijection $\Lambda$ in Theorem \ref{thm.BP.Bijection}(\ref{thm:BP.c}, and the fact the classically infinitely divisible distributions supported on $[0,\infty)$ have all convolution roots supported on $[0,\infty)$, that $\Lambda(\mathrm{ID}_\ast(\mathbb{R}_+)) \subseteq \mathrm{ID}_\boxplus^+(\mathbb{R}_+)$; the reverse containment is one of the main theorems of \cite{arizmendi2011law}.

In the $L^2$ setting most relevant to this paper, the free L\'evy--Khintchine measure $\nu$ of \eqref{LK.via.nica} and Theorem \ref{main theorem 2}(\ref{thm2.b} is related to $\varsigma$ by $\nu(dt) = \frac{1}{\sigma^2}(t^2+1)\varsigma(dt)$ where $\sigma^2$ is the variance of $\mu$, cf.\ Remark \ref{remark.nusigma}.  Hence, we have the following important corollary.

\begin{corollary} \label{cor.V0>0} Let $X$ be an $L^2$ positively freely infinitely divisible random variable.  Then its Voiculescu transform $\varphi_X$ has the form
\[ \varphi_X(z) = m+\sigma^2 G_{V_0}(z) \]
where $m = E[X]$, $\sigma^2 = \mathrm{Var}[X]$, and $V_0\ge 0$; i.e.\ the free L\'evy--Khintchine measure is supported on $[0,\infty)$.

\end{corollary}

\section{Existence and properties of the Free Zero Bias and other Transformations}\label{sec:exist}
In Section \ref{subsec:exist} we first prove the existence of a distributional transformation based on the fact that the ``geometric mean" of two Cauchy transforms of probability measures is again the Cauchy transform of a probability measure, cf.\ Lemma \ref{lem:root.Geometric.Mean}. In Section \ref{sec:ElGordo} we use a special case of Lemma \ref{lem:root.Geometric.Mean} to produce a new transformation we call the {\em El Gordo transform}; we also review the square bias transformation and provide some of its properties.  Combining these transformations in Section \ref{sect:freezerobias.construction} then yields the construction of the free zero bias putatively defined by \eqref{char.freezb}, and the proof that its unique fixed point is the semi-circle distributions.
Finally, in Section \ref{subsec:prop.examples} we consider a few examples. 

\subsection{Geometric Mean Transformation}\label{subsec:exist}

Complex square roots will come up frequently in this section and beyond, so we begin by setting notation.

\begin{notation} \label{notat:sqrt} For $\zeta\in\mathbb{C}$, the (principal) {\bf square root} $\sqrt{\zeta}$ is defined uniquely in polar coordinates: writing $\zeta = re^{i\theta}$ for $r\ge 0$ and $\theta\in[0,2\pi)$, $\sqrt{\zeta} = r^{1/2} e^{i\theta/2}$.  The function $\zeta\mapsto\sqrt{\zeta}$ is holomorphic on $\mathbb{C}\setminus[0,\infty)$, with image contained in $\mathbb{C}_+$.

Note that if $\zeta\in\mathbb{C}$ and $y>0$ then $\sqrt{y\zeta} = \sqrt{y}\sqrt{\zeta}$; but if $\zeta,\xi\in\mathbb{C}_-$ then $\sqrt{\zeta\xi} = -\sqrt{\zeta}\sqrt{\xi}$.  Similarly, if $z\in\mathbb{C}\setminus[0,\infty)$ then $\sqrt{1/z} = -1/\sqrt{z}$.  Finally: note that if $\zeta,\xi\in\mathbb{C}_-$, then $\zeta\xi\in\mathbb{C}\setminus[0,\infty)$.

\end{notation}

With the square root in hand, the following elementary lemma yields a (surprisingly) new operation on probability distributions.

\begin{lemma}\label{lem:root.Geometric.Mean}
Given any two random variables $X$ and $Y$, there is a unique probability measure (which we denote as the law of a new random variable $X\bflat Y$) whose Cauchy transform satisfies
\begin{align} \label{def:GXsharpY}
G_{X \bflat Y}(z)= -\sqrt{G_X(z)G_Y(z)} = \sqrt{G_X(z)}\sqrt{G_Y(z)}.
\end{align}
I.e.\ $G_{X\bflat Y}$ is the ``geometric mean'' of the Cauchy transforms of $X$ and $Y$.  The map $(\mathcal{L}(X),\mathcal{L}(Y)) \rightarrow \mathcal{L}(X \bflat Y)$ is bi-continuous in the sense that $X_n \bflat Y_n \rightharpoonup X \bflat Y$ whenever $X_n \rightharpoonup X$ and $Y_n \rightharpoonup Y$, as $n \rightarrow \infty$.
\end{lemma}

\begin{proof} $G_X$ and $G_Y$ are analytic maps from $\mathbb{C}_+$ to $\mathbb{C}_-$, hence for any $z\in\mathbb{C}_+$ the product $z\mapsto G_X(z)G_Y(z)$ defines an analytic function taking values in $\mathbb{C}\setminus \mathbb{R}$. Hence $G_{X\bflat Y}(z) = -\sqrt{G_X(z)G_Y(z)}$ is analytic, and takes values in the lower half plane $\mathbb{C}_-$.  That it is equal to $\sqrt{G_X(z)}\sqrt{G_Y(z)}$ follows from the second paragraph in Notation \ref{notat:sqrt}.

Hence, to show that $G_{X\bflat Y}$ is the Cauchy transform of a probability measure, it suffices to verify that $iy G_{X\bflat Y}(iy)\to 1$ as $y\uparrow\infty$.  For this, we appeal to the second representation in \eqref{def:GXsharpY}.  Since $G_X$ and $G_Y$ are Cauchy transforms of probability measures, $iyG_X(iy)\to 1$ and $iyG_Y(iy)\to 1$, and therefore $yG_X(iy)\to -i$ and $yG_Y(iy)\to -i$ as $y\uparrow\infty$.  Since $y\mapsto yG_X(iy)$ is continuous and takes values in $\mathbb{C}\setminus[0,\infty)$, the function $y\mapsto\sqrt{yG_X(iy)}$ is also continuous, and hence $\sqrt{yG_X(iy)}\to \sqrt{-i} = e^{\frac{3 \pi}{4}i}$; similarly $\sqrt{yG_Y(iy)}\to e^{\frac{3 \pi}{4}i}$.  Hence
\[ yG_{X\bflat Y}(iy) = \sqrt{yG_X(iy)}\sqrt{yG_Y(iy)} \to (e^{\frac{3 \pi}{4}i})^2 = -i \quad\text{as }y\uparrow\infty \]
and we conclude that $iyG_{X\bflat Y}(iy)\to i(-i) = 1$, showing that $G_{X\bflat Y}$ is the Cauchy transform of a probability measure as claimed, and uniqueness follows from Stieltjes inversion (cf.\ Proposition \ref{prop.Cauchy.robust}.

For the continuity claim: if $X_n\rightharpoonup X$ and $Y_n\rightharpoonup Y$ then $G_{X_n}\to G_X$ and $G_{Y_n}\to G_Y$ uniformly on compact subsets of $\mathbb{C}_+$ by Proposition \ref{prop.Cauchy.robust}(\ref{robust.1}).  Since $G_{X_n}$ and $G_{Y_n}$ take values in $\mathbb{C}_-$ where the square root is continuous, it then follows that $G_{X_n\bflat Y_n} = \sqrt{G_{X_n}}\sqrt{G_{Y_n}}$ converges uniformly on compact subsets to $\sqrt{G_X}\sqrt{G_Y} = G_{X\bflat Y}$.  The convergence $X_n\bflat Y_n\rightharpoonup X\bflat Y$ now follows from Proposition \ref{prop.Cauchy.robust}(\ref{robust.2}).
\end{proof}

\ignore{

\begin{remark}
Lemma \ref{lem:root.Geometric.Mean} has a (surprising) consequence regarding moment generating functions of the form
$$
M_X(z)=\sum_{k=1}^\infty z^k m_k
$$
for random variables $X$ for which $m=E[X^k]$ exists for all $k \ge 0$. Indeed, recognizing the relations 
$$
M_X(z)=\frac{1}{z}G_X
\left(\frac{1}{z}\right) \qmq{and} G_X(z)=\frac{1}{z}M_X
\left(\frac{1}{z}\right)
$$
we see that we may rephrase \eqref{lem:root.Geometric.Mean} to obtain
$$
M_{X \bflat Y}(z)=\frac{1}{z}G_{X \bflat Y}
\left(\frac{1}{z}\right)= \sqrt{\frac{1}{z} G_X \left( \frac{1}{z}\right)} \sqrt{\frac{1}{z} G_Y \left( \frac{1}{z}\right)} = \sqrt{M_X(z)} \sqrt{M_Y(z)}.
$$
That is, that the geometric mean of any two moment generating functions is again a moment generating function. In the case where $Y=\delta_0$, as $G_Y(z)=1/z$ we have $M_Y(z)=G_Y(1/z)/z=1$, implying
$$
M_{X \bflat Y}(z) = \sqrt{M_X(z)}.
$$
That is, the square root of any moment generating function is again a moment generating function. 
\end{remark}
}

\begin{example} \label{ex:XbflatY.arcsine}
\begin{enumerate}[wide = 0pt,label=\alph*)]

\item For any random variable $X$,
\[ G_{X\bflat X}(z) = \sqrt{G_X(z)}\sqrt{G_X(z)} = G_X(z) \]
i.e.\ $X\flat X \displaystyle{\mathop{=}^d} X$.

\item \label{ex:XbflatY.arcsine.b} Consider the case that $X$ and $Y$ are constants; e.g.\ $X=1$ and $Y=-1$.  Then \begin{align}\label{eq:double.flat}
G_{X \bflat Y}(z) = -\sqrt{G_X(z)G_Y(z)} = -\sqrt{\frac{1}{z-1}\frac{1}{z+1}}=\frac{1}{\sqrt{z^2-1}}. 
\end{align}
(The last equality comes from Notation \ref{notat:sqrt}, since $z^2-1\in\mathbb{C}\setminus[0,\infty)$ when $z\in\mathbb{C}_+$.)  Now employing the Stieltjes inversion formula (Proposition \ref{prop.Cauchy.robust}), the law $X\bflat Y$ can be recovered, and possesses a density:
\begin{align} \label{eq:Cinversion}
\varrho_{X\bflat Y}(x)=-\frac{1}{\pi} \lim_{y \downarrow 0}{\rm Im}\,G_\nu(x+iy) = \frac{1}{\pi\sqrt{1-x^2}}{\bf 1}_{[-1,1]}(x).
\end{align}
That is to say: the law of $X\bflat Y$ in \eqref{eq:Cinversion} is the arcsine distribution 
on $[-1,1]$.  More generally, letting $X$ and $Y$ be point masses at $a<b$ respectively, scaling and translating appropriately we obtain
\begin{align*}
G_{a\flat b}(z) = \frac{1}{\sqrt{(z-a)(z-b)}} \qmq{and} \varrho_{a\flat b}(x)=\frac{1}{\pi \sqrt{(x-a)(b-x)}}{\bf 1}_{[a,b]}(x)
\end{align*}
which is the arcsine distribution on $[a,b]$.

\item We can combine any random variable $X$ with a point mass $a$ this way, producing a new distribution $\mathcal{L}(a\bflat X)$ whose Cauchy transform is $G_{a\bflat X}(z) = -\sqrt{G_X(z)/(z-a)}$.  This construction, with $a=0$, will be very important through the remainder of the paper, so we address it in the next result.

\end{enumerate}

\end{example}

\subsection{The Size Bias $s$, Square Bias $\Box$ and El Gordo $\bflat$ Transformations}\label{sec:ElGordo}

In this section, we introduce two old (size bias $X\mapsto X^s$, square bias $X\mapsto X^\Box$) and one new (El Gordo $X\mapsto X^\bflat$) transformation on probability measures, focusing on their actions on Cauchy transforms.  We begin with El Gordo.

\begin{lemma}\label{lem:bflat}
For any probability measure $\mathcal{L}(X)$, 
\begin{align} \label{G.sharp}
G_{X^\bflat}(z) = -\sqrt{\frac{1}{z}G_X(z)}
\end{align}
is the Cauchy transform of a probability distribution.  We call $\mathcal{L}(X^\flat)$ the {\bf El Gordo transform} of $\mathcal{L}(X)$.  This transformation on the set of all probability measures is continuous in the topology of convergence in distribution, has point mass at zero as its unique fixed point, satisfies $(\alpha X)^\bflat\equaldist \alpha X^\bflat$ for all $\alpha \in \mathbb{R}$, and is injective but not surjective.
\end{lemma}

\begin{proof}
As $G_{\delta_0}(z)=1/z$, we find that
$X^\bflat\equaldist X\bflat \delta_0$, and the existence and continuity properties follow by Lemma \ref{lem:root.Geometric.Mean}. 
That $\delta_0$ is the map's unique fixed point follows immediately by solving the equality $G_X(z)=-\sqrt{G_X(z)/z}$, implying either $G_X(z) = 1/z$ or $G_X\equiv 0$ which is not the Cauchy transform of a probability measure.

Next, if $\alpha=0$ we have that $(\alpha X)^\bflat\equaldist \alpha X^\bflat$ as the left hand side is zero with probability $1$ since the El Gordo transform takes $\delta_0$ to $\delta_0$. When $\alpha \not =0$ then 
\begin{multline*}
G_{(\alpha X)^\bflat}^2(z)=\frac{1}{z}E\left[ 
\frac{1}{z-\alpha X}
\right]= \frac{1}{\alpha z}E\left[ 
\frac{1}{z/\alpha-X}\right]=\frac{1}{\alpha z}G_X\left(\frac{z}{\alpha}\right)\\
=\frac{1}{\alpha^2}\left(\frac{\alpha}{z}G_X\left(\frac{z}{\alpha}\right)\right)=\frac{1}{\alpha^2}G_{X^\bflat}^2\left( \frac{z}{\alpha}\right)=G_{\alpha X^\bflat}^2(z).
\end{multline*}
Taking square roots gives $G_{(\alpha X)^\bflat} = G_{\alpha X^\bflat}$, and this implies $(\alpha X)^\bflat\equaldist \alpha X^\bflat$ by the Stieltjes inversion formula, cf.\ Proposition \ref{prop.Cauchy.robust}.

To show that the El Gordo transform is injective, when the distributions of $X^\bflat$ and $Y^\bflat$ are equal, then 
\[ -\sqrt{\frac{1}{z}G_X(z)}=-\sqrt{\frac{1}{z}G_Y(z)} \quad \mbox{for all $z \in \mathbb{C}_+$,} \]
and squaring  both sized and multiplying by $z$ implies $G_X=G_Y$, and yielding $X\equaldist Y$ by the Stieltjes inversion formula.

To show that the map is not surjective, we show that no point mass, other than zero, is in the range of the El Gordo transform.  Indeed, if $\mathcal{L}(Y^\bflat) = \delta_p$, meaning that $-\sqrt{G_Y(z)/z} = G_{Y^\bflat}(z) = G_{\delta_p}(z) = (z-p)^{-1}$, it follows that
$$
G_Y(z)=zG_{Y^\bflat}^2(z) = \frac{z}{(z-p)^2}= \frac{z(\overline{z}-p)^2}{|z-p|^4}.
$$
Note, however, that $\mathrm{Im}\,G_Y(z)$ is then a positive multiple of
\begin{align*} \mathrm{Im}\left(z(\bar{z}-p)^2\right)
= \mathrm{Im}\left(z\bar{z}^2 -2pz\bar{z}+p^2z\right)
&= |z|^2\,\mathrm{Im}\,\bar{z}+p^2\mathrm{Im}\,z \\
&= (p^2-|z|^2)\,\mathrm{Im}\,z.
\end{align*}
Hence, unless $p=0$, the above function maps the half-disk of radius $p$ in the upper half plane back into the upper half plane, and so the putative $G_Y$ is not the Cauchy transform of a probability measure.
\end{proof}

\medskip

We now recall the square bias transform given in \eqref{def:Xbox}, and make the convention that if $E[X^2]=0$ then $\mathcal{L}(X^\Box)=\delta_0$, a point mass at zero. In Theorem \ref{thm:Compound.Poisson} and Proposition \ref{prop:ForAnyV} we will also have need to consider the following pseudo-inverse of the square bias transforms. 

\begin{definition} \label{def:X.Box.inverse}
If $\mathcal{L}(X)=\delta_0$, let $\mathcal{L}(X^{-\Box})=\delta_0$. If $E[X^{-2}]<\infty$ then let the distribution $\mathcal{L}(X^{-\Box})$ have Radon-Nikodym derivative with respect to that of $X$ given by
\begin{align*}
\frac{dP^{-\Box}}{dP} = \frac{X^{-2}}{E[X^{-2}]}.
\end{align*}
\end{definition}

The next lemma collects many useful properties of the square bias.

\begin{lemma}\label{lem:X.Xneg0.same}
Let $X$ and $Y$ be $L^2$ random variables, and let $Z$ be any random variable.
\begin{enumerate}[wide = 0pt,label=\alph*)]    
\item Let $\alpha \in [0,1]$ and 
\begin{align*}
\mathcal{L}(W)=\alpha \mathcal{L}(X)+(1-\alpha) \mathcal{L}(Y).
\end{align*}
Then $E[W^2]=\alpha E[X^2]+(1-\alpha)E[Y^2]$, and when $E[W^2]>0$ then 
\begin{align} \label{eq:WBox.as.mixture}
\mathcal{L}(W^\Box)=\frac{\alpha E[X^2]}{E[W^2]} \mathcal{L}(X^\Box)+\frac{(1-\alpha)E[Y^2]}{E[W^2]} \mathcal{L}(Y^\Box).
\end{align}
\item \label{lem:sqbiaa-part2} When $E[X^2]>0$, $X$ and $Y$ 
satisfy $E[X^2]\ge E[Y^2]$ and
$\alpha=E[Y^2]/E[X^2]$, 
then
$$
\mathcal{L}(X^\Box)=\mathcal{L}(Y^\Box) \,\,
\mbox{if and only if}\,\,
\mathcal{L}(Y)=(1-\alpha)\delta_0 + \alpha\mathcal{L}(X).
$$

\item \label{lem:sqbiaa-part3} If $X_n$ is a sequence of square integrable random variables such that $X_n \rightharpoonup X$ and $E[X_n^2] \rightarrow E[X^2] \in (0,\infty)$, then
$
X_n^\Box \rightharpoonup X^\Box.
$
The restriction that the second moment of the limit be positive is necessary. 

\item \label{lem.Box.G} When $E[X^2]>0$ the Cauchy transform of $X^\Box$ is given by 
$$
G_{X^\Box}(z)= \frac{1}{E[X^2]}\left( z^2G_X(z)-E[X]-z\right).
$$

\item \label{eq:box.lin.alpha} For all $\alpha \in \mathbb{R}$ we have $(\alpha X^\Box)\equaldist \alpha X^\Box$.

\item \label{Box.inverse.properties} If $E[Z^{-2}]<\infty$ then
$$
E[(Z^{-\Box})^2]=\frac{1}{E[Z^{-2}]} \qmq{and} \mathcal{L}((Z^{-\Box})^\Box) = \mathcal{L}(Z).
$$

\end{enumerate}
\end{lemma}

\begin{proof}
\begin{enumerate}[wide=0pt,label=\alph*)]
\item The expression given for $E[W^2]$ is clear from the representation of $W$ as a mixture. Under the assumption $E[W^2]>0$, for any bounded continuous $f$
identity \eqref{def:Xbox} yields 
\begin{align} \label{eq:WBox.gen}
E[f(W^\Box)]=\frac{E[W^2f(W)]}{E[W^2]}=\frac{\alpha E[X^2f(X)]+(1-\alpha)E[Y^2f(Y)]}{E[W^2]}.
\end{align}
As at least one of $\alpha E[X^2],(1-\alpha)E[Y^2]$ must be positive, without loss generality, by relabelling $X$ and $Y$ if necessary, we may assume $E[X^2]$ and $\alpha$ are both positive. If $E[Y^2]=0$ then 
$E[W^2]=\alpha E[X^2]$ and the second term in the numerator of 
\eqref{eq:WBox.gen} is zero, producing $\mathcal{L}(W^\Box)=\mathcal{L}(X^\Box)$, in agreement with \eqref{eq:WBox.as.mixture}, whose second term likewise vanishes.

When $E[Y^2]>0$, then continuing from \eqref{eq:WBox.gen} we  obtain the claim by
\begin{multline}
E[f(W^\Box)]
= \frac{1}{E[W^2]}\left( \alpha E[X^2]\frac{E[X^2f(X^2)]}{E[X^2]}+(1-\alpha)E[Y^2]\frac{E[Y^2f(Y^2)]}{E[Y^2]}\right)\\
= \frac{1}{E[W^2]}\left( \alpha E[X^2]E[f(X^\Box)]+(1-\alpha)E[Y^2]E[f(Y^\Box)]\right).
\end{multline}

\item First suppose that $\mathcal{L}(X^\Box)=\mathcal{L}(Y^\Box)$. Note that this common distribution cannot be point mass at zero, being excluded by the condition $E[X^2]>0$.

Letting $A$ be a measurable subset of $\mathbb{R}$ not containing zero and $\mathcal{L}(Z)=\delta_0$, we obtain
\begin{multline}\label{mix.A.not.0}
P(Y \in A)=E[\mathbbm{1}(Y \in A)]=E[Y^{-2}Y^2\mathbbm{1}(Y \in A)]\\
=E[Y^2]E[(Y^{\Box})^{-2}\mathbbm{1}(Y^\Box \in A)]=\alpha E[X^2]E[(X^{\Box})^{-2}\mathbbm{1}(X^\Box \in A)]\\=\alpha E[X^{-2}X^2\mathbbm{1}(X\in A)]=\alpha P(X \in A)\\
= (1-\alpha)P(Z \in A) + \alpha P(X \in A). 
\end{multline}
Specializing to $A=\mathbb{R} \setminus \{0\}$ we obtain 
\begin{multline*}
P(Y \not = 0) = (1-\alpha)P(Z\not = 0)+\alpha P(X \not = 0) = \alpha P(X \not = 0)\\
\mbox{or} \quad P(Y=0)= (1-\alpha)P(Z=0)+\alpha P(X=0), 
\end{multline*}
thus showing the left and right hand sides of \eqref{mix.A.not.0} are equal for all measurable $A \subseteq \mathbb{R}$. The other direction of the claim follows from part a).

\item As $E[X_n]$ converges to a positive limit, for all $n$ greater than some $n_0$ the second moments of $X_n$ are strictly positive. In particular, 
for any bounded continuous function $f$, 
\begin{align*}
\lim_{n \rightarrow \infty} E[f(X_n^\Box)]=\lim_{n \rightarrow \infty} \frac{E[X_n^2f(X_n)]}{E[X_n^2]}=\frac{E[X^2f(X)]}{E[X^2]}=E[f(X^\Box)],
\end{align*}
where for taking the limit in the numerator we use that $\{X_n^2, n\ge n_0\}$ is uniformly integrable and hence so is $\{X_n^2f(X_n), n\ge n_0\}$, by the boundedness of $f$.

If $X_n$ takes on the value $n^{1/4}$ with probability $1/n$, and is otherwise zero, then $X_n \rightharpoonup X$ where $\mathcal{L}(X)=\delta_0$ and  $E[X_n^2]=\sqrt{n}/n \rightarrow 0=E[X^2]$, but, by part a) $X_n^\Box$ takes the value $n^{1/4}$ with probability one, and hence does not converge in distribution to zero, showing the necessity of the positivity condition. 

\item Applying the definition of the Cauchy transform and of $\mathcal{L}(X^\Box)$, and then adding and subtracting $z^2$ yields
\begin{multline*}
G_{X^\Box}(z)  = E\left[ 
\frac{1}{z-X^\Box} \right]= \frac{1}{E[X^2]}E\left[\frac{X^2}{z-X} \right]\\
= \frac{1}{E[X^2]}E\left[\frac{z^2+X^2-z^2}{z-X} \right] = \frac{1}{E[X^2]}E\left[\frac{z^2+(X-z)(X+z)}{z-X} \right]\\
= \frac{1}{E[X^2]} E\left[
\frac{z^2}{z-X}-X-z
\right],
\end{multline*}
yielding the claim. 

\item When $\alpha E[X^2]=0$, then either $\alpha=0$ or $X\equaldist \delta_0$, and in both cases both sides of the claimed identity are equal to point mass at zero. Otherwise, for a given bounded measurable function $f(\cdot)$, letting $g(\cdot)=f(\alpha \cdot)$, we have
\begin{multline*}
E[f((\alpha X)^\Box)]= \frac{E[(\alpha X)^2f(\alpha X)]}{E[(\alpha X)^2]}= \frac{E[X^2f(\alpha X)]}{E[X^2]}\\
= \frac{E[X^2 g(X)]}{E[X^2]}= E[g(X^\Box)] = E[f(\alpha X^\Box)].
\end{multline*}

\item Applying Definition \ref{def:X.Box.inverse} for $f(z)=xz2$ yields
$$
E[(Z^{-\Box})^2]=\frac{1}{E[Z^{-2}]},
$$
and now with this identity, and also the help of Definition \ref{def:X.Box.inverse}, for any bounded continuous $f$ we obtain 
\begin{multline*}
E[f(Z^{-\Box})^{\Box})] 
= \frac{E[(Z^{-\Box})^2f(Z^{-\Box})]}{E[(Z^{-\Box})^2]} \\
= E[Z^{-2}] E[(Z^{-\Box})^2 f(X^{-\Box})]=E[Z^{-2}]\frac{E[Z^{-2}Z^2f(Z)]}{E[Z^{-2}]}=E[f(X)].
\end{multline*}
\end{enumerate}
\end{proof}

Finally, we summarize the definition and a few important properties of the well-known size bias transformation.

\begin{definition}\label{def:Xs} For $\mathcal{L}(X)$ the law of a non-negative random variable $X$ with positive finite mean, let $\mathcal{L}(X^s)$, the size bias transform of $\mathcal{L}(X)$, have Radon-Nikodym derivative with respect to that of $X$ given by
\begin{align}\label{def:X.sb}
\frac{dP^s}{dP} = \frac{X}{E[X]}.
\end{align}
\end{definition}

Similar to Definition \ref{def:X.Box.inverse}, we have 
\begin{definition} \label{def:X.s.inverse}
For $X$ non-negative and $E[X^{-1}]<\infty$, let the distribution $\mathcal{L}(X^{-s})$ have Radon-Nikodym derivative with respect to that of $X$ given by
\begin{align}
\frac{dP^{-s}}{dP} = \frac{X^{-1}}{E[X^{-1}]}.
\end{align}
\end{definition}

The following facts regarding the size bias transformation mirror those about the square bias transform in Lemma \ref{lem:X.Xneg0.same}. 
\begin{lemma}\label{lem:X.Xneg0.same.sb}
Let $X$ and $Y$ be non-negative $L^1$ random variables, and let $Z$ be a non-negative random variable.
\begin{enumerate}[wide = 0pt,label=\alph*)]
\item \label{size.bias.properties.part.1} Let $\alpha \in [0,1]$ and 
\begin{align*}
\mathcal{L}(W)=\alpha \mathcal{L}(X)+(1-\alpha) \mathcal{L}(Y).
\end{align*}
Then $E[W]=\alpha E[X]+(1-\alpha)E[Y]$, and when $E[W]>0$,
\begin{align} \label{eq:WBox.as.mixture.sb}
\mathcal{L}(W^s)=\frac{\alpha E[X]}{E[W]} \mathcal{L}(X^s)+\frac{(1-\alpha)E[Y]}{E[W]} \mathcal{L}(Y^s).
\end{align}
\item \label{lem:sqbiaa-part2.sb} When $E[X]>0$, $X$ and $Y$ 
satisfy $E[X]\ge E[Y]$ and
$\alpha=E[Y]/E[X]$, 
then
$$
\mathcal{L}(X^s)=\mathcal{L}(Y^s) \,\,
\mbox{if and only if}\,\,
\mathcal{L}(Y)=(1-\alpha)\delta_0 + \alpha\mathcal{L}(X).
$$
\item \label{lem:sqbiaa-part3.sb} If $X_n$ is a sequence of square integrable random variables such that $X_n \rightharpoonup X$ and $E[X_n] \rightarrow E[X] \in (0,\infty)$, then
$
X_n^s \rightharpoonup X^s.
$
The restriction that the first moment of the limit be positive is necessary. 
\item \label{lem.Box.G.sb} When $E[X]>0$ the Cauchy transform of $X^s$ is given by 
$$
G_{X^s}(z)= \frac{1}{E[X]}\left( zG_X(z)-1\right).
$$
\item \label{eq:box.lin.alpha.sb} For all $\alpha >0$ we have $(\alpha X^s)\equaldist \alpha X^s$.
\item \label{Box.inverse.properties.sb} If $E[Z^{-1}]<\infty$ then
$$
E[Z^{-s}]=\frac{1}{E[Z^{-1}]} \qmq{and} \mathcal{L}((Z^{-s})^s) = \mathcal{L}(Z).
$$
\item \label{size.bias.inverse.pair} If $Z = Y^s$, then $E[Z^{-1}]=1/E[Y]<\infty$.
\end{enumerate}
\end{lemma}

\begin{proof} Properties \ref{size.bias.properties.part.1} through \ref{Box.inverse.properties.sb} are exact analogs of the properties of square bias, and the proofs of those properties in Lemma \ref{lem:X.Xneg0.same} work mutatis mutandis here.  We therefore provide a proof only of \ref{size.bias.inverse.pair}. 

First note that by Lemma \ref{lem:X.Xneg0.same.sb} 
\ref{lem:sqbiaa-part2.sb} we may assume without loss of generality that $Y$ has no mass at zero, and in particular that $P(Yf(Y)=1)=1$ for $f(y)=1/y$.
Now noting that by standard convergence arguments \eqref{eq:def.mean.m.szb} can be extended to all functions $f$ for which the left hand side exists we obtain
$$
1=E[Yf(Y)]=E[Y]E[f(Z)] \qmq{yielding} E[Z^{-1}] = \frac{1}{E[Y]} < \infty.
$$

\end{proof}

\subsection{Construction of the Free Zero Bias $\circ$\label{sect:freezerobias.construction}}

We now arrive at the main construction: the free zero bias. To motivate it, we return for a moment to the classical zero bias $X^\ast$ which is defined by the functional equation \eqref{char.zb}, and has the  probabilistic construction \eqref{eq:blah} whose
proof can be found in \cite[Theorem 2.1]{goldstein2005distributional}.  This construction appears in (12) of that work when specializing to the case $P(x)=x$ and $m=0$, upon noting that the random variable $Y$ whose distribution is given in (11) has the $X^\Box$ distribution by virtue of the fact that $Q(x)$ defined in (10) is equal to $P(x)$ in this case.

One may wonder if some version of this statement, with independence replaced by free independence and uniform replaced by some other distribution perhaps, might characterize the free zero bias of \eqref{char.freezb}.  While nothing of this nature has presented itself, the following definition does have the same flavor, and (as we see shortly) works.

\begin{definition} 
\label{def:fzb}
If $X$ is an $L^2$ random variable with $E[X]=0$, the {\em free zero bias} of $X$, denoted $X^\circ$, is defined (by its law) as follows:
\begin{equation}\label{eq:circ.equiv.box}
X^\circ \equaldist (X^\Box)^\bflat.
\end{equation}
That is: the law $\mathcal{L}(X^\circ)$ is defined by its Cauchy transform, which satisfies
\[ G_{X^\circ}(z) = G_{(X^\Box)^\bflat}(z), \qquad z\in\mathbb{C}_+ \]
following Lemmas \ref{lem:bflat} and \ref{lem:X.Xneg0.same}(\ref{lem.Box.G} defining the $\flat$ and $\Box$ transforms in terms of Cauchy transforms.
\end{definition}

Note: the free zero bias, like the square bias and El Gordo transforms it is composed of, acts only on probability measures; throughout this paper, we have it act on random variables by acting on their laws, and abuse notation writing (as above) $X^\circ$ rather than $\mathcal{L}(X)^\circ$.

\begin{remark}\label{rem:U=Gordo}
We see now that to compute either the classical or free zero bias of a mean zero $X$ the first step is to compute $X^\Box$, followed by multiplying by an independent $U \equaldist \mathrm{Unif}[0,1]$ in the classical case. That multiplication in the free case is replaced by taking the El Gordo transformation in \eqref{eq:circ.equiv.box}. 
\end{remark}

Here are several important properties of the free zero bias, including a concrete formula for its Cauchy transform.
\begin{lemma} \label{lem:Gcirc.props}
The free zero bias defined above satisfies the following properties.
\begin{enumerate}[wide = 0pt,label=\alph*)]

\item \label{lem.Gcirc.1} $X^\circ\equaldist Y^\circ$ if and only if $X^\Box\equaldist Y^\Box$.

\item \label{lem.Gcirc.2} For all $\alpha \in \mathbb{R}$ we have $(\alpha X)^\circ\equaldist \alpha X^\circ$. 

\item \label{lem.Gcirc.3} If 
$X_n \rightharpoonup X$ and $E[X_n^2] \rightarrow E[X^2]>0$,  then $X_n^\circ \rightharpoonup X^\circ$.

\item \label{lem.Gcirc.4} The Cauchy transform of $X^\circ$ satisfies
\begin{align} \label{eq:GX.circ.EX=0}
G_{X^\circ}(z)=-\sqrt{\frac{zG_X(z)-1}{\mathrm{Var}(X)}}.
\end{align}
\end{enumerate}
\end{lemma}

\begin{proof}
For a), if $X^\Box, Y^\Box$ have the same law then the same holds for $X^\circ,Y^\circ$ via Definition \ref{def:fzb}. Conversely, as the El Gordo transform is injective by Lemma \ref{lem:root.Geometric.Mean}, whenever $X^\circ,Y^\circ$ are equally distributed the same is true for $X^\Box, Y^\Box$. Next, for b), note that by Lemma \ref{lem:bflat} and e) of Lemma \ref{lem:X.Xneg0.same}, for all $\alpha \in \mathbb{R}$,
$$(\alpha X)^\circ=((\alpha X)^\Box)^\bflat =(\alpha X^\Box)^\bflat= \alpha (X^\Box)^\bflat = \alpha X^\circ.
$$
For the next claim, part c) of Lemma \ref{lem:X.Xneg0.same} shows that the assumed convergence conditions on $X$ imply that $X_n^\Box \rightharpoonup X^\Box$, so the claim follows by the distributional continuity of the El Gordo transform provided by Lemma \ref{lem:root.Geometric.Mean}.

Writing out the composition in \eqref{eq:circ.equiv.box} using the definition \eqref{G.sharp} and then applying d) of Lemma \ref{lem:X.Xneg0.same}(\ref{lem.Box.G} in the case $E[X]=0$ yields
$$
G_{(X^\Box)^\bflat}(z)=-\sqrt{\frac{1}{z}G_{X^\Box}(z)}=-\sqrt{\frac{1}{zE[X^2]}\left(
z^2G_X(z)-z
\right)}
$$
which immediately simplifies to \eqref{eq:GX.circ.EX=0}. \end{proof}

\begin{remark} From Lemma \ref{lem:Gcirc.props}(\ref{lem.Gcirc.2}, if $X$ is symmetric (i.e.\ has the same distribution as $-X$) then so is $X^\circ$: $-X^\circ \equaldist (-X)^\circ \equaldist X^\circ$.  (The same argument shows this result holds for the classical zero bias.)  We will see this observation witnessed in several examples below.
\end{remark}

\begin{remark} Neither the square bias or El Gordo transforms require centered random variables, and indeed Definition \ref{def:fzb} makes perfect sense for all $L^2$ random variables, centered or not.  All the properties in Lemma \ref{lem:Gcirc.props} parts \ref{lem.Gcirc.1}, \ref{lem.Gcirc.2}, and \ref{lem.Gcirc.3} hold true if $E[X]\ne0$, while the reader can readily verify that \eqref{eq:GX.circ.EX=0} is replaced more generally with
\begin{align}\label{eq:Gcirc:-mean/z}
G_{(X^\Box)^\bflat}(z)=-\sqrt{\frac{1}{E[X^2]}\left(zG_X(z)-\frac{E[X]}{z}-1\right)}.
\end{align}

Nevertheless, as mentioned in \eqref{eq:def.mean.m.fzb}, where $X$ has mean $m \in \mathbb{R}$, we define
\begin{align}\label{eq:Xcirc.def.mean.m}
X^\circ \equaldist (X-m)^\circ+m, 
\end{align}
noting that this definition is in agreement with the previous one in the case $m=0$. Extending the definition in this way preserves the fixed point property of the transformation, as explored in Example \ref{ex:rad}(\ref{ex:two-point-centered}.

However, the mean $0$ condition is required for the defining functional equation to hold, in both the classical and free cases. In comparison to the classical case: if $U\equaldist \mathrm{Unif}[0,1]$ is independent from $X^\Box$ then $UX^\Box$ is a random variable, but it only coincides with $X^\ast$when $E[X]=0$.  
\end{remark}

We now show the main reason why the free zero bias transformation is of interest. We will find it convenient to work with resolvent functions $f_z\colon\mathbb{R}\to\mathbb{C}$
\begin{align}\label{def:res.fun}
f_z(u)=\frac{1}{z-u},\; z \in \mathbb{C}_+
\qmq{which satisfy} f^{[1]}_z(u,v)=\frac{1}{(z-u)(z-v)}.
\end{align}
These functions are bounded and smooth on $\mathbb{R}$.

We recall that the semi-circle law $S_{\sigma^2}$ with mean zero and variance $\sigma^2$ has Cauchy transform
\begin{align}\label{eq:C.trans.S.sigma^2}
G_{S_{\sigma^2}}(z) = \frac{z - \sqrt{z^2-4\sigma^2}}{2\sigma^2}.
\end{align}

\begin{lemma} \label{lem:Xcirc.fixed.pt}
For any $\sigma^2>0$ the mapping $\mathcal{L}(X) \rightarrow \mathcal{L}(X^\circ)$ with domain $\mathcal{D}_{0,\sigma^2}$ has the semi-circle law with variance $\sigma^2$ as its unique fixed point, and for all $X \in \mathcal{D}_{0,\sigma^2}$
\begin{align}\label{free.Stein.identity}
E[Xf(X)]=\sigma^2E[f^{[1]}(X^\circ,Y^\circ)]
\end{align}
for all Lipshitz functions $f$, and 
\begin{align}\label{eq:X*.unif.interpolation}
X^* \equaldist UY^\circ+(1-U)X^\circ
\end{align}
where $Y^\circ,X^\circ$ and $U \equaldist \mathcal{U}[0,1]$ are independent, and $Y^\circ \equaldist X^\circ$.
\end{lemma}

\begin{proof}
Setting $X^\circ\equaldist X$ in \eqref{eq:GX.circ.EX=0} of Lemma \ref{lem:Gcirc.props} yields
$$
\sigma^2 G_X^2(z)-zG_X(z)+1=0,
$$
implying that $G_X(z)=G_{S_{\sigma^2}}(z)$ via \eqref{eq:C.trans.S.sigma^2}.  For the second claim, letting $f_z$ be a resolvent function for a given $z \in \mathbb{C}_+$, and using \eqref{eq:GX.circ.EX=0}, we have
\begin{multline}\label{eq:Xf_z(Z)}
E[Xf_z(X)] = E \left[\frac{X}{z-X}\right] = E \left[\frac{z+X-z}{z-X}\right]
= zG_X(z)-1= \sigma^2 G_{X^\circ}^2(z).
\end{multline}
Now, from \eqref{def:res.fun}, 
\[ E[f^{[1]}_z(X^\circ,Y^\circ)] = E\left[\frac{1}{(z-X^\circ)(z-Y^\circ)}\right] = 
E\left[\frac{1}{z-X^\circ}\right]E\left[\frac{1}{z-Y^\circ}\right] = G_{X^\circ}^2(z). \]
This identity, along with \eqref{eq:Xf_z(Z)}, shows that \eqref{free.Stein.identity} holds for all $f_z, z \in \mathbb{C}_+$.

Now, for all such $f_z$, rewriting $E[f^{[1]}_z(X^\circ,Y^\circ)]$ in terms of the derivative of $f_z$ and an independent $U \equaldist {\mathcal U}[0,1]$, we obtain
\begin{multline*}
\sigma^2 E[f_z'(X^*)]=\sigma^2 E[Xf_z(X)]=\sigma^2 E[f^{[1]}_z(X^\circ,Y^\circ)] \\= \sigma^2 E\left[
\frac{f_z(Y^\circ)-f_z(X^\circ)}{Y^\circ-X^\circ}
\right]
=\sigma^2 E[f_z'(UY^\circ+(1-U)X^\circ)].
\end{multline*}
Letting $\Upsilon = UY^\circ +(1-U)X^\circ$, we have therefore shown that $E[f_z'(X^\ast)] = E[f_z'(\Upsilon)]$ for all $z\in\mathbb{C}_+$.  Since $f_z'(u) = 1/(z-u)^2 = -\frac{d}{dz}f_z(u)$, we have
\[ E[f_z'(\Upsilon)] = -E[\frac{d}{dz}f_z(\Upsilon)] = -\frac{d}{dz}E[f_z(\Upsilon)] = -\frac{d}{dz}G_\Upsilon(z) \]
where the second equality (differentiating under the integral) follows because $(u,z)\mapsto f_z(u)$ and $(u,z)\mapsto f_z'(u)$ are continuous on a neighborhood of each point $(u,z)\in\mathbb{R}\times\mathbb{C}_+$.  Similarly, $E[f_z'(X^\ast)] = -\frac{d}{dz}G_{X^\ast}(z)$.  Hence, we've established that $\frac{d}{dz}G_{\Upsilon}(z) = \frac{d}{dz}G_{X^\ast}(z)$ for all $z\in\mathbb{C}_+$, which means $G_{\Upsilon}$ and $G_{X^\ast}$ differ by a constant, which must be $0$ since both Cauchy transforms are $O(1/z)$ near $\infty$.  Hence $G_{\Upsilon} = G_{X^\ast}$, and by the Stieltjes inversion formula, it follows that $\Upsilon\equaldist X^\ast$, establishing \eqref{eq:X*.unif.interpolation}.

Now, via \eqref{eq:X*.unif.interpolation}, for all Lipschitz $f$ we have
\begin{multline*}
E[Xf(X)] = \sigma^2E[f'(X^*)] = \sigma^2E[f'(UY^\circ+(1-U)X^\circ)]\\=\sigma^2 E\left[ \frac{f(Y^\circ)-f(X^\circ)}{Y^\circ-X^\circ} \right]
= \sigma^2 E[f^{[1]}(X^\circ,Y^\circ)],
\end{multline*}
as desired.

\end{proof}

\begin{remark} As stated, the final equality in the proof of Lemma \ref{lem:Xcirc.fixed.pt} a priori requires $f^{[1]}$ to be defined as $f'$ on the diagonal.  We will see in Theorem \ref{thm:ac} below that the law of $X^\circ$ is always absolutely continuous, and so it is not actually necessary to define $f^{[1]}$ on the diagonal when taking expectations of $f^{[1]}(X^\circ,Y^\circ)$.
\end{remark}

\subsection{The `Replace One' Property}
In the classical world, much of the success that the zero bias transformation enjoys is due to what can be called `the replace one property'. 
The following lemma gives its free analog; we include the classical one for the sake of comparison. For a clearer parallel, we extend the scope of the classical zero bias to general mean random variables, as in \eqref{eq:Xcirc.def.mean.m}.

We begin with the following fact which is useful in both the size bias and free zero bias cases below:
\begin{lemma} \label{eq:sub.Xf_z(X)}
When $S_n=X_1+\ldots+X_n$, a sum of freely independent variables, and $f_z$ is a resolvant function, then
$$
E[S_n f_z(S_n)] = \sum_{i=1}^n E[X_i f_{\omega_{\mu_i,\nu_{n \neg i}}(z)}(X_i)],
$$
where $\omega_{\mu_i,\nu_{n \neg i}}$ is a subordinator as in Proposition \ref{prop:subordinator.def}.
\end{lemma}

\begin{proof}
Applying (free) conditioning as in Section \ref{sec:free prob}), \eqref{eq:cond.resolve}
for the third equality, and the
definition of $f_\omega$ in the last, we obtain
\begin{multline*}
E[S_nf_z(S_n)]=\sum_{i=1}^n E[X_i f_z(S_n)] = \sum_{i=1}^nE[X_iE[ f_z(S_n)|X_i]] \\= \sum_{i=1}^nE\left[ \frac{X_i}{\omega_{\mu_i,\nu_{n \neg i}}(z)-X_i}\right] = \sum_{i=1}^nE[X_i f_{\omega_{\mu_i,\nu_{n \neg i}}(z)}(X_i)].
\end{multline*}
\end{proof}

In what follows, we adopt the definition of the classical $X^*$ for non-mean zero random variables just as was done in \eqref{eq:def.mean.m.fzb} for the free case.

\begin{theorem} \label{eq:Todd's.formula}
Consider variables $X_1,\ldots,X_n$ with means $m_1,\ldots,m_n$ and finite, non-zero variances $\sigma_1^2,\ldots,\sigma_n^2$, and sum $S_n=X_1+\cdots+X_n$ where $\sigma^2=\mathrm{Var}(S_n)$. For $i=1,\ldots,n$, let $I$ be a random index with distribution $P(I=i)=\sigma_i^2/\sigma^2$.

When the summand variables are classically independent, and $I$ is independent of $\{X_i,X_i^*,i=1,\ldots,n\}$ where $X_i^*$ is independent of $X_j, j \not =i$ and has the $X_i$ zero-bias distribution, then $S_n^*$ given by $S_n-X_I+X_I^*$ has the $S_n$-zero bias distribution, and the characteristic function of $S_n^*$ satisfies
\begin{align}\label{eq:sigh.Sn}
\psi_{S_n^*}(t)=E[\psi_{X_I^*}(t)\psi_{S_n - X_I}(t)].
\end{align}

When the summand variables are freely independent, then the Cauchy transform of $S_n^\circ$ satisfies

\begin{equation} \label{eq:free.replace.one}
G_{S_n^\circ}^2(z)=E[G_{X_I^\circ}^2(\omega_{X_I,S_n-X_I}(z))].
\end{equation}
When the summands are identically distributed as $X$, these relations respectively specialize to 
\begin{align} \label{eq:free.replace.first.one}
\psi_{S_n^*}(t)=\psi_{X^*}(t)\psi_{S_{n-1}}(t) \qmq{and} G_{S_n^\circ}(z)=G_{X^\circ}(\omega_{X_n,S_n-X_n}(z)).
\end{align}
\end{theorem}

\begin{remark} As a reminder: $\omega_{V,W}$ is the subordinator function introduced in Proposition \ref{prop:subordinator.def} by the relation $G_{V+W}(z) = G_V(\omega_{V,W}(z))$ for $z\in\mathbb{C}_+$ when $V,W$ are freely independent.  Here $V = X_n$ and $W = S_n-X_n = X_1+\cdots+X_{n-1}$ are indeed freely independent, since all the $X_k$ are free from each other.  Subordinators connect Cauchy transforms of free sums, in an analogous (but more complicated) manner to the simple multiplicative relationship between characteristic functions of independent sums.  It is therefore natural for them to appear in the free analog of ``the replace one property''. \end{remark}

\begin{remark} \label{rk.circ.replace.one.fails} The classical ``replace one property'' for, say, i.i.d.\ variables does {\em not} hold for freely independent random variables; that is to say, in general
\[ S_n^\circ\; \mathop{\ne}^{\mathrm{d}}\; S_n-X_n + X_n^\circ \]
where $X_n^\circ$ is freely independent from $X_1,\ldots,X_{n-1}$.  Indeed: from Proposition \ref{prop:subordinator.def},
\[ G_{S_n-X_n+X_n^\circ}(z) = G_{X_n^\circ}(\omega_{X_n^\circ,S_n-X_n}(z)). \]
From Theorem \ref{eq:Todd's.formula}, this would equal $G_{S_n^\circ}(z)$ if and only if $G_{X_n^\circ}(\omega_{X_n^\circ,S_n-X_n}(z)) = G_{X_n^\circ}(\omega_{X_n,S_n-X_n}(z))$, and (from the univalence of $G_{X_n^\circ}$ and subordinators near $\infty$) this would imply that $X_n^\circ \equaldist X_n$.  Thus, from Lemma \ref{lem:Xcirc.fixed.pt}, the {\em only} case in which $S_n^\circ \equaldist S_n-X_n+X_n^\circ$ is when the summands $X_n$ are all semicircular --- in which case the equation reduces to the tautology $S_n \equaldist S_n$.
\end{remark}

\begin{myproof}{\it Theorem}{\ref{eq:Todd's.formula}}
For a proof of \eqref{eq:sigh.Sn} in the mean zero case, see \cite{goldstein1997stein}. Using that result in the third equality, and with $m=E[S_n]$, in the general mean case we have 
\begin{align*}
\psi_{S^*}(t)=\psi_{(S-m)^*+m}(t)&=e^{itm}\psi_{(S-m)^*}(t)
\\
&=e^{itm} 
E\left[E[\psi_{(X_I-m_I)^*}(t)\psi_{(S_n -m) - (X_I-m_I)}(t)|I]\right]
\\
&=E\left[E[\psi_{(X_I-m_I)^*+m_I}(t)\psi_{(S_n -m)+m - [(X_I-m_I)+m_I]}(t)|I]\right] \\
&=E[\psi_{X_I^*}(t)\psi_{S_n -X_I}(t)].
\end{align*}

For the free case, this same type of reasoning, and the simply proved fact that
$$
\omega_{\mu,\nu}(z-(b+a))= \omega_{\mu+a,\nu+b}(z)-a
$$
allows us to reduce to the mean zero case. Next, using \eqref{def:res.fun}, for any random variable $V$ with independent copy $W$, and any $z\in\mathbb{C}_+$ we have
\begin{align} \nonumber E[f^{[1]}_z(V,W)] &= E[(z-V)^{-1}(z-W)^{-1}] \\
&= E[(z-V)^{-1}]E[(z-W)^{-1}] = G_V^2(z). \label{eq.E.del.fz} \end{align}
Now, let $\mu_i$ and $\nu_{n \neg i}$ be the probability distributions of $X_i$ and $S_n-X_i$ respectively.  Using \eqref{eq.E.del.fz} with $V = S_n^\circ$, $T_n^\circ$ an independent copy of $S_n^\circ$, \eqref{free.Stein.identity} of Lemma \ref{lem:Xcirc.fixed.pt}, and the definition of $S_n$, we have
\begin{align*} \sigma^2 G_{S_n^\circ}^2(z)= \sigma^2 E [f^{[1]}_z(S_n^\circ,T_n^\circ)] = E[S_n f_z(S_n)]=\sum_{i=1}^n E[X_i f_{\omega_{\mu_i,\nu_{n \neg i}}(z)}(X_i)],
\end{align*}
invoking Lemma \ref{eq:sub.Xf_z(X)} for the final equality.
 
Applying again the functional equation \eqref{free.Stein.identity}, the summands in the last expression are equal to
\[ \sigma_i^2 E[f^{[1]}_{\omega_{\mu_i,\nu_{n \neg i}}(z)}(X_i^\circ,Y_i^\circ)] = \sigma_i^2G_{X_i^\circ}^2(\omega_{\mu_i,\nu_{n \neg i}}(z)) \]
with the equality again via \eqref{eq.E.del.fz}.  Summing up, we have therefore shown that
\[ \sigma^2 G_{S_n^\circ}^2(z) = \sum_{i=1}^n \sigma_i^2G_{X_i^\circ}^2(\omega_{\mu_i,\nu_{n \neg i}}(z)) \]
and dividing through by $\sigma^2$ yields \eqref{eq:free.replace.one}.  The final statement of the theorem, in the identical distribution case, follows immediately as in this case $\sigma^2=n\sigma_i^2$ and $\omega_{\mu_i,\nu_{n \neg i}}=\omega_{\mu_n,\nu_{n \neg n}}$ for all $i =1,\ldots, n$. 
\end{myproof}

We now give the parallel result for the size bias. 
\begin{theorem} \label{eq:size.bias.change.1.formula}
Consider non-negative variables $X_1,\ldots,X_n$ with finite non-zero means $m_1,\ldots,m_n$ and sum $S_n=X_1+\cdots+X_n$ with $m=E[S_n]$. For $i=1,\ldots,n$, let $I$ be a random index with distribution $P(I=i)=m_i/m$. 
When the summand variables are classically independent, and $I$ is independent of $\{X_i,X_i^s,i=1,\ldots,n\}$ where $X_i^s$ is independent of $X_j, j \not =i$ and has the $X_i$ size-bias distribution, then $S_n^s$ given by $S_n-X_I+X_I^s$ has the $S_n$-zero bias distribution, and the characteristic function of $S_n^s$ satisfies
\begin{align}\label{eq:sigh.Sn.sb}
\psi_{S_n^s}(t)=E[\psi_{X_I^s}(t)\psi_{S_n - X_I}(t)].
\end{align}
When the summand variables are freely independent, then the Cauchy transform of $S_n^s$ satisfies
\begin{equation*} 
G_{S_n^s}(z)=E[G_{X_I^s}(\omega_{X_I,S_n-X_I}(z))].
\end{equation*}
When the summands are identically distributed as $X$, these relations respectively specialize to 
\begin{align} \label{eq:free.replace.first.one.sb}
\psi_{S_n^s}(t)=\psi_{X^s}(t)\psi_{S_{n-1}}(t) \qmq{and} G_{S_n^s}(z)=G_{X^s}(\omega_{X_n,S_n-X_n}(z)).
\end{align}
\end{theorem}

\begin{proof}
For the classical result \eqref{eq:sigh.Sn.sb} see, for example, \cite{goldstein1996multivariate} or \cite{arratia2019size}. Next, 
\begin{multline*}
mG_{S_n^s}(z)=m E[f_z(S_n^s)]=E[S_nf_z(S_n)]=\sum_{i=1}^n E[X_i f_{\omega_{\mu_i,\nu_{n \neg i}}(z)}(X_i)]\\
= \sum_{i=1}^n m_i E[f_{\omega_{\mu_i,\nu_{n \neg i}}(z)}(X_i^s)]= m \sum_{i=1}^n \frac{m_i}{m} G_{X_i^s}(\omega_{\mu_i,\nu_{n \neg i}}(z)),
\end{multline*}
where we have used Lemma \ref{eq:sub.Xf_z(X)} for the third equality, and the size bias property \eqref{eq:def.secmom.sqb} for the fourth; the proof now follows as for Theorem \ref{eq:Todd's.formula}.
\end{proof}

\begin{remark} Mirroring Remark \ref{rk.circ.replace.one.fails}, it follows now that freely independent sums do {\em not} satisfy the classical ``replace one'' property for size bias.  Indeed: if it were true that $S_n^s \equaldist S_n-X_n + X_n^s$, it would then follow that $G_{S_n^s}(z) = G_{X_n^s}(\omega_{X_n^s,S_n-X_n}(z))$; comparing to \eqref{eq:free.replace.first.one.sb}, we see this is only possible if $X_n^s \equaldist X_n$, meaning that all the summands are constant.
\end{remark}

\subsection{A Transform that Fixes the Free Poisson Distribution}

The (classical) zero bias arises as a tool for Gaussian approximation in Stein's method.  As noted just following its definition \eqref{char.zb}, the normal distribution $X\equaldist\mathcal{N}(0,\sigma^2)$ is the unique fixed point of the zero bias transform $X\mapsto X^\ast$; correspondingly, certain metric differences between $X$ and $X^\ast$ yield quantitative bounds on distances between $\mathcal{L}(X)$ and $\mathcal{N}(0,\sigma^2)$.

Stein's method is not solely restricted to the realm of normal approximation, however.  Along similar lines, the (shifted) size bias plays the analogous role in the world of {\em Poisson approximation}: $\mathrm{Poiss}(\lambda)$ is the unique fixed point of the transform $X\mapsto X^s-1$ on the domain $\mathcal{D}^+_\lambda$ (positively supported probability distributions with mean $\lambda>0$).
Accordingly, functionals of the difference $X-X^s+1$ can be used to construct bounds on the total variation distance between $\mathcal{L}(X)$ and $\mathrm{Poiss}(\lambda)$; see
\cite{chen1975poisson}, \cite{barbour1992poisson} and, for example, \cite{goldstein2006total}, where in the first two references the use of size bias is implicit.

The main result of this section is the following identification of a transform (involving both the size bias and El Gordo transforms) for which the free Poisson law $\mathrm{FP}(\lambda,\alpha)$ of Section \ref{sect:FP} is the unique fixed point in $\mathcal{D}^+_{\alpha\lambda}$.

\begin{theorem} \label{thm.FP.fixed}
On the domain $\mathcal{D}_\beta^+$, 
the transform
\[ \mathcal{L}(X) \mapsto \mathcal{L}(X^{\bullet,\lambda}):=  \left(\lambda \mathcal{L}(X^s)+(1-\lambda)\mathcal{L}(X) \right)^\flat, \quad \lambda \in (0,1], \]
has the free Poisson distribution $\mathrm{FP}(\lambda,\beta/\lambda)$ as its unique fixed point.
\end{theorem}

\begin{proof} If $\mathcal{L}(X)$ is a fixed point, then $\mathcal{L}(X) = \left(\lambda \mathcal{L}(X^s)+(1-\lambda)\mathcal{L}(X) \right)^\flat$, which means the two sides have the same (squared) Cauchy transform:
\[ G_X^2(z) = G_{X^{\bullet,\lambda}}^2(z). \]
From the definition of $\flat$, and the fact that Cauchy transforms of mixtures are given by the corresponding convex combinations of Cauchy transforms, this yields
\[ G_X(z)^2= \frac{1}{z}\left[\lambda G_{X^s}(z) + (1-\lambda)G_X(z)\right]. \]
For notational convenience, set $\alpha= \beta/\lambda$, so that $\beta = \alpha\lambda$.  On $\mathcal{D}^+_\beta = \mathcal{D}^+_{\alpha\lambda}$, the definition of the size bias then gives
\[ zG_X^2(z) = \lambda\cdot\frac{1}{\alpha\lambda}(zG_X(z)-1) + (1-\lambda)G_X(z) \]
which is the quadratic equation
\[ \alpha z G_X^2(z) = zG_X(z)-1+\alpha(1-\lambda)G_X(z) = (z+\alpha(1-\lambda))G_X(z)-1. \]
The quadratic formula then gives the solution
\[ G_X(z) = \frac{z+\alpha(1-\lambda)-\sqrt{(z+\alpha(1-\lambda))^2-4\alpha z}}{2\alpha z}. \]
Finally, a calculation shows that the quantity inside the square root satisfies
\[ (z+\alpha(1-\lambda))^2-4\alpha z = (z-\alpha(1+\lambda))^2-4\alpha^2 \lambda \]
hence showing that $G_X = G_{\mathrm{FP}(\lambda,\alpha)}$.  This shows that $\mathrm{FP}(\lambda,\alpha)$ is the only possible fixed point; reading the proof backwards shows that if $X\equaldist\mathrm{FP}(\lambda,\alpha)$ then $G_X^2 = G_{X^{\bullet,\lambda}}^2$.  Thus the two Cauchy transforms must be equal (they must have the same sign in the square root, owing to the asymptotic and mapping constraints), and it follows that $\mathcal{L}(X) = \mathcal{L}(X^{\bullet,\lambda})$ as desired.  \end{proof}

\begin{remark} The mixture defining $X^{\bullet,\lambda}$ only makes sense if $0\le\lambda\le 1$.  For $\lambda>1$, the free Poisson is a rescaling of the $1/\lambda$ case with a point mass removed (see the discussion following \eqref{e.FP.Cauchy}), and so one could in principle write down a modified version of $X\mapsto X^{\bullet,\lambda}$ that applies in this setting (involving adding and removed a point mass).

It is also worth noting that Theorem \ref{thm.FP.fixed} shows that the $\flat$ of a distribution need not be absolutely continuous: $\mathrm{FP}(\lambda,\alpha)$ has a point mass for $0<\lambda<1$, and it is the $\flat$ of another distribution (a mixture of itself and its size bias, as the theorem shows).  In Section \ref{sec:abscont} we will show that this cannot happen if we also square bias: that is, the free zero bias $X^\circ  = (X^\Box)^\flat$ is always absolutely continuous.
\end{remark}

\ignore{
\subsection{Transformation fixing the Free Poisson}
\lcolor{
In the classical world, the zero bias transformation $X \mapsto X^*$ on $\mathcal{D}_{0,\sigma^2}$ and the transformation $X \mapsto X^s-1$ on $\mathcal{D}_\lambda^+$ have, respectively, the normal $\mathcal{N}(0,\sigma^2)$, and the Poisson $\mathcal{P}_\lambda$ as their unique fixed points. The free zero bias is the free analog of the first, and the following result shows that for $\lambda \in (0,1]$ and $\alpha>0$ the transformation $X \mapsto X^{\bullet,\lambda}$ with domain $\mathcal{D}_\alpha^+$ specified by
\begin{align}
\mathcal{L}(X^{\bullet,\lambda}):=
\left(\lambda \mathcal{L}(X^s)+(1-\lambda)\mathcal{L}(X) \right)^\flat
\end{align}
\label{eq:def.G.bull.lamb}
is the free analog of the second; the image of $X$ under this transform is clearly a probability distribution, as it is the $\flat$ of a mixture of probability distributions.}

\begin{proof}

\lcolor{One can verify without difficulty that the fixed point identity $G_{X^{\bullet,\lambda}}(z)=G_X(z)$ is equivalent to the quadratic \eqref{eq:quad.for.free.Pois.via.zfb} in Example \ref{ex:free.Pois}, and thus corresponds to the Free Poisson distribution with the given parameters. }
\end{proof}
}

\subsection{Free Zero Bias Transform Examples}\label{subsec:prop.examples}

\begin{example}\label{ex:rad} 

\begin{enumerate}[wide = 0pt,label=\alph*)]

\item We compute $\mathcal{L}((X^\Box)^\flat)$ for two distributions that do not have mean zero, and compare it to \eqref{eq:Xcirc.def.mean.m} that defines the free zero bias in the general mean case in order to demonstrate that the results of these two transforms, equal when the mean is zero, are not necessarily equal in general. 

First consider the semi-circle distribution $S$ with mean zero and variance one, and let $X=S+m$ for $m \in \mathbb{R}$. Using  \eqref{eq:C.trans.S.sigma^2} with $S\equaldist S_{0,1}$ to obtain
\begin{align*}
G_X(z)=G_{S+m}(z)=G_S(z-m)=\frac{z-m - \sqrt{(z-m)^2-4}}{2},
\end{align*}
identity \eqref{eq:Gcirc:-mean/z} yields
\begin{multline*}
G_{(X^\Box)^\bflat}^2(z)=\frac{1}{m^2+1}\left(zG_{S}(z-m)-\frac{m}{z}-1\right)\\
= \frac{1}{m^2+1} \left(\frac{1}{2}z^2-m\left(\frac{z}{2}+\frac{1}{z}\right)-1- \frac{z}{2}\sqrt{(z-m)^2-4}\right).
\end{multline*}

Using \eqref{eq:Xcirc.def.mean.m} in the second equality, and that $S^\circ \equaldist S$ in the third, we obtain
$$
X^\circ = (S+m)^\circ = S^\circ +m = S+m =X,
$$
thus confirming that $X$ is a fixed point, and hence
\begin{align*}
G^2_{X^\circ}(z)=G^2_{S+m}(z)
= \left(\frac{z-m-\sqrt{(z-m)^2-4}}{2}\right)^2.
\end{align*}
We see that $G_{(X^\Box)^\bflat}$ and $G_{X^\circ}$ cannot be not equal, the former not possessing a limit as $y \downarrow 0$ when evaluated at $z=iy$, unlike the latter.

For our second such example consider mixtures of point masses at zero and at some non-zero value $a$,
$$
\mathcal{L}(X)= \alpha \delta_a + (1-\alpha)\delta_0 \quad \mbox{for $\alpha  \in (0,1]$.}
$$
Then $E[X]=\alpha a$, and shifting the distribution to have mean zero we have
$$
\mathcal{L}(X-\alpha a) = \alpha \delta_{(1-\alpha)a} + (1-\alpha)\delta_{-\alpha a}.
$$
We have $E(X-\alpha a)^2=\alpha (1-\alpha) a^2$,
and hence
\begin{align*}
P(X^\Box=(1-\alpha)a) = \frac{((1-\alpha)a)^2P(X=(1-\alpha)a)}{\alpha (1-\alpha) a^2}=(1-\alpha), 
\end{align*}
thus implying
\begin{align*}
\mathcal{L}((X-\alpha a)^\Box) 
= (1-\alpha) \delta_{(1-\alpha)a} + \alpha \delta_{-\alpha a}.
\end{align*}
Therefore
$$
G_{(X-\alpha a)^\Box}(z)=\frac{1-\alpha}{z-(1-\alpha)a}+\frac{\alpha}{z+\alpha a}=\frac{z}{(z+\alpha a)(z-(1-\alpha)a)}.
$$
Now taking the El Gordo transformation produces
$$
G_{((X-\alpha a)^\Box)^\flat}(z)=-\sqrt{\frac{1}{(z+\alpha a)(z-(1-\alpha)a)}}
$$
and adding the mean $\alpha a$ back in we obtain
\begin{align}\label{eq:wrong.way.right.answer}
 G_{((X-\alpha a)^\Box)^\flat+\alpha a}(z)
=G_{((X-\alpha a)^\Box)^\flat}(z-\alpha a)
=\frac{1}{\sqrt{z(z-a)}},   
\end{align}
corresponding to the arcsine distribution on $[0,a]$.

Alternatively, computing $(X^\Box)^\flat$ as in Definition \ref{def:fzb}, we first have 
$\mathcal{L}(X^\Box)=\delta_a$ from b) of Lemma \ref{lem:X.Xneg0.same}, and from \eqref{G.sharp} we find
\begin{align*}
G_{(X^\Box)^\flat}(z) = \frac{1}{\sqrt{z(z-a)}},
\end{align*} 
agreeing with \eqref{eq:wrong.way.right.answer}.

\item \label{ex:two-point-centered} For the mean zero, variance $ab$ mixture 
$$
\mathcal{L}(X)= \frac{b}{b+a} \delta_{-a} + \frac{a}{b+a}\delta_b 
$$
of point masses at $-a,b$ for positive $a,b$, by a) of Lemma \ref{lem:X.Xneg0.same} we have 
$$
\mathcal{L}(X^\Box)=  \frac{a^2}{ab}\frac{b}{b+a} \delta_{-a} + \frac{b^2}{ab}\frac{a}{b+a}\delta_b =  \frac{a}{b+a} \delta_{-a} + \frac{b}{b+a}\delta_b.
$$
Hence, using Definition \ref{def:fzb} for the second equality, 
$$
G_{X^\Box}(z)=\frac{z}{(z+a)(z-b)} \qmq{and} G_{X^\circ}(z)=\frac{1}{\sqrt{(z+a)(z-b)}}.
$$
In particular, when $X$ is Rademacher, 
$$
G_{X^\circ}(z)=\frac{1}{\sqrt{z^2-1}}.
$$
Recalling Remark \ref{rem:U=Gordo} that contrasts the constructions of the classical zero bias with the free one, note that $U$, used in the classical construction, has the $\mathrm{Beta}(1,1)$ distribution, while the standard arcsine, of which all examples here are scaled and shifted versions via Example \ref{ex:XbflatY.arcsine}, has the $\mathrm{Beta}$ distribution.

\item Taking one further iterate of the free zero bias distribution of the Rademacher in b), the free zero bias distribution of the arcsine law on $[-1,1]$, having mean zero and variance $1/2$, has Cauchy transform
\begin{align*}
	G_{(X^\circ)^\circ}(z) = -\sqrt{2}\sqrt{\frac{z}{\sqrt{z^2-1}}-1}.
\end{align*}
Using Stieltjes inversion (Proposition \ref{prop.Cauchy.robust}), we evaluate $z=x+i\epsilon$ and let $\epsilon\downarrow0$.  Note that $z^2-1 = x^2-1-\epsilon^2 +2i\epsilon x$ and this number is in the {\em lower} half plane when $x<0$; thus, its square root has {\em negative} real part, in particular $\sqrt{z^2-1} \to -\sqrt{x^2-1}$ for $x<0$ as $\epsilon \downarrow 0$.  Similarly $\sqrt{z^2-1} \to \sqrt{x^2-1}$ when $x>0$.  Now evaluating, we see that $G_{(X^\circ)^\circ}(x+i\epsilon)$ converges as $\epsilon\downarrow 0$ to real values when $|x|>1$, meaning that such points are outside the support of $(X^\circ)^\circ$. Meanwhile, in the interval $[-1,1]$, the formula for the imaginary part of the complex square root yields the density
\begin{align} \label{e.iterated.circ}
	\varrho_{(X^\circ)^\circ}(x)=\frac{1}{\pi}\sqrt{1+\frac{1}{\sqrt{1-x^2}}}{\bf 1}_{x \in [-1,1]}.
\end{align} 

\setcounter{figure}{0}

\begin{figure}
  \includegraphics[scale=0.45]{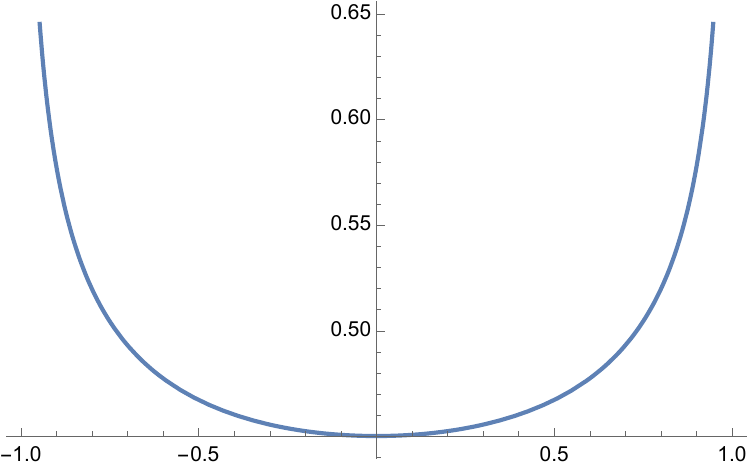}
  \caption{The density \eqref{e.iterated.circ} of the iterated free zero bias of the Rademacher distribution, i.e.\ the free zero bias of the arcsine law.}
  \label{fig.3.17}
\end{figure}
\end{enumerate}

\end{example}

\section{Regularity of the Free Zero Bias}\label{sec:abscont}

In this section we prove that the free zero bias $\mu^\circ$ of any distribution $\munu$ is absolutely continuous with respect to Lebesgue measure on $\mathbb{R}$, with support contained in the convex hull of the support of $\munu$. This behavior is similar to, but not exactly the same as, that of classical zero bias $\mu^*$, which is also absolutely continuous with respect to Lebesgue measure on $\mathbb{R}$, but whose support is always {\em equal} to the convex hull of the support of $\munu$, cf.\ \cite{goldstein1997stein}.  By comparison, $\mu^\circ$ may have support {\em strictly} contained in the support of $\mu$, cf.\ Example \ref{ex:three.pt.zb}.

Section \ref{subsec:abs.cont} proves the absolute continuity of the free zero bias, including quantitative bounds, and Section \ref{sec:support} studies the 
support of $\munu^\circ$ for finitely-supported $\munu$, showing in particular that it is always disconnected when the finite support has size three or more.

\subsection{Absolute Continuity of the Free Zero Bias}\label{subsec:abs.cont}
The following result provides quantitative bounds that were previously unknown even in the  classical case; the 
the remaining claims regarding $\munu^*$ are included for comparison and completeness, and can be found in \cite{goldstein1997stein} and \cite{chen2010normal}.

\begin{theorem}\label{thm:ac}
Let $\munu^*$ and $\munu^\circ$ respectively
be the classical and free zero bias 
of a mean zero probability measure $\munu$ with finite, non-zero variance $\sigma^2$. Then
\begin{align} \label{eq:Holder.1/2}
\munu^*([a,b]) \le \left(\frac{b-a}{\sigma^2}\right)E[|X|] \qmq{and} \munu^\circ([a,b])^2 \le \left(\frac{b-a}{\sigma^2}\right)E[|X|],
\end{align} 
and both $\munu^*$ and $\munu^\circ$ are absolutely continuous with respect to Lebesgue measure. It holds that $\mathrm{supp}(\munu^*)={\rm co}(\mathrm{supp}\munu)$, and $\mathrm{supp}(\munu^\circ)\subseteq {\rm co}(\mathrm{supp}\munu)$.
\end{theorem}

\begin{proof} Let $X$ have law $\munu$, and consider the case where $\mathrm{Var}(X)=1$. 
We first show that $\munu^\circ$ has no point masses. For the sake of contradiction, assume that $X^\circ$ takes on the value $a \in \mathbb{R}$ with non-zero probability $p_a$. Then, in view of \eqref{Xstar.intermsof.Xcirc}, from  
$$
X^*=UX^\circ  + (1-U)Y^\circ
$$
where $Y^\circ$ is an independent copy of $X^\circ$, we see that $X^*=a$ with probability at least $p_a^2$, in contradiction to the absolute continuity of the $X^*$ distribution, see \cite{goldstein1997stein}. Recalling $g^{[1]}(x,y)=(g(y)-g(x))/(y - x)$, as $Y^\circ-X^\circ$ has a continuous distribution, we have that $X^\circ=Y^\circ$ with probability zero, hence the value along the diagonal $x=y$ of double integrals such as \eqref{eq:ignore.diagonal} below, with respect to the product measure $\munu^\circ \times \munu^\circ$, is zero. 
 
With $a<b$ arbitrary, take $g(t)=(t-a)_+ \wedge (b-a)
$. 
Then
\begin{multline} \label{eq:EXg(X).abs.cont}
E[Xg(X)]=E[X(X-a){\bf 1}(a < X \le b)]+ (b-a)E[X{\bf 1}(X>b)]\\
\le (b-a)E[|X|{\bf 1}(a < X \le b)]+ (b-a)E(|X|{\bf 1}(X>b)]\\
= (b-a)E[|X|{\bf 1}(X > a)] \le (b-a)E[|X|]. 
\end{multline}
Using that $g'(t)={\bf 1}_{[a,b]}$ and \eqref{char.zb}, which also holds for Lipshitz functions \cite{chen2010normal}, we find via \eqref{eq:EXg(X).abs.cont} that
$$
\munu^*([a,b])=E[g'(X^*)]=E[Xg(X)] \le (b-a)E[|X|],
$$
thus showing the first claim of \eqref{eq:Holder.1/2} when $\sigma^2=1$.

For the second claim, dividing up the plane according to Figure \ref{fig:abs.cont.}, we first consider the (upper $y=x$ diagonal) collection of  cases, the remaining cases across the diagonal given by interchanging $x$ and $y$.
\begin{align*}
g^{[1]}(x,y) = \left\{
\begin{array}{cl}
	1 & a \le y \le b,a \le x \le b\\
	\frac{y-a}{y-x} & a \le y \le b, x < a\\
	\frac{b-x}{y-x} & y > b, a \le x \le b\\
	\frac{b-a}{y-x} & y>b,x<a\\
	0               & y<a, x<a \qmq{or} y>b, x>b.
\end{array}
\right.
\end{align*}
\newcommand{\sn}{1}
\newcommand{\bn}{2}
\newcommand{\scale}{1.25}
\MULTIPLY{\sn}{-1}{\nsn}
\MULTIPLY{\bn}{-1}{\nbn}

\MULTIPLY{\scale}{\sn}{\ssn}
\MULTIPLY{\scale}{\bn}{\sbn}
\MULTIPLY{\scale}{\nsn}{\snsn}
\MULTIPLY{\scale}{\nbn}{\snbn}

\begin{figure}[ht]
  \centering
\scalebox{0.55}{
  \begin{tikzpicture}
    \coordinate (TL) at (\snsn, \sbn);
    \coordinate (TR) at (\ssn, \sbn);
    \coordinate (BL) at (\snsn, \snbn);
    \coordinate (BR) at (\ssn, \snbn);
    \coordinate (LT) at (-4, \ssn);
    \coordinate (LB) at (-4, \snsn);
    \coordinate (RT) at (4, \ssn);
    \coordinate (RB) at (4, \snsn);
    \draw[dashed] (TL) -- (BL); %
    \draw[dashed] (TR) -- (BR); %
    \draw[dashed] (LT) -- (RT); %
    \draw[dashed] (LB) -- (RB); 
     \newcommand{\scbx}{2}
    \node at (\snbn, \sbn) {\scalebox{\scbx}{$\frac{b-a}{y-x}$}};
    \node at (0,\sbn) {\scalebox{\scbx}{$\frac{b-x}{y-x}$}};
  \node at (\sbn, \sbn) {\scalebox{\scbx}
    {$0$}};
    \node at (\snbn, 0) {\scalebox{\scbx}{$\frac{y-a}{y-x}$}};
    \node at (0, 0) {\scalebox{\scbx}
    {$1$}};
    \node at (\sbn, 0) {\scalebox{\scbx}{$\frac{y-b}{y-x}$}};
    \node at (\snbn, \snbn) {\scalebox{\scbx}
    {$0$}};
    \node at (0, \snbn) {\scalebox{\scbx}{$\frac{a-x}{y-x}$}};
    \node at (\sbn, \snbn) {\scalebox{\scbx}{$\frac{a-b}{y-x}$}};
  \end{tikzpicture}
  }
  \caption{Evaluation of $\partial g(x,y)$ on the partition of the plane induced by the horizontal and vertical lines $y=a,y=b$ and $x=a,x=b$ for $a<b$.}
  \label{fig:abs.cont.}
\end{figure}
Using symmetry when interchanging $x$ and $y$, we obtain
\begin{multline}\label{eq:ignore.diagonal}
E[g^{[1]}(X^\circ,Y^\circ)]   = \munu^\circ([a,b])^2+ 2\int_{(-\infty,a)}\int_{[a,b]} \frac{y-a}{y-x} d\munu^\circ(y)d\munu^\circ(x)\\
+
2\int_{[a,b]} \int_{(b,\infty)} \frac{b-x}{y-x} d\munu^\circ(y)d\munu^\circ(x)
+
2\int_{(-\infty,a)} \int_{(b,\infty)} \frac{b-a}{y-x} d\munu^\circ(y)d\munu^\circ(x).
\end{multline}
As the integrands are all non-negative, using \eqref{eq:EXg(X).abs.cont} for the final inequality
we obtain
\begin{align*} 
\munu^\circ([a,b])^2 \le E[g^{[1]}(X^\circ,Y^\circ)] = E[Xg(X)] \le (b-a)E[|X|],
\end{align*}
showing \eqref{eq:Holder.1/2} when $\sigma^2=1$.

Now turning to the support, letting $g(t) =(t-a)_+$, for $x < y$ we have
\begin{align*}
g^{[1]}(x,y) = \left\{
\begin{array}{cc}
	\frac{y-a}{y-x} & y > a, x < a\\
	1 & y > a, x \ge  a\\
	0 & y < a, x <  a,
\end{array}
\right.
\end{align*}
and hence
\begin{multline*}
E[X(X-a){\bf 1}(X>a)]=E[Xg(X)]\\=E [g^{[1]}(X^\circ,Y^\circ)] 
= 2\int_{-\infty}^a \int_a^\infty \frac{y-a}{y-x}d\munu^\circ(y) d\munu^\circ(x) + \munu^\circ((a,\infty))^2.
\end{multline*}
As $y > a$ in the first integrand, the first integral is non-negative and hence the right hand side is non-negative. 

For any $a$ such that $P(X \ge a)=0$ the left hand side above is zero, and hence both terms on the right are zero; in particular $P(X^\circ>a)=\munu^0((a,\infty))=0$. 
Now let $a$ be such that $P(X \le a)=0$, or, equivalently, that $P(-X \ge -a)=0$. Using the scaling property b) of Lemma 
\ref{lem:Gcirc.props} with $\alpha=-1$ yields
$(-X)^\circ\equaldist - X^\circ$, hence $0=P((-X)^\circ \ge -a)=P(-X^\circ \ge -a)=P(X^\circ \le a)=\munu^\circ((-\infty,a]))$, showing that the support of $\munu^\circ$ is contained in $[a,b]$.

For $X$ having measure $\munu$ with arbitrary variance $\sigma^2>0$, applying the case shown for the unit variance variable $Y=X/\sigma$ and using that $X^* \equaldist \sigma Y^*$  yields
\begin{multline*}
\munu^*([a,b])=P(X^* \in [a,b])=P(\sigma Y^* \in[a,b])=
P(Y^* \in[a/\sigma,b/\sigma]) \\
\le \left(\frac{b-a}{\sigma}\right) E|Y| = \left(\frac{b-a}{\sigma^2}\right) E|X|.
\end{multline*}
The argument for general $\sigma$ in the free zero bias case \eqref{eq:Holder.1/2} is identical. 

The absolute continuity claims now follows, and the final one, regarding the support of $\munu^\circ$ follows from the unit variance case, and that for any $\sigma>0$ we have $\mathrm{supp}(X^\circ)\subseteq {\rm co}(\mathrm{supp}X)$ if and only if  $\mathrm{supp}((\sigma X)^\circ)\subseteq {\rm co}(\mathrm{supp}(\sigma X))$.
\end{proof}

\subsection{Support of $\munu^\circ$ for Finitely-Supported $\munu$}
\label{sec:support} 
The (classical) zero bias $\munu^\ast$ of a probability measure is always supported exactly on the convex hull of $\mathrm{supp}\,\munu$, cf.\ \cite{goldstein1997stein}. In this section we show that this fails to be true in the free case, where $\mathrm{supp}\,\munu^\circ$ can be a strict subset of the convex hull of $\mathrm{supp}\,\munu$.  In fact, we show that when $\munu$ is finitely supported the support of $\munu^\circ$ is a finite union of disjoint closed intervals. 

We begin with the following illustrative example.
\begin{example}\label{ex:three.pt.zb} Consider the three-point centered random variable
\begin{equation} \label{e.3point.meas} X\equaldist \frac{3}{5}\delta_{-1} + \frac{1}{5}\delta_1 + \frac{1}{5}\delta_2 \qmq{with} \mathrm{Var}(X) = \frac{8}{5} =: \sigma^2. \end{equation}
The free zero bias $X^\circ$ has Cauchy transform satisfying
\[  \sigma^2 G_{X^\circ}^2(z) = zG_X(z)-1 = E\left[\frac{z}{z-X}\right]-1 = E\left[\frac{X}{z-X}\right].
\]
Hence
\[ \frac{8}{5}G_{X^\circ}^2(z) = \frac{3}{5}\frac{-1}{z+1} + \frac{1}{5}\frac{1}{z-1} + \frac{1}{5}\frac{2}{z-2} = \frac{8z-10}{5(z^2-1)(z-2)} \]
and from \eqref{eq:GX.circ.EX=0} it follows that
\[ G_{X^\circ}(z) = -\sqrt{\frac{z-5/4}{(z^2-1)(z-2)}}. \]
The rational function inside the square root is continuous on $\mathbb{C}\setminus\{-1,1,2\}$, and hence at all other points $x\in\mathbb{R}$,
\[ -\frac{1}{\pi}\lim_{y\downarrow 0}\mathrm{Im}\,G_{X^\circ}(x+iy) = \frac{1}{\pi}\mathrm{Im}\,\sqrt{\frac{x-5/4}{(x^2-1)(x-2)}}. \]
The rational expression takes negative values on the disconnected set $(-1,1)\sqcup(5/4,2)$, and non-negative values elsewhere.  Thus, with the possible exception of point masses at $\pm 1$ and $2$, the density of $X^\circ$ is
\[ \varrho_{X^\circ}(x) = \frac{1}{\pi}\sqrt{\frac{x-5/4}{(x^2-1)(2-x)}}\mathbbm{1}_{\{x\in[-1,1]\sqcup[5/4,2]\}}. \]
With the aid of Mathematica (or a sufficiently intricate trigonometric substitution) one can check that $\varrho_{X^\circ}$ is genuinely a probability density (total mass $1$), and hence there are no point masses.  Thus, $X^\circ$ has a continuous density (with three poles) on a {\bf disconnected} support set.

\begin{figure}[htb!]
  \includegraphics[scale=0.6]{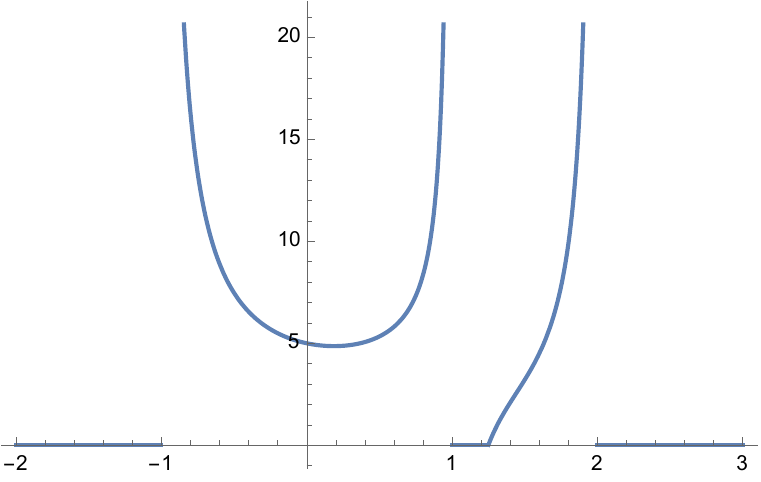}
  \caption{The density of the free zero bias of the three point distribution in \eqref{e.3point.meas}.}
  \label{fig.dens.fzb.3pt}
\end{figure}
\end{example}

\begin{example} The phenomenon demonstrated in Example \ref{ex:three.pt.zb} above and Theorem \ref{theorem disjoint support} below is not restricted to the finite-support case.  Indeed, one can mimic the calculations in Example \ref{ex:three.pt.zb} with ``bumps'' of support instead of point masses.  For example: let $X$ have the arcsine law on $[0,1]$ (with density $1/\pi\sqrt{x(1-x)}$), and let $\munu$ mixture the mixture of the laws of $X-1$, $X+1$, and $X+3$ with coefficients $\frac{5}{6}$, $\frac{1}{12}$, and $\frac{1}{12}$ (which is centered).  Then the Cauchy transform of $\munu$ is
\begin{align*} G_{\munu}(z) &= \frac{5}{6}G_X(z+1) + \frac{1}{12}G_X(z-1) + \frac{1}{12}G_X(z-3) \\
&= \frac{5}{6}\frac{1}{\sqrt{z(z+1)}} + \frac{1}{12}\frac{1}{\sqrt{(z-1)(z-2)}}  + \frac{1}{12}\frac{1}{\sqrt{(z-3)(z-4)}}
\end{align*}
A laborious but elementary calculation provides an explicit radical-rational formula for the square $G_{\munu^\circ}^2$ of the Cauchy transform of the free zero bias $\munu^\circ$, and a simple analysis like above then shows that the support of the law $\munu^\circ$ includes the intervals $[-1,0]$ and $[1,2]$, but does not include any points in $(0,1)$.
\end{example}

Theorem \ref{theorem disjoint support} below shows that this case is `canonical' for finitely supported $X$. Toward that goal, we begin with the following lemma which we will then apply to the square bias (en route to the free zero bias).

\begin{lemma}\label{lem:Gnu.discrete}
Let $X \equaldist \mu:= \sum_{i=1}^n p_i \delta_{s_i}$ with $s_1<\cdots<s_n$, $p_i>0$, and $\sum_{i=1}^n p_i=1$. The Cauchy transform
\begin{align}\label{def:Gnu.lemma.discrete}
G_\mu(z)=\sum_{i=1}^n \frac{p_i}{z-s_i}
\end{align} 
can be written as the rational function $G_\mu(z)=Q(z)/P(z)$, where
\begin{align}\label{eq:G.as.rational}
P(z)=\prod_{i=1}^n (z-s_i), \quad 
Q(z)=\sum_{i=1}^n p_i Q_i(z)
\qmq{with}
Q_i(z)=\prod_{1 \le j \le n, j \not =i} (z-s_j).
\end{align}
The polynomials $P$ and $Q$ are both monic, with respective degrees of $n$, and $n-1$. The polynomial $Q$ has exactly $n-1$ real roots $r_j \in (s_j,s_{j+1}),1,\ldots,n-1$.  

When $E[X^{-1}]=0$
there exists $i^*\in \{1,\ldots,n-1\}$ such that $0 \in (s_{i^*},s_{i^*+1})$, 
and
\begin{align*} 
\{x\in\mathbb{R}: G_{\mu^\flat}(x) \le 0\} = \bigcup_{1 \le j \le i^*-1}[s_j,r_j] \cup [s_{i^*},s_{i^*+1}] \cup \bigcup_{i^*+1 \le j \le n-1}[r_j,s_{j+1}].
\end{align*}
\end{lemma}

\begin{proof} The representation of $G_\mu$ as a rational function with denominator and numerator as in
\eqref{eq:G.as.rational} is immediate, as is the fact that for all $i=1,\ldots,n$, $P$ and $Q_i$ are monic polynomials of respective degrees $n$ and $n-1$.  In particular, the coefficient of $x^{n-1}$ of $Q$ is given by $\sum_{i=1}^n p_i=1$.

Now, for any $s \in \mathbb{R}$,
\begin{align}\label{eq:up.down.pole}
\lim_{x \uparrow s}\frac{1}{x-s}=-\infty \qmq{and} 
\lim_{x \downarrow s}\frac{1}{x-s}=\infty,
\end{align}
and so we see that
\begin{align*} 
\lim_{x \uparrow s_{j+1}}G_\mu(x) = -\infty 
\qmq{and} 
\lim_{x \downarrow s_j}G_\mu(x) = \infty \qmq{for $j=1,\ldots,n-1$.}
\end{align*}
As $Q(x)$ is continuous in $I_j =(s_j,s_{j+1}), j=1,\ldots,n-1$, it must have 
least one root in this interval. As $Q$ has degree $n-1$ and therefore at most this many distinct real roots, each of the $n-1$ intervals $I_j, j=1,\ldots, n-1$ must contain exactly one real root, $r_j$.

For the next claim, 
when $E[X^{-1}]=0$ the random variable $X$ must have both positive and negative support, and so $s_1<0<s_n$. Hence, as all support points are non-zero there exists exactly one index $i^* \in \{1,\ldots,n-1\}$ such that $0 \in (s_{i^*},s_{i^*+1})$. Since
\begin{align*} 
0=-E[X^{-1}]=G_\mu(0)\qmq{and} Q(0)=G_\mu(0)P(0) \qmq{as} P(0)=(-1)^n\prod_{i=1}^n s_i \not =0,
\end{align*}
we find that $Q$ has a root at zero, that is, that $r_{i^*}=0$, allowing us to write
$$
Q(x)=xT(x) \qmq{where} T(x)= \prod_{j: 1 \le j \le n-1, j \not =i^*}(x-r_j).
$$
Cancelling the factor of $z$ in \eqref{G.sharp} one obtains
\begin{align}\label{eq:G.nu.bflat.squared}
G_{\mu^\bflat}^2(z)=\frac{T(z)}{P(z)},
\end{align}
where $T$ is a monic polynomial in $z$ of degree $n-2$, with real coefficients and having no root in the interval $I_{i^*}$.

As $n$ and $n-2$ share the same parity and $P$ and $T$ are monic of those degrees respectively, the limits of $G_{\mu^\bflat}^2(x)$ at both plus and minus infinity are 1, and we find that the sign of \eqref{eq:G.nu.bflat.squared} must be positive on $(-\infty,s_1)$ and $(s_n,\infty)$.

Assume that for some $1 \le j \le i^*-1$ the sign of \eqref{eq:G.nu.bflat.squared} is positive in some non trivial open interval with right endpoint $s_j$. At  $s_j$, via \eqref{eq:up.down.pole}, the sign of a multiplicative factor of \eqref{eq:G.nu.bflat.squared} changes, thus flipping the sign of \eqref{eq:G.nu.bflat.squared} to negative, and remaining negative until encountering the root $r_j$. At this (single) root the sign flips to positive, and remains so until the end of the interval at the pole $s_{j+1}$.

As the premise is true for $j=1$, it remains true for all $j \le i^*-1$, and the sign of \eqref{eq:G.nu.bflat.squared} is negative in the intervals $(s_j,r_j)$ and positive in the interval $(r_j,s_{j+1})$. As the sign, via the same reasoning, changes to negative at $s_{i^*}$ and the root previously in $I_{i^*}$ has been cancelled, the sign of \eqref{eq:G.nu.bflat.squared} is negative throughout the interval $(s_{i^*},s_{i^*+1})$. 

Arguing under the premise that the sign of \eqref{eq:G.nu.bflat.squared} is negative in some non trivial open interval with right endpoint $s_j$, the same inductive argument shows that the sign of \eqref{eq:G.nu.bflat.squared} is negative in the intervals $(r_j,s_{j+1})$ for  $i^*+1 \le j \le n-1$, completing the proof of the lemma. 
\end{proof}

We now move to the consideration of the support of the free zero bias of a finitely supported measure. 
\begin{theorem} \label{theorem disjoint support}
Let $X$ have mean zero and variance $\sigma^2 \in (0,\infty)$, and be supported on a finite set, and 
let $S={\rm supp}(X) \setminus \{0\}=\{s_1,\ldots, s_n\}$ with $s_1<\ldots<s_n$ and $P(X=s_j)=p_j, 1 \le j \le n$. Then
$$
G_{X^\circ}^2(z)=\frac{1}{z\sigma^2}\sum_{j=1}^n \frac{s_i^2p_i}{z-s_j},
$$
and letting
$$
Q(z) = \sum_{i=1}^n \frac{s_i^2p_i}{\sigma^2} Q_i(z)\qmq{where} Q_i(z)=\prod_{1 \le j \le n, j \not =i}(z-s_j),
$$
then the polynomial $Q(z)$ has $n-1$ real roots $r_1<\cdots<r_{n-1}$, there exists $i^* \in \{1,\ldots,n-1\}$ such that $0 \in (s_{i^*}, s_{i^*+1})$,
\begin{align} \label{eq:sup.disc.zb}
{\rm supp}(X^\circ) = \bigcup_{1 \le j \le i^*-1}[s_j,r_j] \cup [s_{i^*},s_{i^*+1}] \cup \bigcup_{i^*+1 \le j \le n-1}[r_j,s_{j+1}],
\end{align}
and the density $\varrho_{X^\circ}$ of $X^\circ$ is given by 
$$
\varrho_{X^\circ}(x) = \frac{1}{\pi}\sqrt{\max\{0,-G_{X^\circ}^2(x)\}}.
$$
\end{theorem}

\begin{proof}
Identity \eqref{def:Xbox} yields
$$
P(X^\Box=s_i)=\frac{s_i^2p_i}{\sigma^2} \qmq{for} s_i \in {\rm Sup}(S),
$$
and as these probabilities sum to one we obtain ${\rm supp}(X^\Box)=S$. 
In addition, applying Definition \ref{def:X.Box.inverse},
$$
E[(X^\Box)^{-1}] = \frac{E[(X^{-1})X^2]}{E[X^2]}= \frac{E[X]}{E[X^2]}=0.
$$
The results on $Q$ and \eqref{eq:sup.disc.zb}
now follow by applying Definition \ref{def:fzb}, which yields $X^\circ \equaldist (X^\Box)^\bflat$, and the 
final part of Lemma \ref{lem:Gnu.discrete} with $\mathcal{L}(X)$ replaced by $\mathcal{L}(X^\Box)$, and $p_i$ in \eqref{def:Gnu.lemma.discrete} replaced by $s_i^2p_i/\sigma^2, s_i,i=1,\ldots,n$.

The density formula follows from the Stieltjes inversion formula (Proposition \ref{prop.Cauchy.robust}), noting that in the case at hand taking limits of $z=x+i\epsilon$ as $\epsilon \downarrow 0$ yields the real valued function $G_{X^\circ}^2(x)$ for all $x \not \in \{s_1,\ldots,s_n\}$, and hence the square root of $G_{X^\circ}(x)$ will be imaginary if and only if $G_{X^\circ}^2(x)<0$.
\end{proof}

\section{Infinite Divisibility}
\label{sec:infdiv}
This section is devoted to the proofs of Theorems \ref{main theorem 2} and \ref{main theorem 3}, which connect the free zero bias and the size bias transformations to free infinite divisibility; i.e.\ to random variables $X\equaldist X_{1,n}+\cdots+X_{n,n}$ for any $n\in\mathbb{N}$, where $\{X_{j,n}\}_{1\le j\le n}$ are freely independent and identically distributed.  The seminal papers \cite{bercovici1993free,bercovici1992levy} by Voiculescu and Bercovici explore and characterize the distributions of such random variables in terms of an analogue \eqref{LK.via.nica} to the classical L\'evy--Khintchine representation.  In their complex analytic approach, compactly-supported laws are treated in \cite{bercovici1992levy},  and extended to all freely infinitely divisible laws (with no moment assumptions) in \cite{bercovici1993free}.  Our contributions below begin by giving a new intuitive and probabilistic proof of the free L\'evy--Khintchine formula in the $L^2$ case, via a new Cauchy transform identity using the free zero bias.

In addition, our new approach shows that {\em any} finite measure $\rho$ corresponds, through \eqref{LK.via.nica}, to a unique (up to a choice of mean) $L^2$ freely infinitely divisible random variable $X$ whose variance is the mass of $\rho$, and that the probability measure $\rho/\rho(\mathbb{R})$ has a probabilistic interpretation connecting $X$ and its free zero bias $X^\circ$: this probability measure is the law of the random variable $V_0$ in \eqref{eq.lem:G.circ.hypothesis}-\eqref{eq:R.thm.conclusion.2}.  This free L\'evy law has an immediate probabilistic interpretation: it is the weak limit of the square biases $X_{1,n}^\Box$ of the $n$th root distributions of $X$ (ergo the $L^2$ requirement for this approach).

In parallel, we give a new characterization of all {\em positively} freely infinitely divisible $L^1$ random variables (meaning that all free convolution roots are positively-supported) and we prove a completely new positive L\'evy--Khintchine type formula \eqref{eq:LK+} for these laws.  The measure in this positive formulation is the law of a different random variable $V_+$ connected to $X$ via the size bias, and also having an appealing probabilistic construction: $V_+$ is the weak limit of the size biases $X_{1,n}^s$ of the $n$th root distributions of $X$ (ergo the $L^1$ assumption here). What's more, at the intersection of the two cases (positively freely infinitely divisible $L^2$ random variables), $V_+$ is the size bias of $V_0$, and this is characterized by $V_0$ having a negative moment, as shown in Theorem \ref{prop:L2.pos.fee.inf.div}.

We begin in Section \ref{subsec:relations}, which is encompassed largely by Theorems \ref{thm:G.R.corresp} and \ref{thm:G.R.corresp.sb}.  These theorems prove the ``symbolic equivalence'' expressed in Theorems \ref{main theorem 2} and \ref{main theorem 3} --- in the first case, the equivalence of \eqref{eq.lem:G.circ.hypothesis}--\eqref{eq:R.thm.conclusion.2} in the first case and of \eqref{eq.lem:G.circ.hypothesis.sb}--\eqref{eq:R.thm.conclusion.2.sb} in the second case,
{\em on some truncated domains} --- with the majority of the remainder of the section devoted to extending the relations to the full upper half plane.  To that end, Section \ref{subsec:cpd} focuses on compound free Poisson distributions (cf.\ Section \ref{sect:FP}) and shows that are rich enough to produce all the laws $\mathcal{L}(V_0)$ for which $E[V_0^{-2}]<\infty$ in the first case, and all the laws $V_+ \ge 0$ for which $E[V_+^{-1}]<\infty$ in the second.

Section \ref{sec:inf.div} begins, in parallel between the $V_0$ and $V_+$ cases, by removing the negative moment constraints with an approximations schemes (therefore escaping the compound free Poisson world) in Proposition \ref{prop:ForAnyV}, and then completes the pieces of the proof of Theorems \ref{main theorem 2} and \ref{main theorem 3} in Theorems \ref{thm:non.compact.case} and \ref{thm:non.compact.case.sb}.  It then combines these results in the overlapping case to Prove Theorem \ref{prop:L2.pos.fee.inf.div}.  Section \ref{sect:compact} gives an alternate approach 
under the stronger assumption of bounded random variables; this approach is more self-contained, and also introduces a result, Lemma \ref{lem:uniform.compact.support}, on the uniformity of the support of free convolution roots that is of independent interest.  Throughout the section, and at its conclusion, are several examples demonstrating the computational effectiveness of our new perspective on free infinite divisibility. 

\subsection{New Formulations of Free Infinite Divisibility}\label{subsec:relations}

This section is devoted to preliminary versions of our main Theorems \ref{main theorem 2} and \ref{main theorem 3}, in the forms of Theorem \ref{thm:G.R.corresp} and \ref{thm:G.R.corresp.sb} below.  Theorem \ref{thm:G.R.corresp} shows that Bercovici and Voiculescu's free L\'evy--Khintchine formula in this $L^2$ setting is equivalent to our new formulation \eqref{eq.lem:G.circ.hypothesis} relating the reciprocal Cauchy transform of the free zero bias $X^\circ$ of the freely infinitely divisible random variable $X$ to that of the El Gordo transform $V_0^\bflat$ of its free L\'evy--Khintchine measure $\mathcal{L}(V_0)$.  But the equivalence here is only proved for these analytic transforms on certain truncated open domains of the complex upper half-plane; further tools later in this section are needed to extend these results to their full strength.  Meanwhile, Theorem \ref{thm:G.R.corresp.sb} provides a completely new kind of free L\'evy--Khintchine formula \eqref{eq:R.thm.conclusion.sb} for $L^1$ {\em positively} freely infinitely divisible random variables, and proves an equivalent formulation \eqref{eq.lem:G.circ.hypothesis.sb} relating the reciprocal Cauchy transform of the size bias of $X$ to that of its (new) free L\'evy--Khintchine measure $\mathcal{L}(V_+)$ (cf.\ Theorem \ref{main theorem 3}(\ref{thm2.b.sb}), again with formulas holding for now only on certain truncated domains, with extension to the full upper half-plane proved in following sections.

\ignore{
In this section we \lcolor{first} develop relations between the four types of transforms associated with a distribution $\mathcal{L}(X)$ which are connected to the form of Cauchy transform of $\mathcal{L}(X^\circ)$. In particular, Theorem \ref{thm:G.R.corresp} affirms the equivalence of relation \eqref{eq.lem:G.circ.hypothesis}, involving a reciprocal Cauchy transform depending on $\mathcal{L}(X^\circ)$, to \eqref{eq:R.thm.conclusion} and \eqref{eq:R.thm.conclusion.2}, specifying the forms of the $R$ and Voiculescu transforms $\varphi$ of $\mathcal{L}(X)$ on certain truncated cone domains. The equivalences are those required to yield those between \ref{thm2.b} and \ref{thm2.c} in Theorem \ref{main theorem 2}, once extended to all of $\mathbb{C}_+$. \lcolor{Theorem \ref{thm:G.R.corresp.sb} provides parallel results obtained via the use of size bias.}
}

\begin{theorem}\label{thm:G.R.corresp}
Let $X$ be a mean $m$ random  variable with variance $\sigma^2 \in (0,\infty)$, and $V_0$ a random variable.  The following are equivalent.
\begin{align}\label{eq.lem:G.circ.hypothesis}
F_{X^\circ}(w)=F_{V_0^\bflat}(F_X(w))
\end{align}
for $w\in\mathbb{C}_+$.
\begin{align}\label{eq:R.thm.conclusion}
R_X(z)=m+\sigma^2 E\left[
\frac{z}{1-zV_0}\right]
\end{align}
for $z$ in a reciprocal truncated cone.
\begin{align}\label{eq:R.thm.conclusion.2}
G_{V_0}(z)=(\varphi_X(z)-m)/\sigma^2,
\end{align}
for $z$ in a truncated cone.

If these equivalent conditions hold for both the pairs $(X,V_0)$ and $(X,V_0')$ then $\mathcal{L}(V_0)=\mathcal{L}(V_0')$, and if for both pairs $(X,V_0)$ and $(X',V_0)$ then, assuming $\sigma_X^2 \ge \sigma_{X'}^2$ without loss of generality, 
\begin{align*} 
X \equaldist \left(X'-m_{X'}\right)^{\boxplus t} + m_X \qmq{where} t= \frac{\sigma_X^2}{\sigma_{X'}^2},
\end{align*}
where, for any $W$, $W^{\boxplus t}$ is the $t$-fold free convolution of $W$. 
\end{theorem}

\ignore{
See Remark \ref{remark:free-conv-semigroup} for a discussion of free convolution powers $\mu^{\boxplus t}$ for $t\ge 1$.
}

\begin{proof}
To begin, note that $\varphi_X(z) = R_X(1/z)$ from the definitions \eqref{Grr:relations} and \eqref{def:Voiculescu.trans} of the $R$-transform and the Voiculescu transform.  Hence, if \eqref{eq:R.thm.conclusion.2} holds we have
\begin{multline*} R_X(z)=\varphi_X(1/z)=m+\sigma^2G_{V_0}(1/z) \\= m+\sigma^2 E\left[\frac{1}{1/z-V_0}\right] = m+\sigma^2 E\left[\frac{z}{1-zV_0}\right],
\end{multline*}
yielding \eqref{eq:R.thm.conclusion}.  Similarly if \eqref{eq:R.thm.conclusion} holds then 
\begin{align*}
(\varphi_X(z)-m)/\sigma^2 = (R_X(1/z)-m)/\sigma^2 = E\left[ \frac{1/z}{1-V_0/z}\right]= G_{V_0}(z).
\end{align*}

To show the equivalence of  \eqref{eq.lem:G.circ.hypothesis} and \eqref{eq:R.thm.conclusion.2} we may reduce to the case $m=0$.  Indeed, if we know that 
\eqref{eq.lem:G.circ.hypothesis} holds for mean zero variables then
$$
F_{X^\circ}(z)=F_{(X-m)^\circ +m}(z)=
F_{(X-m)^\circ }(z-m)
=F_{V_0^\bflat}(F_{X-m}(z-m))=F_{V_0^\bflat}(F_X(z)),
$$
and if \eqref{eq:R.thm.conclusion.2} holds for mean zero variables, then since $\varphi_X(z) = m+\varphi_{X-m}(z)$, we find 
\[ (\varphi_X(z)-m)^2/\sigma^2 = \varphi_{X-m}(z)^2/\sigma^2 = G_{V_0}(z). \]
Hence we may assume $m=0$ without loss of generality.

Now, suppose that \eqref{eq:R.thm.conclusion.2} holds in some truncated cone $\Gamma_{\alpha,\beta}$ in the domain of $\varphi_X$ (ergo the reciprocal Cauchy transform $F_X$ has an analytic inverse $F_X^{-1}$ there, cf.\ Proposition \ref{prop.Stolzify}); under the assumption $m=0$ this is the statement $\varphi_X(z) = \sigma^2 G_{V_0}(z)$.  From the definition \eqref{def:Voiculescu.trans}, $\varphi_X(z) = F_X^{-1}(z)-z$ and hence $F_X^{-1}(z) - z = \sigma^2 G_{V_0}(z)$ on $\Gamma_{\alpha,\beta}$.  Letting $w=F_X^{-1}(z)$ for $z\in\Gamma_{\alpha,\beta}$, we have $F_X(w) = z$, and so
\begin{align}\label{eq:w-F} 
w - F_X(w) = \sigma^2 G_{V_0}(F_X(w)).
\end{align}
Using \eqref{eq:GX.circ.EX=0} for the first inequality and then the result of multiplying \eqref{eq:w-F} through by $G_X(w)/\sigma^2$ and noting $F_X = 1/G_X$, and finally applying definition \eqref{G.sharp} for the last, we have
\begin{multline} G^2_{X^\circ}(w) = \frac{wG_X(w) - 1}{\sigma^2} = G_X(w)G_{V_0}(F_X(w))
\\= \frac{1}{F_X(w)}G_{V_0}(F_X(w)) = G^2_{V_0^\bflat}(F_X(w)).\end{multline}
The equality $G_{X^\circ}^2(w) = G_{V_0^\bflat}^2(F_X(w))$ implies that $G_{X^\circ}(w) = G_{V_0^\bflat}(F_X(w))$ (since the negative of a Cauchy transform is not the Cauchy transform of a positive measure).  Taking reciprocals verifies \eqref{eq.lem:G.circ.hypothesis} for $w\in F_X^{-1}(\Gamma_{\alpha,\beta})$, which is open (since $\Gamma_{\alpha,\beta}$ is open and $F_X$ is continuous).  Thus, the two analytic functions $F_{X^\circ}$ and $F_{V_0^\bflat}\circ F_X$ are equal on an open set, and hence open on their full domain $\mathbb{C}_+$ as claimed. 
The reverse implication \eqref{eq.lem:G.circ.hypothesis}$\implies$\eqref{eq:R.thm.conclusion.2} is proved in exactly the same manner in reverse.

Turning to the second half of the theorem: When any of these equivalent identities is satisfied for $X$ then $V_0$ is uniquely determined by its Voiculescu transform in \eqref{eq:R.thm.conclusion.2}, which depends only on the distribution of $X$, which itself encodes the mean. 
When both $X$ and $X'$ satisfy any of the three, and therefore all of the given identities for the same $V_0$, by \eqref{eq:R.thm.conclusion} we must have 
$$
(\varphi_X(z)-m_X)/\sigma_X^2  = (\varphi_{X'}(z)-m_{X'})/\sigma_{X'}^2,
$$
and hence
$$
R_X(z)=t (R_{X'}(z)-m_{X'})+m_X,
$$
thus yielding the final claim. 
\end{proof}

\begin{example}\label{ex:semi.sats}
When $X$ is a distribution with zero mean and variance $\sigma^2 > 0$, we know by Lemma \ref{lem:Xcirc.fixed.pt} that the unique solution to $X^\circ \equaldist X$ is the semicircle distribution with variance $\sigma^2$.
In addition, this distribution must also satisfy 
\eqref{eq.lem:G.circ.hypothesis}
with $G_{V_0^\bflat}(z)=1/z$, hence, by Lemma \ref{lem:bflat}, with $\mathcal{L}(V_0)=\delta_0$. In this case the $\zcirc$ Lévy measure in \eqref{eq:R.thm.conclusion} is given by point mass at zero, implying the well-known
 fact that the $R$-transform of a semicircle is given by $R_{S_{\sigma^2}}(z)=\sigma^2 z$.
In particular $R_{S_{\sigma^2}}(z)$ is an $R$-transform for all $\sigma^2  \ge 0$, and thus, as for all $n \ge 1$ we have $R_{S_{\sigma^2}}(z)= R_{S_{\sigma^2/n}}(z)
+ \cdots + R_{S_{\sigma^2/n}}(z)$, the distribution of $S_{\sigma^2}$ is infinitely divisible. 

In the other direction, given the infinity divisibility of the semi-circle distributions and their corresponding $R$-transforms $R_S(z)=\sigma^2 z$, one can find their L\'evy--Khintchine measures by \eqref{eq:R.thm.conclusion}. Indeed, from there we obtain
$$
G_{V_0}(z)=\varphi_S(z)/\sigma^2=R_S(1/z)/\sigma^2=1/z,
$$
demonstrating that ${V_0} \equaldist \delta_0$.
\end{example}

We now present an analogous result for non-negative random variables, and size biasing, under the weaker condition of the existence of a first moment.
\begin{theorem}\label{thm:G.R.corresp.sb}
Let $X$ be a non-negative, mean $m \in (0,\infty)$ random  variable, and $V_+$ a random variable. The following are equivalent. 
\begin{align}\label{eq.lem:G.circ.hypothesis.sb}
F_{X^s}(w)=F_{V_+}(F_X(w))
\end{align}
For $w\in\mathbb{C}_+$,
\begin{align}\label{eq:R.thm.conclusion.sb}
R_X(z)=mE\left[
\frac{1}{1-z{V_+}}\right]
\end{align}
for $z$ in a reciprocal truncated cone, and 
\begin{align}\label{eq:R.thm.conclusion.2.sb}
G_{V_+}(z)=\varphi_X(z)/mz
\end{align}
 for $z$ in a truncated cone.

If these equivalent conditions hold for both the pairs $(X,{V_+})$ and $(X,V_+')$ then $\mathcal{L}({V_+})=\mathcal{L}(V_+')$, and if for both pairs $(X,{V_+})$ and $(X',V_+)$ then, assuming $m_X \ge m_{X'}$ without loss of generality, 
$$
X \equaldist \left(X'\right)^{\boxplus t} \qmq{where} t= \frac{m_X}{m_{X'}},
$$
where, for any $W$, $W^{\boxplus t}$ is the $t$-fold free convolution of $W$.
\end{theorem}

\begin{remark} As we will see below in Proposition \ref{prop:ForAnyV.sb}, when the equivalent conditions of this theorem hold (on larger domains), the random variable $V_+$ must be be $\ge 0$.
\end{remark}

\begin{proof}
The equivalence between \eqref{eq:R.thm.conclusion.sb} and \eqref{eq:R.thm.conclusion.2.sb} follows from a simple calculation noting that $\varphi_X(z) = R_X(1/z)$, just as in the proof of Theorem \ref{thm:G.R.corresp}.

To prove the equivalence of \eqref{eq:R.thm.conclusion.2.sb} and \eqref{eq.lem:G.circ.hypothesis.sb} on appropriate domains, we follow the same outline as in Theorem \ref{thm:G.R.corresp} (although with fewer calculations needed in this case).  On a truncated cone where $\varphi_X(z)$ is defined, it is equal to $F_X^{-1}(z)-z$.  Hence, on the image domain under $F_X$ (which, according to Proposition \ref{prop.Stolzify}, contains another truncated cone), setting $z=F_X(w)$, \eqref{eq.lem:G.circ.hypothesis.sb} is equivalent to
\[ mF_X(w)G_{V_+}(F_X(w)) = w-F_X(w). \]
Multiplying both sides by $G_X(w)/m = 1/mF_X(w)$ (which is never $0$) yields
\[ G_{V_+}(F_X(w)) = \frac{wG_X(w)-1}{m} \]
and the right-hand-side is equal to the Cauchy transform of the size bias $G_{X^s}(w)$, cf.\ Lemma \ref{lem:X.Xneg0.same.sb}(\ref{lem.Box.G.sb}.  Taking reciprocals yields \eqref{eq.lem:G.circ.hypothesis.sb}.  Every step was reversible, and so the two are in fact equivalent.  \end{proof}

\subsection{Free Compound Poisson Distributions Revisited}\label{subsec:cpd}

Free Poisson distributions $\mathrm{FP}(\lambda,\nu)$ for any $\lambda>0$ and any probability measure $\nu$ on $\mathbb{R}$ were introduced in Section \ref{sect:FP}.  They are preeminent examples of freely infinitely divisible distributions, as they are defined as weak limits of f.i.d.\ triangular arrays, cf.\ Definition \ref{Def.FP}.  Equation \ref{eq:R.FP} gives a general expression for the $R$-transform of any (compound) free Poisson; this formula is left as an exercise in \cite{Nica-Speicher-96}.  It will be derived below as a consequence of our new perspective, with more insight into the nature of the ``jump distribution''.  Indeed, let $U$ be a random variable with $U\equaldist \nu$; \ref{eq:R.FP} then reads
\begin{equation} \label{eq:R.FP.U}
R_{\mathrm{FP}(\lambda,\nu)}(z) = \lambda E\left[\frac{U}{1-zU}\right].
\end{equation}
Note: the free infinite divisibility of $\mathrm{FP}(\lambda,\nu)$ follows immediately from \eqref{eq:R.FP.U}: since the $n$th $\boxplus$-root of a distribution $\mu$ (should it exist) has $R$-transform $\frac{1}{n}R_\mu$, \eqref{eq:R.FP.U} shows that $\mathrm{FP}(\lambda/n,\nu)$ is the $n$th $\boxplus$-root of $\mathrm{FP}(\lambda,\nu)$ for each $n$.

Suppose that $U\in L^2$.  Noting that $\frac{U}{1-zU} = U+\frac{z U^2}{1-zU}$, it then follows that
\begin{equation} \label{eq:FP.L2.1} R_{\mathrm{FP}(\lambda,\nu)}(z) = \lambda E\left[U+\frac{z U^2}{1-zU}\right] = \lambda E[U] + \lambda E[U^2]E\left[\frac{z}{1-zU^\Box}\right] \end{equation}
where the last equality is an application of the characterizing equation \eqref{eq:def.secmom.sqb}for the square bias with test function $f(u) = \frac{z}{1-zu}$. Moreover, since the $R$-transform records the mean and variance as its orders $0$ and $1$ coefficients, \eqref{eq:R.FP} shows that the mean of $\mathrm{FP}(\lambda,\nu)$ is $\lambda E[U]$ and the variance of $\mathrm{FP}(\lambda,\nu)$ is $\lambda E[U^2]$.  Note that, for any random variable $X$ and any $m\in\mathbb{R}$, $R_{X-m}(z) = R_X(z)+m$.  As such, \eqref{eq:FP.L2.1} shows the following.

\begin{lemma} \label{lem:FP.L2} For every $L^2$ random variable $U$, and every $m\in\mathbb{R}$ and $\sigma^2\ge0$, there is a translated compound free Poisson random variable $X$ with mean $m$ and variance $\sigma^2$ satisfying
\begin{equation} \label{eq:FP.L2.2} R_X(z) = m+\sigma^2E\left[\frac{z}{1-zU^\Box}\right]. \end{equation}
\end{lemma}

Similarly, suppose instead that $U\in L^1$, $U\ge 0$, and $E[U]>0$.  Then the size bias $U^s$ is well-defined by \eqref{eq:def.mean.m.szb}, and applying it with the test function $f(u) = \frac{1}{1-zu}$ yields
\begin{equation} \label{eq:FP.L1.1} R_{\mathrm{FP}(\lambda,\nu)}(z) = \lambda E\left[\frac{U}{1-zU}\right] = \lambda E[U]E\left[\frac{1}{1-zU^s}\right]. \end{equation}

This yields the following result, complementing Lemma \ref{lem:FP.L2}.
\begin{lemma} \label{lem:FP.L1}
For every $L^1$ random variable $U\ge 0$ with $E[U]>0$, and every $m\in\mathbb{R}$, there is a compound Poisson random variable $X\ge 0$ with mean $m$ satisfying
\begin{equation} \label{eq:FP.L1.2} R_X(z) = mE\left[\frac{1}{1-zU^s}\right]. \end{equation}
\end{lemma}

\begin{remark} \label{rk:FP.positive} The fact that the random variable $X$ in Lemma \ref{lem:FP.L2} is non-negative follows by definition: here $X\equaldist \mathrm{FP}(\lambda,\nu)$ where $\lambda = E[U]/m$ and $\nu = \mathcal{L}(U)$. Since $\nu$ is positively supported, all the summand distributions $(1-\frac{\lambda}{n})\delta_0+\frac{\lambda}{n}\nu$ in Definition \ref{Def.FP} are positively supported, and hence so are their free convolution powers and the weak limit $X$.
\end{remark}

Now we turn Lemmas \ref{lem:FP.L2} and \ref{lem:FP.L1} around, thinking of $U^\Box$ or $U^s$ as the basic inputs to Equations \eqref{eq:FP.L2.2} and \eqref{eq:FP.L1.2} instead of $U$. Indeed, from Lemma \ref{lem:X.Xneg0.same}(\ref{Box.inverse.properties}, a random variable $V$ is the square bias of some other random variable $U$, $V = U^\Box$, if and only if $E[V^{-2}]<\infty$; similarly, from Lemma \ref{lem:X.Xneg0.same.sb} parts \ref{Box.inverse.properties.sb} and \ref{size.bias.inverse.pair}, a random variable $V\ge 0$ is the size bias of some other $L^1$ random variable $U\ge 0$, $V = U^s$, if and only if $E[V^{-1}]<\infty$.  This argument gives helpful intuition for the following theorem, which is the main result of this section.

\begin{theorem}\label{thm:Compound.Poisson} Let $V_0$ be any random variable satisfying $E[V_0^{-2}]<\infty$, and let $m\in\mathbb{R}$ and $\sigma^2\ge 0$.  There is a (unique in law) freely infinitely divisible random variable $X$ whose $R$-transform is
\begin{equation}
\label{eq:Rmusigma}
R_X(z)  = R_{(m,\sigma^2,V_0)}(z) := m+\sigma^2 E\left[\frac{z}{1-zV_0}\right].
\end{equation}
Similarly, let $V_+$ be any non-negative random variable satisfying $E[V_+^{-1}]<\infty$.  There is a (unique in law) positively freely infinitely divisible random variable $X$ whose $R$-transform is
\begin{equation}\label{eq:Rmusigma.sb}
R_X(z) = R_{(m,V_+)}(z) := mE\left[\frac{1}{1-zV_+}\right].
\end{equation}
\end{theorem}

(A reminder to the reader: {\em positively freely infinitely divisible} measures are characterized by having all convolution roots positively supported, cf.\ Definition \ref{def:pfid}.)

Although Theorem \ref{thm:Compound.Poisson} essentially follows from the preceding lemmas, we will provide a different proof below which derives explicitly from the triangular array construction of compound free Poisson distributions, and thus identifies the jump distribution $\mathrm{FP}(\lambda,\nu)$ probabilistically in terms of the free summands.  The reader could then view our approach below, together with the calculations in \eqref{eq:FP.L2.1} and \eqref{eq:FP.L1.1} (read backwards) as alternate proofs of the characterizing equation \eqref{eq:R.FP} of the $R$-transform of $\mathrm{FP}(\lambda,\nu)$, at least in the $L^2$ or $L^1$ non-negative cases.

To proceed, we need several building blocks which hold in the general context of a sequence $\{X_n\}_{n \ge n_0}$, for some $n_0 \ge 1$, of row sums of triangular arrays:
\begin{align}\label{def:Xn}
X_n \equaldist X_{1,n}+\cdots+X_{n,n} \quad \mbox{where $\{X_{j,n}\}_{1 \le j \le n}$ are f.i.d.}
\end{align}
(Here and in the sequel ``f.i.d.'' is short for {\em freely independent and identically distributed}.) The first lemma, though simple, will be used multiple times below.
\begin{lemma} \label{lem:sublemma} Let \eqref{def:Xn} hold.
\begin{enumerate}

\item If $\lim_{n \rightarrow \infty}\mathrm{Var}(X_n)/n = 0$ in \eqref{def:Xn} then $\mathrm{Var}(X_{n,n}) \rightarrow 0$, and if in addition $E[X_n]/n \rightarrow 0$ then $X_{n,n} \rightharpoonup 0$.

\item If the summands $X_{j,n}$ are non-negative for $n \ge n_0$ and $E[X_n]/n \rightarrow 0$ then $X_{n,n} \rightharpoonup 0$.
\end{enumerate}
\end{lemma}

\begin{proof}
For (1), by the free independence and identical distribution assumption we have
$$
\limsup_{n \rightarrow \infty}\mathrm{Var}(X_{n,n})=\limsup_{n \rightarrow \infty}\mathrm{Var}(X_n)/n=0,
$$
and hence $X_{n,n} = (X_{n,n}-E[X_{n,n}])+E[X_{n,n}] \rightarrow_p 0$ by Chebyshev's inequality and the fact that $E[X_{n,n}]=E[X_n]/n$, and the claim follows as convergence in probability and in distribution are equivalent when the limit is constant. 
For (2), by Markov's inequality, for $\epsilon>0$ we have 
$$
P(X_{n,n}>\epsilon) \le \frac{E[X_{n,n}]}{\epsilon} = E[X_n]/\epsilon n \rightarrow 0, 
$$
thus showing $X_{n,n}$ converges to zero in probability to zero, and hence also in distribution, as the limit is constant.
\end{proof}

Letting
$X_n^{(n)}=X_n-X_{n,n}$, the free ``replace one property" Theorem \ref{eq:Todd's.formula}, notably in the form \eqref{eq:free.replace.first.one}, says
\begin{align}\label{eq:pre.limit.fid.identity}
	G_{X_n^\circ}(z) = G_{X_{n,n}^\circ}(\omega_{X_{n,n},X_n^{(n)}}(z)).
\end{align}

\begin{lemma}\label{lem:Xn.sums.of.Vn.} Let $\{X_n\}_{n \ge n_0}$ be given as in \eqref{def:Xn}.
\begin{enumerate}

\item \label{lem:Xn.sums.of.Vn.1} If the summands $X_{j,n}$ have finite variance, and along some subsequence $\{n_k\}_{k \ge 1}$ we have a random variable $V_0$ such that $X_{n_k,n_k}^\Box \rightharpoonup V_0$, $X_{n_k} \rightharpoonup X$, and $\mathrm{Var}(X_{n_k}) \rightarrow \mathrm{Var}(X)$, then \eqref{eq.lem:G.circ.hypothesis}--\eqref{eq:R.thm.conclusion.2} hold for the pair $X,V_0$. 

\item \label{lem:Xn.sum.of.Vn.2} If the summands $X_{j,n}$ are non-negative with finite strictly positive mean, and along some subsequence $\{n_k\}_{k \ge 1}$ we have a random variable $V_+$ such that $X_{n_k,n_k}^s \rightharpoonup V_+$, $X_{n_k} \rightharpoonup X$, and $E[X_{n_k}]\rightarrow E[X]>0$, then \eqref{eq.lem:G.circ.hypothesis.sb}--\eqref{eq:R.thm.conclusion.2.sb} hold for the pair $X,V_+$.
\end{enumerate}
\end{lemma}

\begin{proof} For convenience, we drop the subscript $k$ from the subsequence. We prove (1) here; the proof of (2) is parallel, using (2) of Lemma \ref{lem:sublemma}(2).

As $\mathrm{Var}(X_n) \rightarrow \mathrm{Var}(X)$, the sequence $X_n, n \ge \lambda$ is uniformly integrable, and $E[X_n] \rightarrow E[X]$. Hence (1) of Lemma \ref{lem:sublemma} obtains on the subsequence, and by the definition $X_n^{(n)}=X_n-X_{n,n}$ we have $X_n-X_n^{(n)}=X_{n,n} \rightharpoonup 0$.

As $X_n\rightharpoonup X$, we conclude that $X_n^{(n)}\rightharpoonup X$.  Therefore, by Proposition \ref{prop:subordinator}, it follows that $\omega_{X_{n,n},X_n^{(n)}}(z) \to \omega_{0,X}(z)$ uniformly on compact subsets of $\mathbb{C}_+$ as $n\to\infty$. Since $0$ is freely independent from any random variable $X$, by definition $\omega_{0,X}$ satisfies $G_X(z) = G_{0+X}(z) = G_0(\omega_{0,X}(z))$, and $G_0(z) = 1/z$, so we conclude that
\begin{equation*}
\omega_{X_{n,n},X_n^{(n)}}(z)\to 1/G_X(z) = F_X(z)
\end{equation*}
uniformly on compact subsets of $\mathbb{C}_+$ as $n\to\infty$.

As $X_n^\circ \rightharpoonup X^\circ$ and $X_{n,n}^\Box \rightharpoonup V_0$ by Lemma \ref{lem:Gcirc.props} and the hypotheses, we have $(W_n^\Box)^\bflat \rightharpoonup V_0^\flat$
by Lemma \ref{lem:bflat}, and taking the limit in \eqref{eq:pre.limit.fid.identity} along the (uniformly converging) subsequence yields \eqref{eq.lem:G.circ.hypothesis}, and \eqref{eq:R.thm.conclusion} by the equivalence provided by Theorem \ref{thm:G.R.corresp}. 
\end{proof}

We now prove Theorem \ref{thm:Compound.Poisson}.

\begin{myproof}{\it Theorem}{\ref{thm:Compound.Poisson}} We begin with the first conclusion, \eqref{eq:Rmusigma}.  By assumption $E[V_0^{-2}]<\infty$, and so $V_0$ is not identically $0$; therefore $E[V_0^2]>0$.  From Lemma \ref{lem:X.Xneg0.same}(\ref{Box.inverse.properties}, we may take $U:= V_0^{-\Box}$, in which case $U^\Box\equaldist V_0$ and $E[U^2] = 1/E[V_0^2]>0$.

Let $\lambda = \sigma^2/E[U^2]$, and for each $n$ let $\{X_{j,n}\}_{1\le j\le n}$ be freely independent random variables each with law $(1-\frac{\lambda}{n})\delta_0 + \frac{\lambda}{n}\mathcal{L}(U)$, and define $X_n = X_{1,n}+\cdots+X_{n,n}$.  Immediate calculation verifies that
\begin{equation} \label{eq:Xnn.moments} E[X_{n,n}^k] = \frac{\lambda}{n}E[U^k], \qquad k\ge 1 \end{equation}
and hence, from the free independence of the $X_{j,n}$, we have
\begin{equation} \label{eq:Var.Xn.1/n} \mathrm{Var}(X_n) = \lambda E[U^2] -\frac{\lambda^2}{n}E[U]^2 = \sigma^2 -\frac{\sigma^4 E[U]^2}{n\,E[U^2]^2}. \end{equation}

If these variables were {\em classically} independent, it is known (from antiquity) that $X_n$ converges in law (and all moments) to a compound Poisson with rate $\lambda$ and jump distribution $\mathcal{L}(U)$.  It therefore follows from \cite[Theorem 3.4]{BercoviciPata1999} that, in this free setting, $X_n\rightharpoonup M$ for some random variable $M$, and all moments of $X_n$ converge to the corresponding moments of $M$.  In particular, from \eqref{eq:Var.Xn.1/n} it follows that $\mathrm{Var}(M) = \sigma^2$.  

Finally, from Lemma \ref{lem:X.Xneg0.same}(\ref{lem:sqbiaa-part2} we have that $X_{n,n}^\Box\equaldist U^\Box = V_0$ for each $n \ge 1$, and hence this sequence certainly converges in distribution to $V_0$.  We have thus confirmed all of the hypotheses of Lemma \ref{lem:Xn.sums.of.Vn.}(\ref{lem:Xn.sums.of.Vn.1}), and we may conclude that 
\[ R_M(z) = m' + \sigma^2 E\left[\frac{z}{1-zV_0}\right] \]
where $m' = \lambda E[U] = \sigma^2 E[U]/E[U^2]$. Note that $R_{M-m'}(z) = R_M(z)+m'$.  Hence, with $M'=M-m'$ we have
\[ R_{M'}(z) = \sigma^2 E\left[\frac{z}{1-zV_0}\right]. \]
It follows that $M'$ is freely infinitely divisible: as $\sigma^2>0$ was chosen at the beginning, we may now construct $M_n'$ as above from $V_0$ and $\sigma^2/n$, and conclude that $R_{M'_n}(z) = \frac{1}{n}R_{M'}(z)$; this means that $\mathcal{L}(M'_n)^{\boxplus n} = \mathcal{L}(M')$.

Finally, take $X = M'+m$.  This random variable is also freely infinitely divisible, and satisfies \eqref{eq:Rmusigma}, as desired. 

The second claim, \eqref{eq:Rmusigma.sb}, can be shown in a parallel way, so we only provide the highlights. The random variable $V_+$ must be non-trivial, and by Lemma \ref{lem:X.Xneg0.same.sb}(\ref{Box.inverse.properties.sb}
$U:= V_+^{-s}$ is $L^1$ and $\ge 0$, in which case $U^s\equaldist V_+$ and $E[U] = 1/E[V_+]>0$.
Let $\lambda = m/E[U]$, and for each $n$ let $\{X_{j,n}\}_{1\le j\le n}$ and $X_n$ be defined as for the previous case. Identity \eqref{eq:Xnn.moments} continues to hold and we may again conclude that $X_n \rightharpoonup M$ for some $M$ with mean $m$. 
From Lemma \ref{lem:X.Xneg0.same.sb}(\ref{lem:sqbiaa-part2.sb} we have that $X_{n,n}^s\equaldist U^s = V_+$ for each $n \ge 1$, and have thus confirmed all of the hypotheses of Lemma \ref{lem:Xn.sums.of.Vn.}(\ref{lem:Xn.sum.of.Vn.2}), and we may conclude that 
\[ R_M(z) = mE\left[\frac{1}{1-zV_+}\right] \]
where $m = \lambda E[U]$. 
It follows that $M$ is freely infinitely divisible as we can follow this same construction with $m$ replaced by $m/n$, similar to the case above.

\end{myproof}

\begin{remark} \label{remark.scaling.inf.div} The following observation was used in the above proof, but is worth commenting on separately. Given any $R(z)$ of the form in \eqref{eq:R.thm.conclusion.free.p+}, replacing $m$ and $\sigma^2$ by $m/t$ and $\sigma^2/t$ respectively for any $t>0$ yields an $R$ transform via Proposition \ref{prop:ForAnyV}. Hence any such $R$ transform corresponds to a variable for which a possibly non-integer $t^{th}$ free convolution root exists.  In particular, such random variables are always freely infinitely divisible.
\end{remark}

\subsection{Free Infinite Divisibility} \label{sec:inf.div}
We now move beyond the compound free Poisson setting, which (as shown in the last section) corresponds to the case that the freely infinitely divisible random variable $X$ has a negative second moment.  We begin with Proposition \ref{prop:ForAnyV} and \ref{prop:ForAnyV.sb} which show that the putative L\'evy--Khintchine $R$-transform formulas \eqref{eq:R.thm.conclusion.free.p+} and \eqref{eq:R.thm.conclusion.free.p+.sb} really do define $R$-transforms of freely infinitely divisible random variables in the $L^2$ or positive $L^1$ settings.  This first step is accomplished by removing the restriction  that $E[V_0^{-2}]<\infty$ in \eqref{eq:Rmusigma} in Theorem \ref{thm:Compound.Poisson} by adapting the argument showing Lemma 1.1.3 in \cite{uwe} that proves an analog of this result in the classical case by constructing a sequence of variables approximating $V_0$ whose inverse second moment exists.

Following that step, we prove Theorems \ref{thm:non.compact.case} and \ref{thm:non.compact.case.sb}, the converse direction, completing the proof of our main result: any freely infinitely divisible ($L^2$ or positive $L^1$) random variable $X$ gives rise to another random variable ($V_0$ or $V_+$) satisfying the conditions of Theorem \ref{thm:G.R.corresp} or \ref{thm:G.R.corresp.sb}, arising as the limit of the square bias transforms of the free convolution roots of $X$. 

\begin{proposition}
\label{prop:ForAnyV}
For any random variable $V_0$, $m \in \mathbb{R}$ and $\sigma^2 \in [0,\infty)$, the function
\begin{align} \label{eq:R.thm.conclusion.free.p+}
R(z)= m + \sigma^2 E\left[
\frac{z}{1-zV_0}
\right]
\end{align}
is an $R$-transform that corresponds to a freely infinitely divisible $X$ with mean $m$ and variance $\sigma^2$.
\end{proposition}

\begin{proof} As in the proof of Theorem \ref{thm:Compound.Poisson}, using the translation property of $R$-transforms when adding constants, it suffices to prove the result for $m=0$. 

If $\sigma^2=0$ then $R(z)=0$, which is the $R$ transform of $\delta_0$, and the result holds trivially.  Now take $\sigma^2>0$. When $\mathcal{L}(V_0)=\delta_0$, in view of Example \ref{ex:semi.sats}, $R(z)$ is the $R$ transform of the semi-circle distribution, which is shown there to be freely infinitely divisible. 

Hence, we restrict to the case where $\sigma^2>0$ and $\mathcal{L}(V_0)\not =\delta_0$. For $k \ge 1$ let $D_k=\mathbb{R} \setminus [-c/k,c/k]$ where we take $c>0$ small enough so that $P(V_0 \in D_1)>0$.  Set $V_{0,k}:= \mathcal{L}(V_{0}|V_0\in D_k)$; in other words, the law of $V_{0,k}$ is determined by
\[ E[f(V_{0,k})] = \frac{E[f(V_0){\bf{1}}_{D_k}]}{P(V_0\in D_k)}. \]
Similarly, define $\mathcal{L}(V_{0,\neg 0})=\mathcal{L}(V_0|V_0\ne 0)$; the restrictions above and choice of $c$ makes all of these well defined laws, and $V_{0,k}\rightharpoonup V_{0,\neg 0}$.

Note that $|V_{0,k}|\ge c/k$ with probability $1$, and hence $E[V_{0,k}^{-2}]<\infty$. Let $\mathcal{L}(U_k)=\mathcal{L}(V_{0,k}^{-\Box})$  as in Definition \ref{def:X.Box.inverse}, so that $\mathcal{L}(U_k^\Box) = \mathcal{L}(V_{0,k})$  by Lemma \ref{lem:X.Xneg0.same}(\ref{Box.inverse.properties}.

By Lemma \ref{lem:FP.L2}, for any $\sigma^2\ge 0$ there is a compound free Poisson random variable $X_k$ whose $R$-transform satisfies
\begin{equation} \label{eq.RXk.intermediate}
R_{X_k}(z)= \sigma^2 E\left[\frac{z}{1-z U_k^\Box}\right] =  \sigma^2 E\left[\frac{z}{1-z V_{0,k}}\right].
\end{equation}
Hence, for any $k\in\mathbb{N}$ and $\sigma^2\ge 0$,
$$
\varphi^{k,\sigma^2}(z):=\varphi_{X_k}(z)=R_{X_k}(1/z) = \sigma^2 E\left[\frac{1}{z-V_{0,k}} \right] = \sigma^2 G_{V_{0,k}}(z)
$$
is a Voiculescu transform, whose domain manifestly has an analytic continuation to the whole complex plane, as it coincides on a truncated cone (an open set) with a scaled Cauchy transform.  As $V_{0,k} \rightharpoonup V_{0,\neg 0}$ as $k \rightarrow \infty$, by Proposition \ref{prop.Cauchy.robust}, $G_{V_{0,k}}$ converges to $G_{V_{0,\neg 0}}$ uniformly on compact subsets of $\mathbb{C}_+$, and hence $\varphi^{k,\sigma^2}$ converges uniformly on compact subsets of $\mathbb{C}_+$ to the analytic function $\varphi = \sigma^2 G_{V_{0,\neg 0}}$.  What's more, $|\varphi^{k,\sigma^2}(z)| = \sigma^2|G_{V_{0,k}}(z)| \le \sigma^2/|\mathrm{Im}(z)|$ (the last inequality comes from the fact that $|G(z)|\le 1/|\mathrm{Im}(z)|$ for any Cauchy transform $G$, as is easily verified using the triangle inequality in the definition \eqref{eq:def.G.int}), and hence $\varphi^{k,\sigma^2}(z) = o(z)$ uniformly in $k$.  Therefore, it follows from Proposition \ref{prop.uniform.Stolz} that the limit $\varphi$ is the Voiculescu transform of some probability measure, and so $R(z) = \varphi(1/z)$ is the $R$-transform of that probability measure.

Letting $\gamma = P(V_0 \not = 0)$ we have 
$$
\mathcal{L}(V_0)=\gamma \mathcal{L}(V_{0,\neg 0}) + (1-\gamma) \delta_0,
$$
and we may thus express
\begin{align} \label{eq.R.sum.limit.fp.ss}
R(z):=\sigma^2 E\left[ \frac{z}{1-zV_0}\right]  =  \gamma \sigma^2 E\left[ \frac{z}{1-zV_{0,\neg 0}}\right]+(1-\gamma) \sigma^2 z 
\end{align}
as the sum of two $R$ transforms, and thus conclude that $R(z)$ is an $R$ transform for all $V_0$ and $\sigma^2 \ge 0$.  The mean and variance can be read off directly from \eqref{eq:R.thm.conclusion.free.p+}.

The fact that $R=R_X$ for a freely infinitely divisible random variable $X$ now follows as in Remark \ref{remark.scaling.inf.div}.
\end{proof}

\begin{remark} Note that $R(z) = (1-\gamma)\sigma^2 z$ is the $R$-transform of a centered semicircular random variable. From \eqref{eq.R.sum.limit.fp.ss} we see directly that the infinitely divisible distribution constructed from the arbitrary $V_0$ is the free convolution of a semicircle with a limit of compound free Poisson distributions.  With the converse, Theorem \ref{thm:non.compact.case}, to Proposition {prop:ForAnyV} in hand, we see that {\em every} $L^2$ freely infinitely divisible distribution has such a decomposition.
\end{remark}

We similarly have the parallel result for the case of non-negative random variables with finite first moments, handled using size biasing. 
\begin{proposition} \label{prop:ForAnyV.sb}
For any non-negative random variable $V_+$ and $m \ge 0$ the function
\begin{align} \label{eq:R.thm.conclusion.free.p+.sb}
R(z)= m E\left[
\frac{1}{1-zV_+}
\right]
\end{align}
is an $R$-transform that corresponds to a positively freely infinitely divisible $X$ with mean $m$.
\end{proposition}

\begin{proof}
We may assume that $V_+\not =\delta_0$ and $m \ge 0$, since in these cases we have the $R$-transform of a non-negative constant random variable which is evidently positively freely infinitely divisible. Take $D_k=[c/k,\infty)$ for $c>0$ small enough that $P(V_+ \in D_1)>0$. Let $\mathcal{L}(V_{+,k})=\mathcal{L}(V_+|V_+ \in D_k)$, $\mathcal{L}(V_{+,\neg 0})=\mathcal{L}(V_+|V_+\not =0)$ so that $V_{+,k} \rightharpoonup V_{+,\neg 0}$. Noting that $V_{+,k}$ satisfies the conditions of Definition \ref{def:X.s.inverse}, let $\mathcal{L}(U_k)=\mathcal{L}(V_{+,k}^{-s})$.

Applying Lemma \ref{lem:FP.L1} we obtain, for any $m>0$, a compound free Poisson random variable $Y_k$ whose $R$-transform satisfies
$$
R_{Y_k}(z)=mE\left[ \frac{1}{1-zV_{+,k}} \right]
$$
Since $V_{+,k}$ is supported in $[0,\infty)$ (in fact in $[c/k,\infty)$), it follows from Remark \ref{rk:FP.positive} that $X_k$ is positively freely infinitely divisible.  Now following exactly the same argument as the proof of Theorem \ref{prop:ForAnyV} following \eqref{eq.RXk.intermediate}, we conclude, since $V_{+,k}\rightharpoonup V_{+,\neg 0}$, that $R(z) = m E[\frac{1}{1-zV_{+,\neg 0}}]$ is the $R$-transform of a random variable $Y$ which, as the weak limit of positively freely infinitely divisible random variables $Y_k$ is itself positively freely infinitely divisible (as the reader may quickly verify).

Letting $\gamma = P(V_+\not = 0)$ we have 
$$
\mathcal{L}(V_+)=\gamma \mathcal{L}(V_{+,\neg 0}) + (1-\gamma) \delta_0,
$$
and therefore we can express
\begin{align*}
R(z):=m E\left[ \frac{1}{1-zV_+}\right]= \gamma m E\left[ \frac{1}{1-zV_{+,\neg 0}}\right]+(1-\gamma) m
\end{align*}
as the sum of $R$ transforms of two positively freely infinitely divisible random variables, which therefore yields the desires conclusion that $R = R_X$ for a positively freely infinitely divisible random variable $X$, whose mean is $R_X(0) = m$.
\end{proof}

\begin{theorem} \label{thm:non.compact.case}
Let $X$ be a square integrable freely infinitely divisible random variable with mean $0$ and variance $\sigma^2 \in (0,\infty)$. For $n \ge 1$, let $\mathcal{L}(X_{n,n})$ be the distribution of the $n$ freely independent, identically distributed summands satisfying $\mathcal{L}(X_{n,n})^{\boxplus n} = \mathcal{L}(X)$.  Then the weak limit $V_0$ of $X_{n,n}^\Box$ exists and uniquely satisfies \eqref{eq:R.thm.conclusion.2}.
\end{theorem}

\begin{proof} Recall that $X_n^{(n)} = X-X_{n,n}$.  The ``replace one'' property, Theorem \ref{eq:Todd's.formula}, in this case says (cf. \eqref{eq:pre.limit.fid.identity}) 
\begin{align*}
G_{X^\circ}(z)=G_{X_{n,n}^\circ}(\omega_{X_{n,n},X_n^{(n)}}(z)).
\end{align*}

\noindent For readability, set $\omega_n(z) = \omega_{X_{n,n},X_n^{(n)}}(z)$. Using \eqref{G.sharp} defining $\bflat$ and \eqref{eq:GX.circ.EX=0} defining $\circ$, and squaring both sides, this yields
\[ \frac{1}{\sigma^2}(zG_X(z)-1) = \frac{1}{\omega_n(z)} G_{X_{n,n}^\Box}(\omega_n(z)) \]
or equivalently
\begin{equation}
\label{eq:conv.Xnnbox}
G_{X_{n,n}^\Box}(\omega_n(z)) = \frac{\omega_n(z)}{\sigma^2}(zG_X(z)-1).
\end{equation}
From Proposition \ref{prop:subordinator.def}, $\omega_n$ is uniquely determined 
by the equation
\begin{equation*} 
F_{X_{n,n}}(\omega_n(z)) = F_{X_{n,n}+X_n^{(n)}}(z) = F_X(z). \end{equation*}
From Proposition \ref{prop.Stolzify} applied in sequence to $F_{X_{n,n}}$ and $F_X$, for $z$ in an appropriate truncated cone $\Gamma$ this is equivalent to $\omega_n(z) = F_{X_{n,n}}^{-1}\circ F_X(z)$, and the range of this function contains another truncated cone $\Gamma'$.  (These truncated cones a priori vary with $n$; since $X_{n,n}\rightharpoonup 0$ by Lemma \ref{lem:sublemma}, it follows from Proposition \ref{prop.uniform.Stolz} that uniform truncated cones exist here.)

Let $w=F_{X_{n,n}}^{-1}\circ F_X(z)$; then for $z\in\Gamma'$ we equivalently have $z = F_X^{-1}\circ F_{X_{n,n}}(w)$.  Thus, for $w\in\Gamma'$, \eqref{eq:conv.Xnnbox} says
\begin{align*} G_{X_{n,n}^\Box}(w) &= \frac{w}{\sigma^2}\left(F_X^{-1}\circ F_{X_{n,n}}(w)\cdot G_X(F_X^{-1}\circ F_{X_{n,n}}(w))-1\right) \\
&= \frac{w}{\sigma^2}\left(\frac{F_X^{-1}\circ F_{X_{n,n}}(w)}{F_{X_{n,n}}(w)}-1\right)
\end{align*}
where the second equality comes from the fact that $G_X\circ F_X^{-1}(\zeta)=1/\zeta$.  As $X_{n,n}\rightharpoonup 0$, $F_{X_{n,n}}(w)$ converges to $F_{\delta_0}(w) = w$ uniformly on compact subsets of $\mathbb{C}_+$ by Proposition \ref{prop.Cauchy.robust}(\ref{robust.1}).  Hence
\begin{equation} \label{eq:conv. GXnn} G_{X^\Box_{n,n}}(w)\to \frac{w}{\sigma^2}\left(\frac{F_X^{-1}(w)}{w}-1\right) = \frac{1}{\sigma^2}\left(F_X^{-1}(w)-w\right) = \frac{1}{\sigma^2}\varphi_X(w) \end{equation}
uniformly on compact subsets of $\Gamma'$.

Taking stock: we have now used the ``replace one'' property, and the definitions of the square bias and free zero bias transforms, to show that the Cauchy transform of $X_{n,n}^\Box$ converges to the analytic function $\varphi_X/\sigma^2$ on some domain $\Gamma'\subseteq\mathbb{C}_+$.  We would like to use Proposition \ref{prop.Cauchy.robust}(\ref{robust.2}) to complete the proof of the theorem that $X_{n,n}^\Box$ converges in distribution to a probability measure.  This requires two more steps: first we must show that the convergence in \eqref{eq:conv. GXnn} occurs on all of $\mathbb{C}_+$, and then we must show that no mass is lost at infinity, that is, the resulting limit is a Cauchy transform with total mass $1$.

For each $n$, $G_{X^\Box_{n,n}}$ is analytic on all of $\mathbb{C}_+$; similarly, since $X$ is freely infinitely divisible, $\varphi_X$ is also analytic on all of $\mathbb{C}_+$ (cf.\ \cite[Theorem 5.10]{bercovici1993free}). Since $\Gamma'$ is an open subset of $\mathbb{C}_+$, it now follows from Montel's theorem that $G_{X^\Box_{n,n}}$ converges to $\varphi_X/\sigma^2$ on all of $\mathbb{C}_+$ (in fact uniformly on compact subsets).  We therefore conclude by Proposition \ref{prop.Cauchy.robust}(\ref{robust.2}) that $\varphi_X$ is the Cauchy transform of a finite measure $\nu$ with $\nu(\mathbb{R})\le 1$, and that the law of $X_{n,n}^\Box$ converges vaguely to $\nu$.  It remains only to show that $\nu(\mathbb{R})=1$ (since a vaguely convergent sequence of probability measures whose limit is a probability measure converges in distribution).  In other words, we merely need to show that $wG_\nu(w) = \frac{w}{\sigma^2}\varphi_X(w)\to 1$ as $w\to\infty$ along some ray in the upper half plane.

To motivate this last step, consider what happens in the special case that $X$ is compactly supported.  In this case, Speicher's theory of free cumulants \cite{nica2006lectures} gives a Taylor series expansion $R_X(z) = \sum_{n\ge0} \kappa_{n+1}[X]z^n$ where $\kappa_1[X] = E[X]=0$ and $\kappa_2[X]=\mathrm{Var}[X]=\sigma^2$.  Hence $\varphi_X(w) = \sigma^2/w + O(1/w^2)$, and thence it follows immediately that $\frac{w}{\sigma^2}\varphi_X(w)\to 0$ as $w\to\infty$.  To extend this argument rigorously to the more generally setting where $X$ need not be compactly supported, but only assumed to have two finite moments, we make use of some results due to Maassen \cite{Maassen}.

$F_X$ is the reciprocal Cauchy transform of a probability measure on $\mathbb{R}$ with finite variance $\sigma^2$ and mean $0$. Therefore, by \cite[Proposition 2.2(b)]{Maassen} it follows that $F_X(z) = z-\sigma^2 G_{\rho}(z)$ for some probability measure $\rho$ on $\mathbb{R}$.  Thus, for $w\in\Gamma$,
\[ w = F_X(F_X^{-1}(w)) = F_X^{-1}(w)-\sigma^2 G_\rho(F_X^{-1}(w)) \]
and hence
\begin{equation} \label{e.w.phi.mean} \varphi_X(w) = F_X^{-1}(w)-w = \sigma^2G_\rho(F_X^{-1}(w)) = \sigma^2 G_\rho(w+\varphi_X(w)). \end{equation}
Multiplying both sides by $(w+\varphi_X(w))/\sigma^2$ yields
\begin{equation} \label{e.w.phi} \frac{w}{\sigma^2}\varphi_X(w) + \frac{\varphi_X^2(w)}{\sigma^2} = (w+\varphi_X(w))G_\rho(w+\varphi_X(w)). \end{equation}
Now, again using only the fact that $X$ has mean $0$ and variance $\sigma^2$, \cite[Lemma 2.4]{Maassen} yields that $|\varphi_X(w)|\le 2\sigma^2/|\mathrm{Im}\, w|$.  Thus, along any path in $\Gamma$ with imaginary part tending to $\infty$, the left-hand-side of \eqref{e.w.phi} has the same limit as $\frac{w}{\sigma^2}\varphi_X(w)$, while the right-hand-side tends to $1$ (since $G_\rho$ is the Cauchy transform of a probability measure).  This concludes the proof that $\nu$ is a probability measure.

Let $V_0$ be a random variable whose law is $\nu$.  Hence, we have shown that $G_{X_{n,n}^\Box} \to G_{V_0}$ on $\mathbb{C}_+$.  As established in \eqref{eq:conv. GXnn}, we also have $G_{X_{n,n}^\Box}\to \varphi_X/\sigma^2$ on $\mathbb{C}_+$, and hence, we have shown that $G_{V_0} = \varphi_X/\sigma^2$ on $\mathbb{C}_+$, establishing \eqref{eq:R.thm.conclusion.2} (in the mean $0$ case that is the context of this theorem).  This concludes the proof.
\end{proof}

The calculations following \eqref{e.w.phi.mean} above give a useful way to compute the mean of any $L^2$ random variable from its Voiculescu transform.

\begin{corollary} \label{cor.mean.phi} If $X\in L^2$, then $\lim_{y\to\infty} \varphi_X(iy) = E[X]$.
\end{corollary}

\begin{proof} Since $\varphi_{X-E[X]}(z) = \varphi_X(z)-E[X]$, it suffices to prove the lemma in the case $E[X]=0$.  As shown just below \eqref{e.w.phi}, $|\varphi_X(iy)|\le 2\mathrm{Var}[X]/y$ for $y>0$, and hence for any probability measure $\rho$, $G_\rho(iy+\varphi_X(iy))\to 0$ as $y\to\infty$.  The result now follows from \eqref{e.w.phi.mean}.   
\end{proof}

We may now conclude the proof of our main Theorem \ref{main theorem 2}.

\medskip

\begin{myproof}{\it Theorem}{\ref{main theorem 2}}
Theorem \ref{thm:non.compact.case} shows that \ref{thm2.a} implies \ref{thm2.b} (by applying to the centered random variable $X-m$ and noting that $\varphi_{X-m}(z) = \varphi_X(z)-m$), Theorem \ref{thm:G.R.corresp} shows that \ref{thm2.b} implies \ref{thm2.c}, and the equivalence of \eqref{eq.lem:G.circ.hypothesis} and \eqref{eq:R.thm.conclusion} provided by Theorem \ref{thm:G.R.corresp} and 
Proposition \ref{prop:ForAnyV} show that \ref{thm2.c} implies \ref{thm2.a}.
\end{myproof}

\ignore{

We now return to the size bias. The assumption in Theorem \ref{thm:non.compact.case.sb} that $X$ be positively freely infinitely divisible, as specified in Definition \ref{def:pfid}, is necessary. Indeed, the theorem statement requires that $X_{n,n}^s$ exists. The following remark gives a canonical example of a non-negative, freely infinitely divisible random variable that is not positively freely infinitely divisible, and therefore whose summands cannot be size biased.
\begin{remark}\label{ex:semi.circle.leakage}
Let $S$ have the semi-circle distribution with mean 0 and variance 1; the support of $S$ is $[-2,2]$. In particular, the variable $\sigma S+m$ with $\sigma>0$ has mean $m$, variance $\sigma^2$ and support $[-2\sigma+m,2\sigma+m]$.
Let $X=\sigma(S+a)$ for $a \ge 2$, corresponding to the choice where $m=a \sigma$, and so $X$ has support $[(a-2)\sigma,(a+2)\sigma)]$, and is non-negative. Clearly $X$ is freely infinitely divisible and satisfies \eqref{eq:X.is.sum.X_{i,n}} with $W_n \equaldist \sigma S/\sqrt{n}+ a \sigma/n$. Consequently, $W_n$ has support $[\sigma(-2/\sqrt{n}+a/n), \sigma(2/\sqrt{n}+a/n)]$, with a negative left endpoint for all $n > a^2/4$.
\end{remark}
}

We now present the `size bias' version of Theorem \ref{thm:non.compact.case}.
\begin{theorem} \label{thm:non.compact.case.sb}
Let $X$ be a non-negative, positively freely infinitely divisible random variable with finite mean $m>0$. For $n \ge 1$, let $\mathcal{L}(X_{n,n})$ be the distribution of the $n$ freely independent, identically distributed summands satisfying $\mathcal{L}(X_{n,n})^{\boxplus n} = \mathcal{L}(X)$.  Then the weak limit $V_+$ of $X_{n,n}^s$ exists and uniquely satisfies \eqref{eq:R.thm.conclusion.2.sb}.
\end{theorem}

\begin{proof}
The proof follows that of Theorem \ref{thm:non.compact.case} with minor changes. To begin, let $\omega_n = \omega_{X_{n,n},X_n^{(n)}}$ be the subordinator used before, which is characterized by $F_{X_{n,n}}(\omega_n(z)) = F_X(z)$.  The ``replace one'' property for size bias, Theorem \eqref{eq:size.bias.change.1.formula}, in the present case yields, by \eqref{eq:free.replace.first.one.sb},
\[ G_{X^s}(z) = G_{X_{n,n}^s}(\omega_n(z)). \]
Lemma \ref{lem:X.Xneg0.same.sb}(\ref{lem.Box.G.sb} for the Cauchy transform of a size bias thus yields
\begin{equation} \label{eq.subordinator.calc.5.16} G_{X_{n,n}^s}(\omega_n(z)) = \frac{zG_X(z)-1}{m}. \end{equation}
But $G_X(z) = 1/F_X(z) = 1/F_{X_{n,n}}(\omega_n(z))$.  Using the same (uniform over $n$) truncated cones $\Gamma$ and $\Gamma'$ from the proof of Theorem \ref{thm:non.compact.case}, we then express $\omega_n(z) = F_{X_{n,n}}^{-1}\circ F_X(z)$ for $z\in\Gamma$, and then express \eqref{eq.subordinator.calc.5.16} as
\[ G_{X_{n,n}^s}(w) = \frac{1}{m}\left(\frac{F_X^{-1}\circ F_{X_{n,n}}(w)}{F_{X_{n,n}}(w)}-1\right) \]
for $w=\omega_n(z)\in\Gamma'$.  As before, since $X_{n,n}\rightharpoonup 0$, $F_{X_{n,n}}(w)\to w$ uniformly on compact sets, and it follows just as before that
\begin{equation} \label{eq:conv. GXnn.sb} G_{X_{n,n}^s}(w)\to \frac{1}{m}\left(\frac{F_X^{-1}(w)}{w}-1\right) = \frac{1}{mw}(F_X^{-1}(w)-w) = \frac{1}{mw}\varphi_X(w) \end{equation}
uniformly on compact subsets of $\Gamma'$ as $n\to\infty$.  The remainder of the proof after \eqref{eq:conv. GXnn.sb} follows mutatis mutandis from the proof of Theorem \ref{thm:non.compact.case} following \eqref{eq:conv. GXnn}.
\end{proof}

\begin{myproof}{\it Theorem}{\ref{main theorem 3}}
Theorem \ref{thm:non.compact.case.sb} shows that \ref{thm2.a} implies \ref{thm2.b}, Theorem \ref{thm:G.R.corresp.sb} shows that \ref{thm2.b.sb} implies \ref{thm2.c.sb}, and the equivalence of 
\eqref{eq.lem:G.circ.hypothesis.sb} and \eqref{eq:R.thm.conclusion.sb} given by 
Theorem \ref{thm:G.R.corresp.sb}, and Proposition \ref{prop:ForAnyV.sb}, show that \ref{thm2.c.sb} implies \ref{thm2.a.sb}.
\end{myproof}
\vspace{0.3cm}

\begin{myproof}{\it Theorem}{\ref{prop:L2.pos.fee.inf.div}}
For the forward direction, suppose \eqref{L2.pos.(1)} holds, and let $X'=X+\alpha$ be a translation of $X$ which is positively freely infinitely divisible, and further large enough that $E[X']>0$.  We restrict to the case $\mathrm{Var}[X']>0$ (otherwise $X$ and $X'$ are constant and the result is straightforward).  For such $X'$, both Theorems \ref{main theorem 2} and \ref{main theorem 3} apply.  Note that $V_0=V_0(X)=V_0(X')$ by Remark \ref{remark.translations.1}.  Let $V_+ = V_+(X')$. Since  both \eqref{eq:LK0} and \eqref{eq:LK+} hold for $X'$, we have
\begin{equation*}
 E[X']+\mathrm{Var}[X']G_{V_0}(z)=E[X']zG_{V_+}(z) \end{equation*}
and thence
\begin{equation}\label{eq:equate0+trans}
z\left(zG_{V_+}(z)-1\right) = \frac{\mathrm{Var}[X']}{E[X']}zG_{V_0}(z).
\end{equation}
Applying Lemma \ref{lem.Cauchy.moment} to the limit as $z \rightarrow \infty$ in \eqref{eq:equate0+trans}, we find that $V_+ \in L^1$ with mean $E[V_+]=\mathrm{Var}[X']/E[X']$. We may now rewrite \eqref{eq:equate0+trans} to obtain
$$
G_{V_0}(z)=\frac{zG_{V_+}(z)-1}{E[V_+]},
$$
and via Lemma \ref{lem:X.Xneg0.same.sb}(\ref{lem.Box.G.sb}, recognize that $V_0=V_+^s$. In particular, $V_0 \ge 0$ and by Lemma
\ref{lem:X.Xneg0.same.sb}(\ref{size.bias.inverse.pair} it possesses a finite inverse first moment as claimed, thus verifying condition \eqref{L2.pos.(2)} and the remaining conclusions of the theorem.

For the reverse direction, suppose \eqref{L2.pos.(2)} holds.  Since $V_0$ is non-negative and has an inverse first moment, the distribution $V:=V_0^{-s}$ exists and satisfies $V^s \equaldist (V_0^{-s})^s \equaldist V_0$ by Lemma \ref{lem:X.Xneg0.same.sb}(\ref{Box.inverse.properties.sb}. Noting that $V=V_0^{-s}$ cannot be trivial and hence $E[V]$ must be positive, by Lemma \ref{lem:X.Xneg0.same.sb}(\ref{lem.Box.G.sb}
the size bias relation between $V$ and $V_0$ can be expressed in terms of Cauchy transforms as
\begin{align*} 
\frac{zG_{V}(z)-1}{E[V]}=G_{V_0}(z).
\end{align*}
Applying  Theorem \ref{main theorem 2}(\ref{thm2.b} to $X$ now yields
\begin{align*}
\varphi_X(z) = E[X] + \mathrm{Var}[X] G_{V_0}(z)
&= E[X] + \mathrm{Var}[X] \left( \frac{zG_{V}(z)-1}{E[V]}\right) \\
&= E[X]-\frac{\mathrm{Var}[X]}{E[V]}+ \frac{\mathrm{Var}[X]}{E[V]}zG_{V}(z).
\end{align*}
Now, set $X' = X-E[X]+\mathrm{Var}[X]/E[V]+\alpha$ for some $\alpha>0$ large enough that $X'\ge 0$ (which must exists since $X$ is assumed to be bounded below).   Then
\[ \varphi_{X'}(z) = \alpha + \beta zG_V(z) \]
where $\beta = \mathrm{Var}[X]/E[V]>0$.

Noting that the point mass $\delta_0$ has Cauchy transform $G_{\delta_0}(z) = 1/z$, we reformulate this as
\[ \varphi_{X'}(z) = \alpha zG_{\delta_0} + \beta z G_V(z)
= (\alpha+\beta)z\left[\frac{\alpha}{\alpha+\beta}G_{\delta_0}(z) + \frac{\beta}{\alpha+\beta}G_V(z)\right] \]
and this convex combination inside the brackets is the Cauchy transform $G_{V_+}(z)$ of a random variable $V_+$ whose law is a mixture of $\delta_0$ and $\mathcal{L}(V)$:
\begin{equation} \label{eq:V+.mixture}
\mathcal{L}(V_+) = \frac{\alpha}{\alpha+\beta}\delta_0 + \frac{\beta}{\alpha+\beta}\mathcal{L}(V).
\end{equation}
Hence, we conclude that
\begin{equation}\label{eq:L2.pos.final} \varphi_{X'}(z) = (\alpha+\beta)zG_{V_+}(z) \end{equation}
which, by Theorem \ref{main theorem 3} together with the positivity of $X'$, proves that $X'$ is positively freely infinitely divisible.  Since $X'$ is a translate of $X$, this proves part (\ref{L2.pos.(1)}) holds true.

Since $V_+$ is a mixture of $V$ with a point mass at $0$, by Lemma \ref{lem:X.Xneg0.same.sb}(\ref{lem:sqbiaa-part2.sb}, $(V_+)^s \equaldist V^s \equaldist V_0$ as claimed.  Finally, taking $z=iy$ and sending $y\to\infty$ in \eqref{eq:L2.pos.final} shows, by Corollary \ref{cor.mean.phi}, that $E[X'] = \alpha+\beta$. Hence \eqref{eq:V+.mixture} shows that $E[V_+] = \frac{\beta}{\alpha+\beta}E[V] = \beta E[V]/E[X']$.  Since $\beta = \mathrm{Var}[X]/E[V]$, we have
\[ E[V_+] = \frac{\beta}{\alpha+\beta}E[V] = \frac{\mathrm{Var}[X]}{\alpha+\beta} = \frac{\mathrm{Var}[X]}{E[X']} = \frac{\mathrm{Var}[X']}{E[X']}\]
yielding the final statement of the theorem, concluding the proof.
\end{myproof}

\begin{remark}
When $X$ is supported on a compact set, we have the following probabilistic way of obtaining the correspondence of Theorem \ref{prop:L2.pos.fee.inf.div}, that is, 
$$
V_0=\lim_{n \rightarrow \infty} (X_{n,n}-m/n)^\Box=\lim_{n \rightarrow \infty} X_{n,n}^\Box = \lim_{n \rightarrow \infty}(X_{n,n}^s)^s=V_+^s,
$$
where the first inequality holds due to Theorem \ref{thm:non.compact.case}, the second by of Lemma \ref{lem:X.Xneg0.same}(\ref{lem:sqbiaa-part3}, the third by the definition of $\Box$, and the last by Theorem 
\ref{thm:non.compact.case.sb} and Lemma \ref{lem:X.Xneg0.same.sb}(\ref{lem:sqbiaa-part3.sb}; compactness is only used to ensure the convergence in expectation when invoking Lemmas \ref{lem:X.Xneg0.same} and \ref{lem:X.Xneg0.same.sb} in the second and final equality.

In addition, as $X_{n,n} \rightharpoonup 0$ and $E[X_{n,n}] \to 0$, we have 
\begin{multline*}
E[V_+]=\lim_{n\rightarrow \infty}E[X_{n,n}^s]=\lim_{n\rightarrow \infty}E[X_{n,n}^s-X_{n,n}]
\\
=\lim_{n\rightarrow \infty}\left( \frac{E[X_{n,n}^2]}{E[X_{n,n}]}-E[X_{n,n}]\right)=\lim_{n\rightarrow \infty} \frac{{\rm Var}(X_{n,n})}{E[X_{n,n}]}=\frac{\mathrm{Var}[X]/n}{E[X]/n},
\end{multline*}
where we have used the size bias relation \eqref{eq:def.mean.m.szb} for the third equality.
\end{remark}

\subsection{The Compactly-Supported Case\label{sect:compact}}

The proofs of Theorems \ref{thm:non.compact.case} and \ref{thm:non.compact.case.sb} use \cite[Theorem 5.10]{bercovici1993free}: the fact that the Voiculescu transform $\varphi_X$ of a freely-infinitely divisible random variable $X$ has an analytic continuation to the whole upper half plane; the continuation is given by the L\'evy--Khintchine formula itself, as in the statement of Theorem \ref{main theorem 2}(\ref{thm2.b}.  We can avoid requiring a priori knowledge of this important fact, if we restrict to the case that $X$ is compactly-supported.  To do so, we need the following ancillary result, which is of some independent interest, and may well be known in the literature but not in literature known to the authors.

\begin{lemma}\label{lem:uniform.compact.support}
Let $\mu$ be a compactly-supported probability measure on $\mathbb{R}$.  There is a compact interval $J$ that contains
the supports of all $\boxplus$-convolution roots of $\mu$.  That is: for any $n\in\mathbb{N}$, if $\nu$ is a probability measure satisfying $\nu^{\boxplus n} = \mu$, then $\nu$ is compactly supported with $\mathrm{supp}\,\nu\subseteq J$.
\end{lemma}

\begin{proof} To begin, note that we may safely assume $\mu$ is centered (which implies that $\nu$ is centered); otherwise we can center it by translating by the mean, which distributes to translating $\nu$ by its mean ($1/n$ times the mean of $\mu$), then apply the result in the centered case and translate the intervals back.

So, let $\mu$ be a centered probability distribution, with support contained in a compact interval $[-R,R]$,  and let $\sigma^2$ denote its variance.  For $n\ge 1$, since $\mu = \nu\boxplus\nu^{\boxplus (n-1)}$, we have the subordination relation $F_\mu(z) = F_\nu(\omega_n(z))$ where $\omega_n$ is the subordinator from $\nu$ to $\nu^{\boxplus (n-1)}$.  In \cite[Theorem 2.5(3)]{BelBer}, the authors show that this subordinator has an elegant form:
\begin{equation} \label{e.omegan} \omega_n(z) = \frac{z}{n}+\left(1-\frac1n\right)F_\mu(z). \end{equation}
(See also \cite[Section 2.3]{JCWang} for a concise presentation.)  For $n\ge 2$, $F_\mu$ =$F_{\nu^{\boxplus n}}$ has a continuous extension to $\mathbb{C}_+\cup\mathbb{R}$ \cite[Theorem 2.5(4)]{BelBer}, so we may discuss the behavior of $F_\mu$ and $F_\nu$ on $\mathbb{R}$.  From the Stieltjes inversion formula, $F_\mu$ is real-valued on $\mathbb{R}\setminus [-R,R]$.

By \cite[Proposition 2.2(b)]{Maassen}, since $F_\mu$ is the reciprocal Cauchy transform of a centered measure with variance $\sigma^2<\infty$, there is a probability measure $\rho$ on $\mathbb{R}$ such that
\[ F_\mu(z) = z-\sigma^2G_{\rho}(z), \qquad z\in\mathbb{C}_+ \]
and the continuity of $F_\mu$ on $\mathbb{C}_+\cup\mathbb{R}$ extends this relation similarly.  Note that $F_\mu$ is real-valued on $\mathbb{R}$ prescisely where $G_{\rho}$ is, and thus $\rho$ is also supported in $[-R,R]$.  Equation \ref{e.omegan} therefore becomes
\begin{equation*}
\omega_n(z) = z - \left(1-\frac1n\right)\sigma^2 G_\rho(z), \qquad z\in\mathbb{C}_+\cup\mathbb{R}. \end{equation*}
Outside $[-R,R]$, the function $G_{\rho}$ is analytic, and for $x>R$,
\[ |G_{\rho}'(x)| = \int_{-R}^R \frac{\rho(dt)}{(x-t)^2} \le \frac{1}{(x-R)^2}. \]  
In particular $G_\rho'(x)\to 0 $ as $x\to\infty$, and so for $x>R$,
\[ |\omega_n(x)-x| = \left(1-\frac{1}{n}\right)\sigma^2 \left|\int_x^\infty G'_\rho(y)\,dy\right| \le \sigma^2 \int_x^\infty \frac{1}{(y-R)^2}\,dy = \frac{\sigma^2}{x-R}. \]
A similar analysis for $x<-R$ shows that $|\omega_n(x)-x| \le \sigma^2/|x+R|$ there.

In particular, if $x>R+\sigma^2$, $|\omega_n(x)-x|<1$.  Since $\omega_n$ is continuous, it follows that its image on $[R+\sigma^2,\infty)$ contains $[R+\sigma^2+1,\infty)$ (and derivative calculations show that it is one-to-one there as well).

Hence every point $y\in[R+\sigma^2+1,\infty)$ has the form $y=\omega_n(x)$ for some $x\in[R+\sigma^2,\infty)$ with $|y-x|<1$, and so $F_\mu(x) = F_\nu(\omega_n(x)) = F_\nu(y)$ shows that $F_\nu(y)$ is real since $F_\mu$ is real valued off $[-R,R]$.  A similar argument shows that $F_\nu$ is real-valued on $[-\infty,-R-\sigma^2-1)$.  This shows that the support of $\nu$ is contained in $[-R-\sigma^2-1,R+\sigma^2+1]$, as desired.
\end{proof}

With this lemma in hand, we can provide a much simpler proof of Theorems \ref{thm:non.compact.case} and \ref{thm:non.compact.case.sb} in the case of compact support.  We focus on Theorem \ref{thm:non.compact.case}; the alternate treatment Theorem \ref{thm:non.compact.case.sb} is entirely analogous.

\medskip

\begin{myproof}{\it Theorem}{\ref{thm:non.compact.case} {\it in the case that $X$ has compact support}}

\ignore{
\begin{theorem}\label{thm:fid.compact.case}
Let $X$ be a square integrable, compactly supported freely infinitely divisible random variable with mean $m \in \mathbb{R}$ and variance $\sigma^2 \in (0,\infty)$.  Then for $n \ge 1$ letting $\mathcal{L}(X_{n,n})$ be the distribution of the $n$ freely independent, identically distributed summands satisfying $\mathcal{L}(X_{1,n})^{\boxplus n} = \mathcal{L}(X)$, the weak limit $V_0$ of $X_{n,n}^\Box, n\ge 1$ exists and uniquely satisfies \eqref{eq.lem:G.circ.hypothesis} and \eqref{eq:R.thm.conclusion}.
\end{theorem}
}

Since $X$ has compact support, by Lemma \ref{lem:uniform.compact.support}, the support of the summand distributions $X_{n,n}, n \ge 1$ in \eqref{def:Xn}, with $X$ replacing $X_n$, are contained in a compact interval $J$. As the distribution of $X_{n,n}^\Box$ is absolutely continuous with respect to that of $X_{n,n}$ by \eqref{def:Xbox}, the supports of $X_{n,n}^\Box, n \ge 1$ are also contained in $J$.  Any sequence of probability measures supported in the fixed compact interval $J$ is tight;  hence every subsequence of $\mathcal{L}(X_{n,n}^\Box)$ has a limit $\mathcal{L}(V_0)$ (which a priori depends on the subsequence). As $X_n$ in \eqref{def:Xn} equals $X$ for all $n \ge 1$, the remaining hypotheses of Lemma \ref{lem:Xn.sums.of.Vn.} hold trivially, and invoking the lemma shows that \eqref{eq:R.thm.conclusion.2} holds for the pair $X,V_0$, and Theorem \ref{thm:G.R.corresp} shows $V_0$ is unique. Therefore, each subsequence has a convergent subsequence and all subsequential limits are the same $V_0$; it follows the full sequence must also converge to $V_0$.
\end{myproof}

\section{New Examples of freely infinitely divisible Distributions\label{sect:examples}}

In this final section, we use Theorems \ref{main theorem 2} and \ref{main theorem 3} to (computationally effectively) produce previously unknown examples of freely infinitely divisible distributions.  We begin with our new formulation --- Theorem \ref{main theorem 2}(\ref{thm2.c} --- of Bercovici--Voiculescu's free L\'evy--Khintchine formula, which categorizes the laws of all $L^2$ freely infinitely divisible random variables in terms of their free L\'evy--Khintchine measure $V_0$ and their free zero bias $X^\circ$.

\subsection{Using the Free Zero Bias}

Proposition \ref{prop.any.V.will.do} shows that for every law $\mathcal{L}(V_0)$, every $m \in \mathbb{R}$ and every $\sigma^2 \in (0,\infty)$ there exists a freely infinitely divisible $X$ with whose Voiculescu transform is given by \eqref{eq:LK0} (i.e.\ the free L\'evy--Khintchine formula).  Our main Theorem \ref{main theorem 2} gives an equivalent characterization in terms of the free zero bias, \eqref{eq:free.Schmock}.  Indeed,
from the definition \eqref{eq:def.mean.m.fzb} of the free zero bias, noting that the contribution of the mean is simply a translation of the measure, \eqref{eq:free.Schmock} yields
\begin{align*} G_{(X-m)^\circ}(z+m) = G_{(X-m)^\circ+m}(z) &= G_{X^\circ}(z) \\
&= G_{V_0^\bflat}(1/G_X(z)) = G_{V_0^\bflat}(1/G_{X-m}(z+m)). \end{align*}
Squaring both sides and using the definition \eqref{G.sharp} of the El Gordo transform then gives the equality
\begin{equation*}
\frac{1}{\sigma^2}\left((z+m)G_{X-m}(z+m)-1\right) = G_{X-m}(z+m)G_{V_0}(1/G_{X-m}(z+m)).
\end{equation*}
and now noting that $G_{X-m}(z+m) = G_X(z)$, we have for each $V_0$ and $\sigma^2$ a one-parameter family of solutions for $G_X(z)$ that are translates by $m$:
\begin{equation} \label{eq:examples.end}
(z+m)G_X(z)-1 = \sigma^2\,G_X(z)G_{V_0}(1/G_X(z)).
\end{equation}
We will use \eqref{eq:examples.end} in the examples below, and restrict to the $m=0$ case, all other solutions being translations.

\begin{remark} \label{remark.symmetry} The free L\'evy--Khintchine formula \eqref{eq:LK0} shows that the measure $\mathcal{L}(V_0)$ is symmetric iff $\mathcal{L}(X-m)$ is symmetric. Indeed, $\mathcal{L}(X)$ is symmetric if and only if, for $z\in\mathbb{C}\setminus\mathbb{R}$
\[ G_X(-z) = E\left[\frac{1}{-z-X}\right] = -E\left[\frac{1}{z+X}\right] = -E\left[\frac{1}{z-X}\right] = -G_X(z) \]
where the penultimate equality follows from the assumption $\mathcal{L}(X)=\mathcal{L}(-X)$; the reverse implication follows from the Stieltjes inversion formula, cf.\ Proposition \ref{prop.Cauchy.robust}.  Reciprocating, the same holds for $F_X$, and using \eqref{def:Voiculescu.trans}, it follows that $\mathcal{L}(X)$ is symmetric iff $\varphi_X(-z) = -\varphi_X(z)$ on its domain.  Since $\varphi_X(z)-m = \varphi_{X-m}(z)$, the equivalence of symmetry for $X-m$ and $V_0$ follows from \eqref{eq:LK0}.  This equivalence will be useful in several examples below.
\end{remark}

\begin{example} Let $V_0$ have the centered semicircle distribution with variance $t$, $G_{V_0}(z) = \frac{1}{2t}(z-\sqrt{z^2-4t})$.  From \eqref{eq:examples.end}, we find that the associated freely infinitely divisible random variable $X$ with mean $0$ and variance $\sigma^2$ satisfies
\begin{align*} zG_X(z)-1 &=\sigma^2G_X(z)\frac{1}{2t}\left(\frac{1}{G_X(z)}-\sqrt{\frac{1}{G_X^2(z)}-4t}\right)\\
&=\frac{\sigma^2}{2t}\left(1-G_X(z)\sqrt{\frac{1}{G_X^2(z)}-4t}\right).
\end{align*}
Isolating the square root on the right hand side and then squaring both sides, we obtain the quadratic equation
\[ \left(\frac{2t}{\sigma^2}(zG_X(z)-1)-1\right)^2 = 1-4tG_X^2(z), \]
which simplifies to
\begin{equation} \label{e.Y.semicircle}
(tz^2+\sigma^4)G_X^2(z) - (2t+\sigma^2)zG_X(z) + \sigma^2+t=0.
\end{equation}
Solving, and choosing the negative square root to produce the required cancellation (to yield $\lim_{z\to\infty}zG_X(z) = 1$), we obtain
\begin{align*}
    G_X(z) &= \frac{(2t+\sigma^2)z - \sqrt{(2t+\sigma^2)^2z^2-4(\sigma^2+t)(tz^2+\sigma^4)}}{2(tz^2+\sigma^4)}\\
    &= \frac{(2t+\sigma^2)z- \sigma^2\sqrt{z^2-4(\sigma^2+t)}}{2(tz^2+\sigma^4)}.
\end{align*}
As te denominator is strictly positive for all $z=x\in\mathbb{R}$, this function has a continuous extension to $\mathbb{C}_+\cup\mathbb{R}$; taking $z=x\in\mathbb{R}$ we see that $G_X(x)$ is purely real if $x^2\ge 4(\sigma^2+t)$, and hence the density is given by
\begin{equation} \label{e.density.Levy.semicircle}
\varrho_X(x) = -\frac{1}{\pi}\lim_{y\downarrow 0}G_X(x+iy)=
\frac{\sigma^2 \sqrt{4(\sigma^2+t)-x^2}}{2\pi(tx^2+\sigma^4)} \mathbbm{1}_{|x|\le 2\sqrt{\sigma^2+t}}.
\end{equation}
Note that we recover the semicircle density as the law of $X$ when $t=0$, i.e.\ when the free L\'evy--Khintchine measure $\mathcal{L}(V_0) = \delta_0$.  Indeed, in this case Equation \eqref{e.Y.semicircle} reduces to $\sigma^2 G_X^2(z)- G_X(z) + 1 = 0$ which is the standard quadratic equation for the Cauchy transform of the semicircle law.  This aligns with Example \ref{ex:semi.sats}.

\begin{figure}
  \includegraphics[scale=0.45]{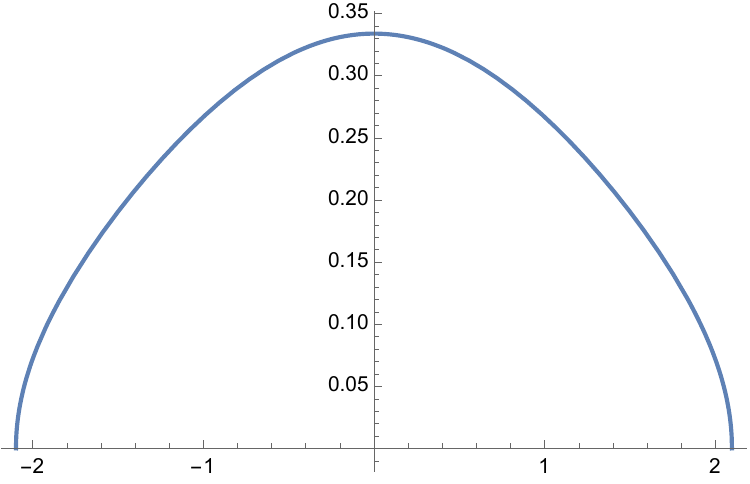}
  \includegraphics[scale=0.45]{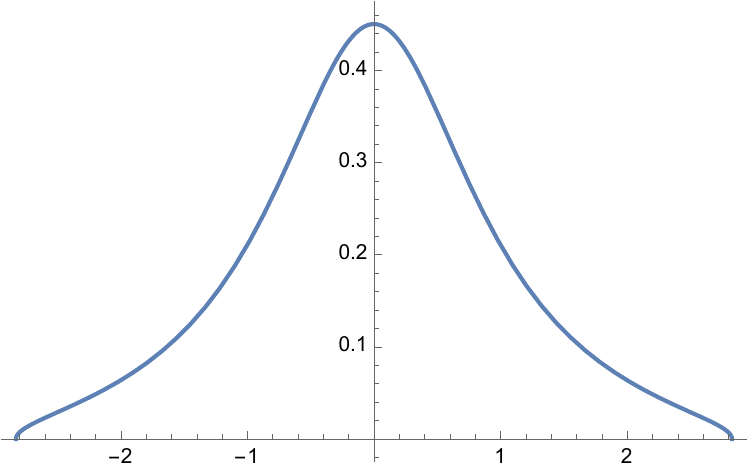}
  \caption{The density \eqref{e.density.Levy.semicircle} of the freely infinitely divisible law with mean $0$, variance $1$, and free L\'evy--Khintchine measure $\mathcal{L}(V_0)$
  a semicircle of variance $t=0.1$ on the left and $t=1$ on the right.}
  \label{fig.5.13}
\end{figure}

\end{example}

\ignore{
\begin{example} Let $Y\equaldist\delta_0$; then $G_Y(w) = \frac{1}{w}$.  Hence \eqref{eq:examples.end} becomes
\[ zG_X(z)-1 = \sigma^2G_X(z)G_X(z) = \sigma^2 G_X^2(z). \]
This well-known quadratic equation has, as solution, the Cauchy transform of the semicircle law of variance $\sigma^2$:
\[ G_X(z) = \frac{z-\sqrt{z^2-4\sigma^2}}{2\sigma^2}. \]
This aligns with Example \ref{ex:semi.sats}.
\end{example}
}

\begin{example} \label{ex:U.is.Rad} Let $V_0 \equaldist\frac12(\delta_1+\delta_{-1})$, a Rademacher distribution.  Here
\[ G_{V_0}(w) = \frac12\left(\frac{1}{w-1}+\frac{1}{w+1}\right) = \frac{w}{w^2-1}. \]
In this case \eqref{eq:examples.end} becomes
\[ zG_X(z)-1 = \sigma^2G_X(z)\cdot\frac{1/G_X(z)}{(1/G_X(z))^2-1} \]
which simplifies to the cubic equation
\begin{equation} \label{e.final.cubic}
zG_X^3(z)+(\sigma^2-1)G_X^2(z)-zG_X(z)+1=0.
\end{equation}
While analytic expressions for the general solution to this cubic are tractable, they become much simpler in the case $\sigma^2=1$, where \eqref{e.final.cubic} becomes the depressed cubic equation
\begin{equation} \label{e.depressed} G_X(z)^3-G_X(z)+\frac{1}{z}=0. \end{equation}

For the general depressed cubic equation $w^3+pw+q=0$ with $p,q\in\mathbb{C}$, the three solutions $w_0,w_{\pm 1}$ can be expresses as follows.  Let $\xi= e^{2\pi i/3} = \frac{-1+\sqrt{3}i}{2}$ be the cube root of unity in $\mathbb{C}_+$; then
\begin{equation} \label{e.w.roots} w_k = \xi^kr-\frac{p}{3\xi^k r}, \qmq{where}  r = \sqrt[3]{-\frac{q}{2}+\sqrt{\frac{q^2}{4}+\frac{p^3}{27}}}. \end{equation}
Here the square and cube roots have branch cuts along the positive real axis (i.e.\ $(re^{i\theta})^{1/n} = r^{1/n} e^{i\theta/n}$ where $\theta\in[0,2\pi)$).  For our case of interest \eqref{e.depressed}, $p=-1$ and $q=1/z$, so
\begin{equation} \label{e.def.r.cubic} r = r(z) = \sqrt[3]{-\frac{1}{2z}+\sqrt{\frac{1}{4z^2}-\frac{1}{27}}}=:\sqrt[3]{-a(z)+b(z)}. \end{equation}
To determine which of the three roots is the Cauchy transform we seek, we look at asymptotics: using Newton's Binomial theorem (twice), we compute that 

\[ r(z) = -\sqrt{3}i +\frac{3}{2}\cdot\frac{1}{z}+O\left(\frac{1}{z^2}\right) \qmq{and}
\frac{1}{r(z)} = \frac{i}{\sqrt{3}} + \frac{1}{2}\cdot\frac{1}{z} +O\left(\frac{1}{z^2}\right). \]
Hence, \eqref{e.w.roots} yields
\[ w_k = w_k(z) = (\xi^{k}-\xi^{-k})\frac{i}{\sqrt{3}} + (\xi^{k}+\xi^{-k})\frac{1}{2z} + O\left(\frac{1}{z^2}\right). \]
Note that if $k=\pm 1$ then $\xi^k\ne\xi^{-k}$, and so $\lim_{z\to\infty} w_k(z) = (\xi^{k}-\xi^{-k}) \ne 0$, hence $w_k$ is not a Cauchy transform.  On the other hand, we see that $w_0(z) = \frac{1}{z} + O(1/z^2)$, exhibiting the correct asymptotic behavior for a Cauchy transform of a probability measure.  Since we know that one of the roots {\em is} $G_X(z)$, we conclude that it is $w_0(z)$, i.e.
\begin{align}\label{eq:GAzadiTower}
G_X(z) = r(z) + \frac{1}{3r(z)}
\end{align}
with $r(z)$ as given in \eqref{e.def.r.cubic}.

Note that the quantity inside the cube root in $r(z)$ is never $0$; as such, the function $r$ has a continuous extension to $\mathbb{C}_+\cup(\mathbb{R}\setminus\{0\})$.  Thus $\mathcal{L}(X)$ is a combination of a density and potentially a point mass at $0$.

Referring to \eqref{e.def.r.cubic},
\[ r(x) = \sqrt[3]{-a(x)+b(x)} = \sqrt[3]{-(a(x)-b(x))}.
\]
Take $0 < x \le \frac32\sqrt{3}$, in which case $1/4x^2-1/27\ge 0$ and $a(x)-b(x)>0$.  Given the choice of branch cut for the cube root, for a positive number $\delta$, $\sqrt[3]{-\delta} = e^{i\pi/3}\sqrt[3]{\delta} = \frac{1+\sqrt{3}i}{2}\sqrt[3]{\delta}$ and hence
\begin{equation} \label{e.Im.r.1} -\frac{1}{\pi}\mathrm{Im}\,r(x) = -\frac{\sqrt{3}}{2\pi}
\sqrt[3]{a(x)-b(x)}.
\end{equation}
An analogous argument applied to $1/r(x)$, noting that $e^{-i\pi/3} = \frac{1-\sqrt{3}i}{2}$, yields
\begin{equation} \label{e.Im.r.2} -\frac{1}{\pi}\mathrm{Im}\,\frac{1}{3r(x)} = \frac{\sqrt{3}}{2\pi}\cdot\frac13
\sqrt[-3]{a(x)-b(x)},
\qquad 0<x \le  \frac{3}{2}\sqrt{3}. \end{equation}
Notice that, in this interval,
\begin{equation} \label{e.cubic.cancel.1} 
\sqrt[3]{a(x)-b(x)}\sqrt[3]{a(x)+b(x)}=\sqrt{a^2(x)-b^2(x)}
= \frac13. 
\end{equation}
Hence
\[ 
\frac13
\sqrt[-3]{a(x)+b(x)}=\sqrt[3]{a(x)-b(x)}.
\]
In light of \eqref{eq:GAzadiTower}, we can thence combine \eqref{e.Im.r.1} and \eqref{e.Im.r.2}, and apply the symmetry principle of Remark \ref{remark.symmetry}, to yield the density 
\begin{equation} \label{e.cubic.density}
\varrho_X(x) = \frac{\sqrt{3}}{2\pi}\left( \sqrt[3]{\frac{1}{2|x|}+\sqrt{\frac{1}{4x^2}-\frac{1}{27}}} - \sqrt[3]{\frac{1}{2|x|}-\sqrt{\frac{1}{4x^2}-\frac{1}{27}}}\right)
\end{equation}
on the punctured interval $[-\frac32\sqrt{3},\frac32\sqrt{3}]\setminus\{0\}$.
\begin{figure}[hbt!]
  \includegraphics[scale=0.55]{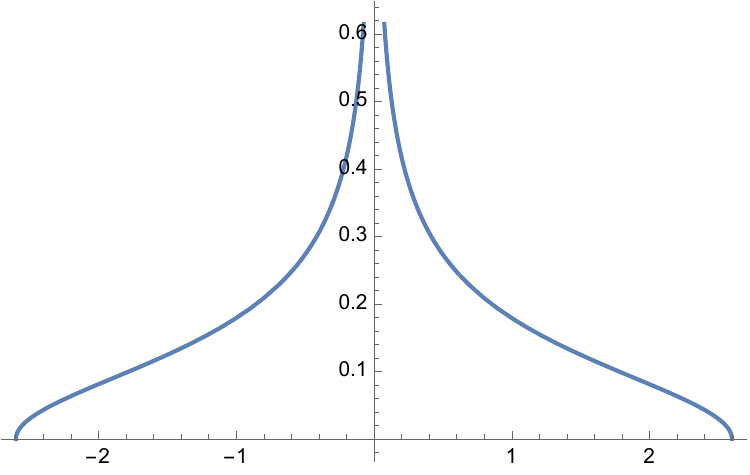}
  \caption{The density in \eqref{e.cubic.density} of the $\boxplus$-infinitely divisible `Azadi Tower' distribution with Rademacher $\zcirc$ Lévy measure.}
  \label{fig:5.14}
\end{figure}

For $|x|>\frac32\sqrt{3}$, $1/4x^2-1/27<0$ and so
\[ r(x) = \sqrt[3]{-\frac{1}{2x}+i\sqrt{\frac{1}{27}-\frac{1}{4x^2}}}. \]
An analogous calculation to \eqref{e.cubic.cancel.1}, relying on the fact that $\sqrt[3]{\alpha\beta} = \sqrt[3]{\alpha}\sqrt[3]{\beta}$ and $\sqrt[3]{\bar{\alpha}} = \overline{\sqrt[3]{\alpha}}$ for $\alpha,\beta\in\mathbb{C}_+$, shows that
\[ \frac{1}{3r(x)} = \sqrt[3]{-\frac{1}{2x}-i\sqrt{\frac{1}{27}-\frac{1}{4x^2}}} = \overline{r(x)}. \]
Hence, for $|x|>\frac32\sqrt{3}$, $G_X(x) = r(x) + 1/3r(x) = r(x) + \overline{r(x)}$ is real-valued, and therefore $\mathcal{L}(X)$ has no mass there.

Finally: A direct calculation (with the aid of Mathematica) shows that the density \eqref{e.cubic.density} integrates to $1$ on the interval $[-\frac32\sqrt{3},\frac32\sqrt{3}]$, which shows that there is no mass at $0$  and hence we conclude that \eqref{e.cubic.density} is the density of $X$, supported on the full interval $[-\frac32\sqrt{3},\frac32\sqrt{3}]$.
\end{example}

Before proceeding to the next example, it will be useful to record the formula for the complex square root, in the next (long-known) lemma.

\begin{lemma} Let $z\mapsto \sqrt{z}$ denote the square root with branch cut along the positive real line:
\[ \sqrt{re^{i\theta}} = r^{1/2} e^{i\theta/2}, \qquad r\ge 0,\;\; \theta\in[0,2\pi). \]
For any complex number $z=c+di$ with $c,d\in\mathbb{R}$,
\begin{equation} \label{eq:complex.root.formula}
\sqrt{c+di}=
{\rm sign}(d)\sqrt{\frac{\sqrt{c^2+d^2}+c}{2}}+i\,
\sqrt{\frac{\sqrt{c^2+d^2}-c}{2}}
\end{equation}
where $\mathrm{sign}(d) = d/|d|$ when $d\ne 0$, and $\mathrm{sign}(0)=1$.
\end{lemma}
Note: \eqref{eq:complex.root.formula} may disagree with formulas found in textbooks (and Wikipedia), where the $\mathrm{sign}(d)$ term appears with the imaginary part instead of the real part.  That alternate formula is correct for a different branch cut, along the {\em negative} real line, i.e.\ where the polar variable is taken with $\theta\in[-\pi,\pi)$.  Different sources call either that square root or this one the ``principal branch'' confusingly; we stick to this one, cut along the {\em positive} real line, throughout this paper.

\begin{example} Let $V_0$ have the Cauchy distribution, with density $\varrho_{V_0}(x) = \frac{1}{\pi(1+x^2)}$, cf.\ Example \ref{example.Cauchy.dist}.  Thus $G_{V_0}(z) = \frac{1}{z+i}$. The corresponding centered, variance $\sigma^2$ freely infinitely divisible distribution $\mathcal{L}(X)$ then satisfies
\[ zG_X(z) - 1=\sigma^2G_X(z)\frac{1}{1/G_X(z)+i} \]
cf.\ \eqref{eq:examples.end}.  The solution is

\ignore{
Simplifying, we see that $G_X(z)$ is a solution $w$ of the quadratic equation
\begin{equation} \label{e.Cauchy.quad}
(\sigma^2-iz)w^2+(i-z)w+1=0.
\end{equation}
The solutions $w_{\pm}(z)$ of \eqref{e.Cauchy.quad} are 
\begin{align*}
w_{\pm(z)}=\frac{z-i \pm \sqrt{(z-i)^2-4(\sigma^2-iz)}}{2(\sigma^2-iz)}
=\frac{z-i \pm \sqrt{(z+i)^2-4\sigma^2}}{2(\sigma^2-iz)}.
\end{align*}
Evaluating at $z=iy$ for $y>0$, we have
\begin{align*} iy\,w_\pm(iy) &= iy\frac{iy-i\pm \sqrt{(iy+i)^2-4\sigma^2}}{2(\sigma^2-i(iy))} \\
&= \frac{-y^2+y\pm iy\sqrt{-(y+1)^2-4\sigma^2}}{2(\sigma^2+y)} \\
&= \frac{-y^2+y\pm iy\cdot i(y+1)\sqrt{1+4\sigma^2/(y+1)^2}}{2(\sigma^2+y)} \\
&= \frac{y}{2(\sigma^2+y)}\left(-y+1\pm(-1)(y+1)\right)\sqrt{1+4\sigma^2/(y+1)^2} \\
&= \frac{y}{2(\sigma^2+y)}\left(-y+1\pm(-1)(y+1)\right)\left(1+O\left(\frac{1}{y^2}\right)\right).
\end{align*}
We see that $iy\,w_+(iy) = O(y)$, while $iy\,w_-(iy) = y/(\sigma^2+y)(1+O(1/y^2)) \to 1$ as $y\uparrow\infty$.  Hence, the desired Cauchy transform $G_X$ must coincide with $w_-$:
}

\begin{equation} G_X(z) = \frac{z-i - \sqrt{(z+i)^2-4\sigma^2}}{2(\sigma^2-iz)} \end{equation}
where the square root is the principal branch, and the negative sign is the correct root to yield $O(1/z)$ asymptotic behavior at infinity.  

To compute the density, we realize the denominator,
\begin{align} \nonumber
G_X(z) &=
\frac{z-i - \sqrt{(z+i)^2-4\sigma^2}}{2(\sigma^2-iz)} \frac{\sigma^2+i\overline{z}}{\sigma^2+i\overline{z}} \\
\label{eq:GX.rationalized}
&=\frac{(z-i - \sqrt{(z+i)^2-4\sigma^2})(\sigma^2+i\overline{z})}{2(\sigma^4+|z|^2)}.
\end{align}
Taking the imaginary part, expanding the numerator, expressing $z=x+iy$ and letting $a(z) = \mathrm{Re}\,\sqrt{(z+i)^2-4\sigma^2}$ and $b(z) = \mathrm{Im}\,\sqrt{(z+i)^2-4\sigma^2}$ we obtain
\begin{multline*}
\mathrm{Im} \left( i|z|^2-\bar{z}+\sigma^2 z-i\sigma^2 - (\sigma^2+i\bar{z})\sqrt{(z+i)^2-4\sigma^2} \right)\\
=x^2+y^2+(1+\sigma^2)y-\sigma^2 - \mathrm{Im}\left[(\sigma^2+y+ix)(a(z)+ib(z))\right]\\
=x^2+y^2+(1+\sigma^2)y-\sigma^2 - xa(z) - (\sigma^2+y)b(z)
\end{multline*}
Substituting into \eqref{eq:GX.rationalized} gives
\begin{equation} \label{e.Cauchy.G.2} 
\mathrm{Im}(G_X(z)) = \frac{x^2+y^2+(1+\sigma^2)y-\sigma^2 - xa(z) - (\sigma^2+y)b(z)}{2(\sigma^4+x^2+y^2)}. \end{equation}
To compute the functions $a(z)$ and $b(z)$, we use the formula \eqref{eq:complex.root.formula} for the real and imaginary part of the complex square root;
in this case $c+di = (z+i)^2-4\sigma^2 = x^2-(y+1)^2-4\sigma^2 + 2x(y+1)\,i$. As $y>0$, $\mathrm{sign}(d) = \mathrm{sign}(x)$.    Hence
\begin{align*}
\sqrt{2}\,a(z) &= \mathrm{sign}(x)\sqrt{\sqrt{(x^2-(y+1)^2-4\sigma^2)^2 + 4x^2(y+1)^2}+x^2-(y+1)^2-4\sigma^2} \\
\sqrt{2}\,b(z) &= \sqrt{\sqrt{(x^2-(y+1)^2-4\sigma^2)^2 + 4x^2(y+1)^2}-x^2+(y+1)^2+4\sigma^2}
\end{align*}
These expressions, and the other terms in \eqref{e.Cauchy.G.2}, all have pointwise limits as $y\downarrow 0$, and so combining and using the symmetry principle of Remark \ref{remark.symmetry} to restrict to $x>0$ and then replace $x$ by $|x|$,  we see that the density of $X$ is
\begin{equation} \label{e.Cauchy.density}
\varrho_X(x) = \frac{1}{2\pi}\frac{|x|a_0(x)+\sigma^2(b_0(x)+1)-x^2}{\sigma^4+x^2}
\end{equation}
where
\begin{align*}
\sqrt{2}\,a_0(x) &= \sqrt{\sqrt{(x^2-1-4\sigma^2)^2 + 4x^2}+x^2-1-4\sigma^2} \\
\sqrt{2}\,b_0(x) &=\sqrt{\sqrt{(x^2-1-4\sigma^2)^2 + 4x^2}-x^2+1+4\sigma^2}
\end{align*}

\begin{figure}[hbt!]
  \includegraphics[scale=0.55]{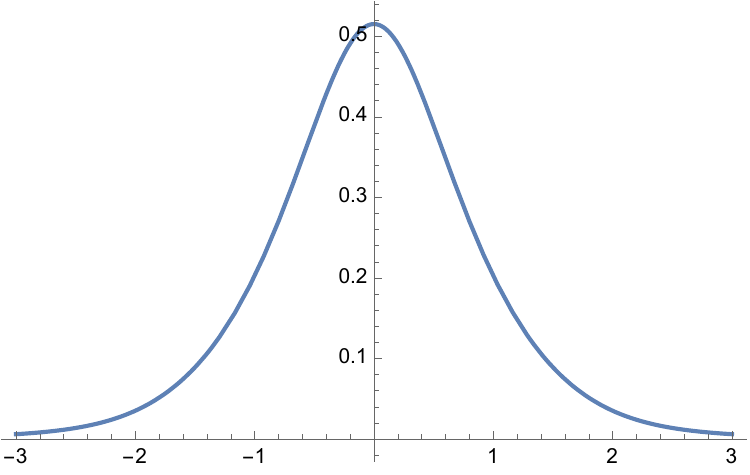}
  \caption{The density in \eqref{e.Cauchy.density} of the mean $0$, variance $\sigma^2=1$, freely infinitely divisible distribution with Cauchy $\zcirc$ Lévy measure.}
  \label{fig:5.15}
\end{figure}

Note that, like the Cauchy distribution itself, this freely infinitely divisible law has a strictly positive density on all of $\mathbb{R}$.  The law of $V_0$ (the Cauchy distribution) has no finite integer moments; nevertheless, it is the free L\'evy--Khintchine measure of the above freely infinitely divisible distribution, which has (exactly) two finite integer moments: $\varrho_X$ is continuous and bounded, and $\lim_{x\to\infty} x^4\cdot \varrho_X(x) = \sigma^2/\pi$.
\end{example}

\subsection{Using the Size Bias}

In this final subsection we produce examples of {\em positively} freely infinitely divisible distributions using Theorem \ref{main theorem 3}(\ref{thm2.c.sb}. Namely, by Proposition \ref{prop:ForAnyV.sb}, any non-negative random variable $V_+\ge 0$ is associated to a one-parameter family of positively freely infinitely divisible distributions $\mathcal{L}(X)$ as in \eqref{eq:free.Schmock.sb}, parametrized by the mean $m>0$ of $X$.  Unpacking \eqref{eq:free.Schmock.sb} using Lemma \ref{lem:X.Xneg0.same.sb}(\ref{lem.Box.G.sb} for the action of size bias on the Cauchy transform, we have
\begin{align}\label{eq:sb.identity.exYPM}
zG_{X}(z)-1 = m\,G_{V_+}(1/G_X(z)).
\end{align}
Hence, specifying $V_+$ we have an implicit equation for the Cauchy transform of $X$ which, in many cases, can be solved explicitly to determine the distribution of $X$.

\begin{example} In the simplest possible case, take $\mathcal{L}(V_+) = \delta_\alpha$ to be a point-mass at some point $\alpha\ge 0$.  Then $G_{V_+}(z) = \frac{1}{z-\alpha}$, and \eqref{eq:sb.identity.exYPM} becomes
\[ zG_X(z)-1= m\frac{1}{1/G_X(z)-\alpha} = \frac{m G_X(z)}{1-\alpha G_X(z)}. \]
Dividing through by $G_X(z)$ and rearranging, this yields
\begin{equation} \label{eq:FP.intermediate.again} z = \frac{m}{1-\alpha G_X(z)} + \frac{1}{G_X(z)}. \end{equation}
This yields a quadratic equation for $G_X$ except in the case $\alpha=0$ where it reduces to $G_X(z) = \frac{1}{z-m}$ meaning $G_X(z)=\frac{1}{z-m}$: i.e.\ if $V_+ = 0$ then $X$ is constant (equal to its mean).  Otherwise, comparing \eqref{eq:FP.intermediate.again} to \eqref{eq:FP.intermediate}, we see that $X \equaldist \mathrm{FP}(m/\alpha,\alpha)$ has a free Poisson distribution.  Indeed, this is consistent with Lemma \ref{lem:FP.L1}, taking the jump distribution $U\equaldist\delta_\alpha$, and noting that for a single positive point mass $U^s = U$.
\end{example}

For our final example, we consider ``doubling down'' on the free Poisson that appears as the law of $X$ when $V_+(X)$ is a constant, and instead consider the distribution $X$ when the law of $V_+(X)$ is a free Poisson $\mathrm{FP}(\lambda,\alpha)$.  This yields a new two-parameter family of positively freely infinitely divisible distributions with interesting behavior as the parameters vary.

\begin{example} Let $\lambda>0$ and $\alpha>0$, and take $V_+ \equaldist \mathrm{FP}(\lambda,\alpha)$ have the free Poisson distribution, whose Cauchy transform is given in \eqref{e.FP.Cauchy}:
\[  G_{V_+}(z) = \frac{z+(1-\lambda)\alpha - \sqrt{(z-(1+\lambda)\alpha)^2-4\alpha^2\lambda}}{2\alpha z}.
\]
In this case, \eqref{eq:sb.identity.exYPM} yields
\begin{align*}
\frac{zG_X(z)-1}{m} = \frac{1/G_X(z)+(1-\lambda)\alpha - \sqrt{(1/G_X(z)-(1+\lambda)\alpha)^2-4\alpha^2\lambda}}{2\alpha/G_X(z)}
\end{align*}
resulting in the quadratic equation 
$$
\alpha z\left[z-(1-\lambda)m \right]G_X^2(z)+\left[
-(2\alpha+m)z+m((1-\lambda)\alpha+m)
\right]G_X(z)+\left( \alpha + m \right)=0.
$$
Applying the quadratic formula yields
\begin{multline}\label{eq:GX.for.MPY}
G_X(z)= \\ \frac{(2\alpha+m)z- m(\alpha(1-\lambda) + m) -  m \sqrt{z^2 -2(\alpha(1+\lambda)+m)z+(\alpha(1-\lambda)+m)^2}}{2\alpha z\left(z-(1-\lambda)m \right)},
\end{multline}
where we have taken the negative root, based on the approximation to first order: for $z\to\infty$ in $\mathbb{C}_+$, 
$$
zG_X(z) = \frac{(2\alpha+m)z \pm m z }{2 \alpha z}+o(z).
$$
To recover the law of $X$ from \eqref{eq:GX.for.MPY}, we employ the Stieltjes inversion formula.  To that end (for readability) it is convenient to introduce the following notation for the constants:
\begin{equation*} 
\delta^2 = 4\alpha\lambda(\alpha+m), \qquad c = \alpha(1+\lambda)+m, \qquad r=(1-\lambda)m. 
\end{equation*}
Completing the square in the quadratic function inside the square root in \eqref{eq:GX.for.MPY}, we have
\begin{equation} \label{eq.GX.MP.2terms} G_X(z) = \tilde{G}_X(z) -\frac{m}{2\alpha}\frac{\sqrt{(z-c)^2-\delta^2}}{z(z-r)} \end{equation}
where
\[ \tilde{G}_X(z) = \frac{(2\alpha+m)z- m(\alpha(1-\lambda) + m)}{2\alpha z(z-r)} \]
is a rational expression that is continuous on $\mathbb{C}_+\sqcup(\mathbb{R}\setminus\{0,r\})$, and takes real values on $\mathbb{R}\setminus\{0,r\}$.  Ergo $\lim_{\epsilon\downarrow 0}\mathrm{Im}\tilde{G}_X(x+i\epsilon) = 0$ at all real $x\ne 0,r$, and so $\tilde{G}_X$ does not contribute to the law of $X$ at any points other than possible $\{0,r\}$.  Similarly, the second summand in \eqref{eq.GX.MP.2terms} is continuous for $z\in\mathbb{C}_+\setminus\{0,r\}$, and so
\begin{equation} \label{eq:initial.density.formula} \lim_{\epsilon\downarrow0}-\frac{1}{\pi}\mathrm{Im}G_X(x+i\epsilon) = \frac{m}{2\pi\alpha}\frac{\sqrt{\delta^2-(x-c)^2}}{x(x-r)}\mathbf{1}\{|x-c|\le\delta\}=: \rho_X(x). \end{equation}
The function $\rho_X$ is integrable and $\ge 0$.  Indeed: elementaty computations show that 
\begin{equation} \label{eq.c2-d2.etc} c^2-\delta^2 = (\alpha(1-\lambda)+m)^2, \qquad (c-r)^2-\delta^2 = (\alpha(1-\lambda)-\lambda m)^2 \end{equation}
from which it follows in particular that $c-\delta\ge \max\{0,r\}$.  Since the denominator of $\rho_X$ is positive except on the interval between $0$ and $r$, and since the numerator is non-negative everywhere, it follows that $\rho_X\ge 0$.  For integrability, we note that the function is uniformly bounded on the support interval $[c-\delta,c+\delta]$ provided that both $0$ and $r$ are {\em strictly} $< c-\delta$.  A simple calculation shows that this holds true except in two cases:
\[ \text{if }\; \lambda = 1+\frac{m}{\alpha}, \quad c-\delta = 0; \qquad \text{if }\; \frac{1}{\lambda} = 1+\frac{m}{\alpha}, \quad c-\delta = r. \]
Aside from these two values for $\lambda$, $\rho_X$ is bounded and compactly supported, hence integrable.  At these values, $\rho_X$ has $(\,\cdot\,)^{-1/2}$ behavior at the pole that has collided with the edge of its support, and thus still remains integrable.

A fairly involved (and entertaining) calculus exercise shows that
\begin{equation}\label{eq.mass.rho} \int_{\mathbb{R}} \rho_X(x)\,dx
= \begin{cases} \frac{m}{\alpha(1/\lambda-1)} & \qquad\; 0<\lambda < \frac{1}{1+m/\alpha} \\
\quad\, 1 & \frac{1}{1+m/\alpha}\le \lambda \le 1+\frac{m}{\alpha} \\
\frac{m}{\alpha(\lambda-1)} & \qquad\qquad\,  \lambda > 1+\frac{m}{\alpha}.
\end{cases}\end{equation}
 In other words: the total mass of the density $\rho_X$ is $1$ if $\lambda$ is between $1+\frac{m}{\alpha}$ and its reciprocal, and is otherwise equal to $\frac{m}{\alpha(\max\{\lambda,1/\lambda\}-1)}$:
\begin{align}\label{eq:rho.a.lam.mu}
\int_{\mathbb{R}} \rho_X(x)\,dx = 
\min\left\{\frac{m}{\alpha(\max(\lambda,1/\lambda)-1)},1 \right\}.
\end{align}

It remains to understand what happens at the poles $0$ and $r$.  To that end, we note that for any $x\in\mathbb{R}$,
\[ \lim_{z\downarrow x} \sqrt{(z-c)^2-\delta^2} = \mathrm{sign}(x-c)\sqrt{(x-c)^2-\delta^2}. \]
(This follows from \eqref{eq:complex.root.formula}.)  In particular, since $x=0<c$, using \eqref{eq.c2-d2.etc} we have
\[ \lim_{z\downarrow 0} \sqrt{(z-c)^2-\delta^2} = -\sqrt{c^2-\delta^2} = -|\alpha(1-\lambda)+m| \]
and hence, at least when $\lambda\ne 1$,
\begin{align*} \lim_{z\downarrow 0} zG_X(z) &= \frac{(2\alpha+m)\cdot 0 -m(\alpha(1-\lambda)+m)+m|\alpha(1-\lambda)+m|}{2\alpha(0-(1-\lambda)m)} \\
&= \frac{\alpha(1-\lambda)+m - |\alpha(1-\lambda)+m|}{2\alpha(1-\lambda)} \\
&= \begin{cases} 1-\frac{m}{\alpha(\lambda-1)} & \lambda > 1+\frac{m}{\alpha} \\
0 & \lambda\le 1+\frac{m}{\alpha}.
\end{cases}
\end{align*}
This shows that $X$ has a point mass at $0$ of weight $1-\frac{m}{\alpha(\lambda-1)}$ when $\lambda>1+\frac{m}{\alpha}$, but no mass at $0$ otherwise.

Similarly, since $r<c$ as well, using \eqref{eq.c2-d2.etc} we find
\[ \lim_{z\downarrow r} \sqrt{(z-c)^2-\delta^2} = \sqrt{(r-c)^2-\delta^2} = -|\alpha(1-\lambda)-\lambda m| \]
and hence, at least when $\lambda\ne 1$,
\begin{align*} \lim_{z\downarrow r} (z-r)G_X(z) &= \frac{(2\alpha+m)r-m(\alpha(1-\lambda)+m)+m|\alpha(1-\lambda)-\lambda m|}{2\alpha r} \\
&= \frac{(2\alpha+m)(1-\lambda)m - m(\alpha(1-\lambda)+m)+m|\alpha(1-\lambda)-\lambda m|}{2\alpha m(1-\lambda)} \\
&= \frac{m(\alpha(1-\lambda)-\lambda m)+m|\alpha(1-\lambda)-\lambda m|}{2\alpha m(1-\lambda)} \\
&= \begin{cases} 1-\frac{m}{\alpha(1/\lambda-1)} & \lambda < \frac{1}{1+m/\alpha} \\
0 & \lambda \ge \frac{1}{1+m/\alpha}.
\end{cases}
\end{align*}
This shows that $X$ has a point mass at $r$ of weight $1-\frac{m}{\alpha(1/\lambda-1)}$ when $\lambda<\frac{1}{1+m/\alpha}$, bot no mass at $r$ otherwise.

Combining with \eqref{eq.mass.rho}, and the formula \eqref{eq:initial.density.formula} for the density $\rho_X$, we thus have the complete description of the law of $X$:
\begin{equation} \label{eq:final.law.end}
\begin{aligned}
\rho_X(x) &= \frac{m}{2\pi \alpha}\frac{\sqrt{(4\alpha\lambda(\alpha+m)-(x-m-\alpha(1+\lambda))^2)_+}}{x(x-(1-\lambda)m)} \\
\mathcal{L}(X) &= \rho_X(x)\,dx + \left(\textstyle{1-\frac{m}{\alpha(\lambda-1)}}\right)_+ \delta_0 + \left(\textstyle{1-\frac{m}{\alpha(1/\lambda-1)}}\right)_+\delta_{(1-\lambda)m}
\end{aligned}
\end{equation}
where, as usual, $f_+ = f\mathbf{1}_{\{f\ge 0\}}$.

Notice that the two values of $\lambda$ where the nature of $\mathcal{L}(X)$ changes are precisely the values for which the density $\rho_X$ is unbounded.  These are also the precise values of $\lambda$ where a point mass collides with the support of $\rho_X$.

\end{example}

\ignore{
In the special case where $\alpha=1,\lambda=1,m=2$, we obtain
\begin{align}\label{e.MP.density}
\rho(x) = \frac{1}{\pi}\frac{\sqrt{12-(x-4)^2}}{x^2} \qmq{for $|x-4| \le \sqrt{12}$,}
\end{align}
}

\begin{figure}[hbt!]
  \includegraphics[scale=0.48]{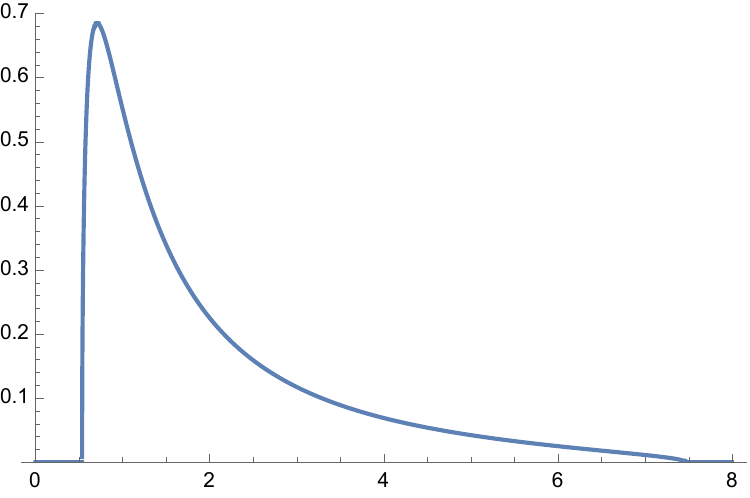}
  \caption{The law in \eqref{eq:final.law.end} with $\alpha=\lambda=1,m=2$ (in which case there are no point masses, only the density.}
  \label{fig:5.16}
\end{figure}

\begin{figure}[hbt!]
  \includegraphics[scale=0.48]{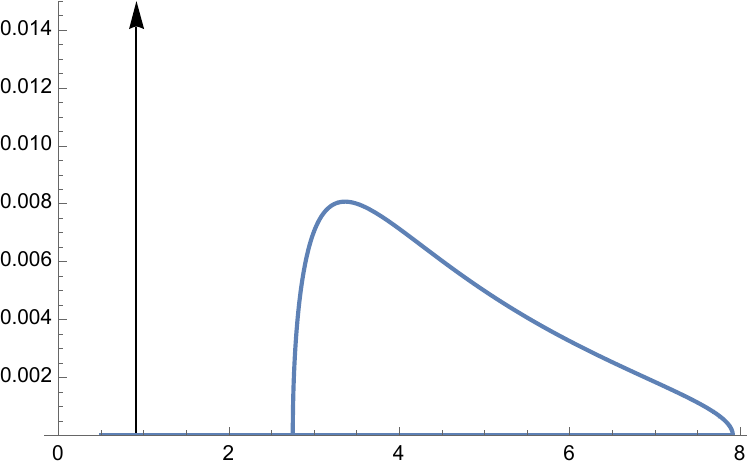}
  \includegraphics[scale=0.48]{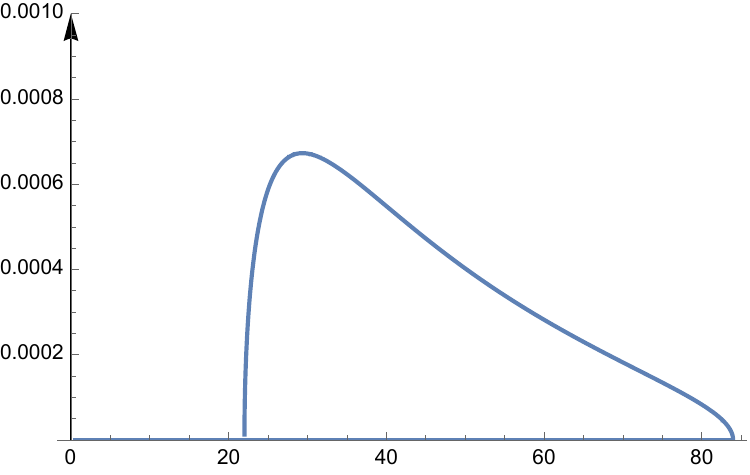}
  \caption{The law in \eqref{eq:final.law.end} with $\alpha=4$, $m=1$, and $\lambda=1/12$ on the left and $\lambda=12$ on the right.  In these complementary cases, most of the mass is concentrated at an atom with weight $\frac{43}{44}$, located at the point $\frac{11}{12}$ on the left, and at $0$ on the right.}
  \label{fig:5.17}
\end{figure}

\section*{Acknowledgments}

The authors would like to thank Serban Belinschi and Ching Wei Ho for several very useful conversations clarifying core ideas of this work.

A few comments on the name `El Gordo' (Lemma \ref{lem:bflat}), which is Spanish for `The Fat One'.  As shown in some examples, e.g.\ Example \ref{ex:XbflatY.arcsine}(\ref{ex:XbflatY.arcsine.b}, the El Gordo transform tends to spread the support out.  As we discussed at the beginning of Section \ref{sect:freezerobias.construction}, $X\mapsto X^\flat$ is a free probability analogue of $X\mapsto UX$ where $U$ is a $\mathrm{Unif}[0,1]$ random variable independent from $X$, and this operation also has the effect of spreading out $X$.  We could give other reasons why this name is sensible, but honestly we chose this name because the insight to introduce this and related transforms came to the authors over many excellent lunches at the incomparable {\em Tacos El Gordo} in Chula Vista, California.

\bibliographystyle{amsplain}
\bibliography{references}

\end{document}